\documentclass[10pt,reqno]{amsart}
\RequirePackage[OT1]{fontenc}
\RequirePackage{amsthm,amsmath}
\RequirePackage[numbers]{natbib}
\RequirePackage[colorlinks,linkcolor=blue,citecolor=blue,urlcolor=blue]{hyperref}
\usepackage{amssymb}
\usepackage{anysize}
\usepackage{hyperref}
\usepackage{extarrows}
\usepackage{multirow}
\usepackage{natbib}
\usepackage{stmaryrd}
\usepackage{color}
\usepackage{amsmath}
\usepackage{enumitem}
\usepackage{mathtools} 
\usepackage{dsfont} 
\usepackage{comment}
\usepackage{soul}

\numberwithin{equation}{section}
\theoremstyle{plain} 
\newtheorem{theorem}{Theorem}[section]
\newtheorem{lemma}[theorem]{Lemma}

\newtheorem{proposition}[theorem]{Proposition}

\theoremstyle{definition}
\newtheorem{definition}[theorem]{Definition}

\theoremstyle{remark}

\renewcommand{\Im}{\mathrm{Im}\,}
\newcommand{\mr}{\mathfrak m}

\newcommand{\E}{{\mathbb E }}

\newcommand{\R}{{\mathbb R }}
\newcommand{\N}{{\mathbb N}}

\renewcommand{\P}{{\mathbb P}}
\newcommand{\C}{{\mathbb C}}

\newcommand{\ii}{\mathrm{i}}
\newcommand{\ee}{\mathrm{e}}
\newcommand{\deq}{\mathrel{\mathop:}=}
\newcommand{\dd}{\mathrm{d}}
\newcommand{\ie}{\emph{i.e., }}
\newcommand{\eg}{\emph{e.g., }}
\newcommand{\cf}{\emph{c.f., }}

\newcommand{\wt}{\widetilde}
\newcommand{\ud}{\underline}

\newcommand{\bs}{\boldsymbol}

\def\Tr{\mathrm{Tr}}
\def\i{\text{i}}

\def\G{\mathcal{G}}

\def\Dim{\Delta \Im\,}
\def\Ai{\mathrm{Ai}}

\def\ee{\mathrm{e}}
\def\nn{\mathfrak{n}}

\def\one{\mathds{1}}

\def\<{\langle}
\def\>{\rangle}
\def\intkappa{\int_{\kappa_1}^{\kappa_2}}
\def\X{\mathcal{X}}

\renewcommand{\mathbf}[1]{\bs{#1}}
 
\marginsize{25mm}{25mm}{25mm}{26mm}  

\allowdisplaybreaks

\begin{document}
	
	\begin{minipage}{0.85\textwidth}
		\vspace{2.5cm}
	\end{minipage}
	\begin{center}
		\large\bf Convergence rate to the Tracy--Widom laws for the largest eigenvalue\\ of sample covariance matrices
		
	\end{center}

	\renewcommand{\thefootnote}{\fnsymbol{footnote}}	
	\vspace{0.5cm}
	
	\begin{center}
		\begin{minipage}{1.4\textwidth}

			\begin{minipage}{0.33\textwidth}
				\begin{center}
					Kevin Schnelli\footnotemark[1]\\
					\footnotesize 
					{KTH Royal Institute of Technology}\\
					{\it schnelli@kth.se}
				\end{center}
			\end{minipage}
			\begin{minipage}{0.33\textwidth}
				\begin{center}
					Yuanyuan Xu\footnotemark[2]\\
					\footnotesize 
					{KTH Royal Institute of Technology}\\
					{\it yuax@kth.se}
				\end{center}
			\end{minipage}
		\end{minipage}
	\end{center}
	
	\bigskip

	\footnotetext[1]{Supported by the Swedish Research Council Grant VR-2017-05195 and the Knut and Alice Wallenberg Foundation.}
	\footnotetext[2]{Supported by the Swedish Research Council Grant VR-2017-05195.}

	\renewcommand{\thefootnote}{\fnsymbol{footnote}}	
	
	\vspace{1cm}
	
	\begin{center}
		\begin{minipage}{0.83\textwidth}\footnotesize{
				{\bf Abstract.}}
			We establish a quantitative version of the Tracy--Widom law for the largest eigenvalue of high dimensional sample covariance matrices. To be precise, we show that the fluctuations of the largest eigenvalue of a sample covariance matrix $X^*X$ converge to its Tracy--Widom limit at a rate nearly $N^{-1/3}$, where $X$ is an $M \times N$ random matrix whose entries are independent real or complex random variables, assuming that both $M$ and $N$ tend to infinity at a constant rate. This result improves the previous estimate $N^{-2/9}$ obtained by Wang~\cite{wang_rate}. Our proof relies on a Green function comparison method~\cite{rigidity} using iterative cumulant expansions, the local laws for the Green function and asymptotic properties of the correlation kernel of the white Wishart ensemble.

		\end{minipage}
	\end{center}

	\vspace{5mm}
	
	{\small
		\footnotesize{\noindent\textit{Date}: August 5, 2021}\\
				\footnotesize{\noindent\textit{MSC class 2020}: 60B20, 62H10}\\
	}
	
	\vspace{2mm}

	\thispagestyle{headings}

\section{Introduction and main results}

\subsection{Introduction and previous work}

Sample covariance matrices are fundamental objects in multivariate statistics, with applications in various fields, \eg economics, population genetics and signal processing. Given $N$ independent samples $\mathbf{y}_1, \cdots, \mathbf{y}_N$ drawn from a centered random vector distribution of dimension $M$, the sample covariance matrix $\frac{1}{N} \sum_{i=1}^N\mathbf{y}_i \mathbf{y}^{*}_i$ has been well studied in the classical setting when $M$ is fixed and $N$ tends to infinity; see \cite{classical1,classical2}. However, in view of the prevalence of high dimensional data in modern applications, the population size is often large and comparable to the sample size \cite{johnstone,johnstone_application}. In this paper, we focus on the regime where $M \equiv M(N)$ depends on $N$, and both $M$ and $N$ tend to infinity at a rate $\varrho_0  \in (0,\infty)$. To be precise, denoting the aspect ratio by
\begin{align}\label{ratio}
	\varrho \equiv \varrho_N:=M/N,
\end{align}
we assume that the limit of $\rho_N$ exists as $N$ tends to infinity and
\begin{align}\label{limit_rho}
	\lim_{N \rightarrow \infty} \varrho_N= \varrho_0 \in (0,\infty).
\end{align}

We consider sample covariance matrices of the form $X^*X$, where the data matrix $X=(X_{ij})$ is an $M \times N$ random matrix whose entries are independent real or complex valued random variables satisfying
\begin{align}\label{sample_condition_1}
\E [X_{ij}]=0, \qquad \E \big[ |\sqrt{N}X_{ij}|^2\big]=1, \qquad\qquad 1\leq i \leq M,\,~1\leq j \leq N.
\end{align}
For the complex case, we moreover assume that 
\begin{align}\label{sample_condition_2}
\E[(X_{ij})^2]=0, \qquad\qquad 1\leq i \leq M,\,~1\leq j \leq N.
\end{align} 
The eigenvalues of $X^*X$ are denoted by $(\lambda_j)_{j=1}^{N}$ in non-decreasing order, and the empirical spectral distribution of $X^*X$ is defined by $\dd \mu_N:=\frac{1}{N} \sum_{j=1}^N \delta_{\lambda_j} (\dd x)$. Marchenko and Pastur \cite{MP} proved that the empirical spectral distribution converges weakly in probability (almost surely) to the Marchenko--Pastur distribution, whose density is given by 
$$\dd\mu_{\mathrm{MP}, \varrho_0}(x):=\frac{1}{2 \pi \varrho_0} \sqrt{\frac{(x- E^-_0)( E^+_0 -x)}{x^2}} \one_{[E^-_0,E^+_0]} (x) \dd x+(1-\varrho_0^{-1})_+ \delta_0(\dd x),$$
with $E^{\pm}_0:=(1 \pm \sqrt{\varrho_0})^2$.

The largest eigenvalue $\lambda_N$ of the sample covariance matrix $X^*X$ is of particular interest in principal component analysis \cite{pca}. We refer the reader to \cite{johnstone_application,johnstone+paul,paul+aue,bai+yao} for reviews of statistical applications. It is also commonly used in classical hypothesis tests, \eg Roy's largest root test \cite{roy}, or signal detection \cite{signal1,signal2}.  The asymptotics of largest eigenvalue of sample covariance matrices were well studied in \cite{German, silverstein,bai}. In particular, $\lambda_N$ converges almost surely to the right endpoint $E^+_0$ of the Marchenko--Pastur law, if the rescaled matrix entries $(\sqrt{N}X)_{ij}$ have finite fourth order moments. 

It is then natural to consider the fluctuations of the largest eigenvalue $\lambda_N$ near $E_0^{+}$. They were first studied for the special case when the rescaled matrix entries $(\sqrt{N} X)_{ij}$ are i.i.d. real or complex-valued standard Gaussian random variables. These Gaussian sample covariance matrices are called the white Wishart ensemble, and they are directly related to the classical Laguerre ensembles of random matrix theory~\cite{metha}. The asymptotics of the fluctuations of the largest eigenvalue were studied in \cite{peter, kurt_complex} for the complex white Wishart ensemble and in \cite{johnstone} for the real white Wishart ensemble. Define the centering and scaling parameters
\begin{align}\label{old_parameter}
 \mu_N:=(\sqrt{M}+\sqrt{N})^2; \quad  \sigma_N:=(\sqrt{M}+\sqrt{N}) \Big( \frac{1}{\sqrt{M}} +\frac{1}{\sqrt{N}}\Big)^{1/3}.
\end{align}
Under the condition in (\ref{limit_rho}), the centering parameter $\mu_N$ is of order $N$ and the scaling parameter $\sigma_N$ is of order $N^{1/3}$. Then the fluctuations of the largest rescaled eigenvalue of the white Wishart ensemble converge to the celebrated Tracy--Widom laws \cite{TW1, TW2}, \ie
\begin{align}\label{tw_limit}
\frac{N \lambda_N-\mu_N}{\sigma_N} \Longrightarrow \mathrm{TW}_{\beta}, \qquad \beta=1,2,
\end{align}
where we use the parameter $\beta=1,2$ to indicate the symmetry class, \ie $\beta=1$ for real-valued sample covariance matrices and $\beta=2$ for the complex case. We remark that the centering and scaling parameters in (\ref{old_parameter}) were chosen slightly differently in \cite{johnstone} for $\beta=1$, with $N$ replaced by $N-1$. This is due to the asymptotic properties of the associated Laguerre polynomials but will not effect the convergence in (\ref{tw_limit}).

The convergence in (\ref{tw_limit}) is not restricted to the white Wishart ensemble, and the limiting laws are universal for general sample covariance matrices. This universal phenomenon of the extreme eigenvalues is referred to as edge universality, which has been established for various classes of random matrices, \eg Wigner matrices \cite{rigidity,LY,So1,tao_vu2}, generalized or sparse Wigner matrices~\cite{AEKS,BEY edge universality, sparse0}. The universality for the largest eigenvalue statistics of sample covariance matrices was first considered in \cite{sasha} for $\varrho_N=1-O(N^{-2/3})$ assuming that the matrix entries $(\sqrt{N}X)_{ij}$ have symmetric distributions and sub-Gaussian tails. The condition on the aspect ratio $\varrho_N$ was removed in \cite{peche}, see also \cite{sodin} for the corresponding results for the smallest eigenvalue when the limiting aspect ratio satisfies $\varrho_0 \neq 1$. Edge universality for sample covariance matrices when $(\sqrt{N}X)_{ij}$ have vanishing third moments was proved in~\cite{wang_ke} under the condition $\varrho_0 \neq 1$. Edge universality without moment matching for sample covariance matrices was proved in~\cite{PY1} when $\varrho_0 \neq 1$. A necessary and sufficient condition on the entries' distributions for the edge universality to hold was given in~\cite{Ding+Yang}. If $\varrho_0=1$ (including the square case $M=N$), the smallest eigenvalue exhibits a different asymptotic behavior, which is referred to as the hard edge, see \cite{Edelman,peter} for the white Wishart ensemble, and the corresponding universality was studied in \cite{Ben+Peche,tao_vu}. We will focus on the largest eigenvalue of sample covariance matrices with the limiting aspect ratio $\varrho_0$ being any positive constant. 

Quantifying the convergence to the Tracy--Widom laws in~\eqref{tw_limit} is not only of interest from a mathematical point of view, but also also of fundamental importance in statistics in order to justify the use of asymptotic results in practice. For the real white Wishart ensemble Ma obtained the following quantitative estimate.

\begin{theorem}[Theorem 1 in~\cite{Ma}, quantitative Tracy--Widom law for the white Wishart ensemble]\label{convergence_gaussian}
Let $\lambda^{(W)}_N$ denote the largest eigenvalue of a real white Wishart matrix satisfying (\ref{limit_rho}). Define
\begin{align}\label{some_parameter}
	\wt \mu_N:=(\sqrt{M_-}+\sqrt{N_-})^2; \quad \wt \sigma_N:=(\sqrt{M_-}+\sqrt{N_-}) \Big( \frac{1}{\sqrt{M_-}} +\frac{1}{\sqrt{N_-}}\Big)^{1/3},
\end{align}
where $N_-=N-\frac{1}{2}$ and $M_-=M-\frac{1}{2}$. Then, for any fixed $r_0 \in \R$, there exists a constant $C \equiv C(r_0)$ such that for any $r \geq r_0$,
\begin{equation}\label{gaussian}
\Big|\P\Big( \frac{N \lambda^{(W)}_N-\wt \mu_N}{\wt \sigma_N}<r \Big)- \mathrm{TW}_{1}(r) \Big| \leq C N^{-2/3} \ee^{-r/2}.
\end{equation}
\end{theorem}
The parallel results for complex white Wishart matrices were obtained by El Karoui \cite{complex_convergence_rate} with slightly different choices of the centering and scaling parameters. 

More recently, Wang \cite{wang_rate} extended the above quantitative Tracy--Widom law to arbitrary sample covariance matrices with an $O(N^{-2/9})$ convergence rate. To be precise, let $X^*X$ be a real sample covariance matrix such that the matrix entries $(X)_{ij}$ satisfy (\ref{sample_condition_1}), have a sub-exponential decay, then for any $\omega>0$, the largest eigenvalue $\lambda_N$ of $X^*X$ satisfies
\begin{align}\label{le rate wang}
\sup_{ r \geq r_0 } \Big|\P \Big( \frac{N \lambda_N- \mu_N}{ \sigma_N} < r \Big) - \mathrm{TW}_{1}(r)  \Big| \leq N^{-2/9+\omega},
\end{align}
for sufficiently large $N$, with $\mu_N$ and $\sigma_N$ given in (\ref{old_parameter}). This is the first explicit rate of convergence to the Tracy--Widom laws for the fluctuations of the largest eigenvalue of non-Gaussian sample covariance matrices.

\subsection{Main result and strategy}
The main result of this paper is an improved bound $N^{-1/3+\omega}$ for the convergence rate of the distribution of the scaled largest eigenvalue to its Tracy--Widom limit. 

We consider a sample covariance matrix of the form $X^*X$, with $X$ an $M \times N$ random matrix satisfying the condition in (\ref{sample_condition_1}) and also (\ref{sample_condition_2}) for the complex case. We further assume that the normalized random variables $(\sqrt{N}X)_{ij}$ have uniformly bounded moments, \ie for any $p \geq 3$, there exists $C_p>0$ independent of $N$ such that for any $1\leq i \leq M,~1\leq j \leq N$,
\begin{equation}\label{moment_condition_sample}
\E \big[ |\sqrt{N} X_{ij}|^{p} \big] \leq C_{p}.
\end{equation}
We believe that the technical condition (\ref{moment_condition_sample}) can be weakened to finite moments up to a sufficient large order using the method in \cite{Ding+Yang}.

\begin{theorem}[Quantitative Tracy--Widom law for sample covariance matrices]\label{kol_dist}
Let $\lambda_N$ be the largest eigenvalue of a sample covariance matrix $X^*X$, with $X$ a real $M \times N$ matrix satisfying the moment conditions in (\ref{sample_condition_1}) and (\ref{moment_condition_sample}), as well as the aspect ration  condition~(\ref{limit_rho}). Then for any fixed $r_0 \in \R$ and any fixed small $\omega>0$, 
\begin{equation} 
\sup_{ r \geq r_0 }\Big|\P \Big( \frac{N \lambda_N- \mu_N}{ \sigma_N} < r \Big) - \mathrm{TW}_{1}(r) \Big| \leq  N^{-\frac{1}{3}+\omega},
\end{equation}
for sufficiently large $N \geq N_0(r_0,\omega)$, with $\mu_N$ and $\sigma_N$ given in (\ref{old_parameter}). The statement holds true for complex sample covariance matrices with $\beta=2$ under the additional assumption in (\ref{sample_condition_2}).
\end{theorem}

The edge universality can be studied using the dynamical approach of Erd\H{o}s, Schlein and Yau. The local relaxation time of Dyson's Brownian motion (DBM) at the spectral edges for Wigner matrices is known~\cite{LandonYau_edge,Bourgade extreme} to be of order $O(N^{-1/3})$. Bourgade's approach~\cite{Bourgade extreme} to study the eigenvalue dynamics under the DBM through a stochastic advection equation via interpolation with integrable ensembles~\cite{fixed generalized wigner,fixed dbm} was extended to sample covariance matrices by Wang~\cite{wang_rate}. Combining these local relaxation estimates with a quantitative Green function comparison theorem for very short times, Wang obtained the convergence rate $O(N^{-2/9})$ in~\eqref{le rate wang}, under the mild technical conditions $\varrho_0=\lim_{N \rightarrow \infty} \varrho_N \neq 1$ or $\rho_N\equiv 1$. This is due to lacking results for the eigenvalue rigidity at the hard edge except in the square case $\rho_N \equiv 1$ \cite{erdos_gram,tao_vu}. This restriction is relaxed in the present paper as our proof does not require strong rigidity estimates at the (hard) lower edge.

In view of the optimal local relaxation time estimates for the DBM of the singular values in \cite{wang_rate}, we suspect that the $O(N^{-1/3})$ rate for the convergence in (\ref{tw_limit}) is optimal for sample covariance matrices with general entries, though some classes of sample covariance matrices may exhibit convergence rates comparable to the white Wishart ensemble. Our speed of convergence estimate explicitly depends on the fourth order cumulants of the matrix entries.

Our proof is based on the Green function comparison method for the edge universality by Erd\H{o}s, Yau and Yin~\cite{rigidity}. To achieve the quantitative edge universality in Theorem \ref{kol_dist}, our main technical result given in Theorem~\ref{green_comparison} compares the expectation of a suitably chosen function of the Green function of the sample covariance matrix $X^*X$ with the corresponding quantity for the white Wishart ensemble. Compared to previous Green function comparison theorems, \eg \cite{Ding+Yang,PY1}, our comparison is on a much smaller spectral scale than the typical $O(N^{-2/3})$ edge scaling along with much finer error estimates. Instead of the traditional Lindeberg type swapping strategy~\cite{Ding+Yang,PY1,tao_vu}, we use a continuous flow interpolating between an arbitrary sample covariance matrix and the corresponding white Wishart ensemble, in combination with the local law \cite{Alex+Erdos+Knowles+Yau+Yin,isotropic,PY1} and cumulant expansions. The usefulness of cumulant expansions in random matrix theory was recognized in~\cite{KKP} and has widely been used since, \eg~\cite{BoutetdeMonvel,isotropic2,moment,He+Knowles,sparse,LP}.

 To achieve the quantitative Green function comparison theorem, we adopt our strategy developed recently in \cite{Schnelli+Xu} for Wigner matrices. In contrast to Wigner matrices, the matrix entries of a sample covariance matrix are no longer independent up to a symmetry. To handle this key difficulty, we follow \cite{isotropic,LS14b} to introduce a linearization of the sample covariance matrix and the corresponding Green function
\begin{align}\label{ttt}
H(z)=\begin{pmatrix}
-zI_N  &  X  \\
X   & -I_M
\end{pmatrix}, \qquad  G(z)=(H(z))^{-1},\qquad z \in \C^+.
\end{align}
We then compute the time derivative of the expectation of the normalized trace of the Green function under the interpolating flow. Via cumulant expansions, it then suffices to estimate the contributions to the Green function flow from the third and fourth order cumulants of the matrix entries; see (\ref{step0}). However, due to the finer spectral scale slightly bigger than $N^{-1}$ required in our Green function comparison, the local law for the Green function entries given in (\ref{G_0}) do not allow us to control these third and fourth order terms directly.

To tackle this difficulty, we introduce an expansion mechanism for averaged products of Green function entries, using the cumulant expansion formula in Lemma \ref{cumulant}. In view of the block structure of the Green function in (\ref{ttt}), the expansion mechanism is more intricate compared with Wigner matrices \cite{Schnelli+Xu}. Similar expansions were carried out in \cite{LS14b} using the Schur decomposition formula and expansions along matrix minors. Due to the finer spectral scale, we need to perform expansions to arbitrary order in terms of the control parameter of the local law in Theorem \ref{locallaw}. Cumulant expansions turn out to be more effective to conduct such expansions repeatedly to arbitrary order. We then observe that all the third order terms from the time derivative in (\ref{step0}) have unmatched indices; see Definition \ref{unmatch_def}.
These third order terms can be expanded in the unmatched indices to arbitrary order to show that they have negligible contributions compared to the fourth order terms.

 The remaining fourth order terms can be reduced to trace-like correlation functions of products of Green functions after expansions to arbitrary order. The resulting trace-like correlation functions can be recursively compared to the corresponding quantities for the white Wishart ensemble using once again the interpolating Green function flow. The desired estimates can be obtained using iterative expansions and the local law for Green function entries, together with the asymptotic properties of the correlation kernels of the white Wishart/Laguerre ensembles~\cite{convergence_kernel,johnstone,Ma} in the edge scaling.

There are many related random matrix models that have numerous applications in multivariate statistics. In recent years, Tracy--Widom limiting laws have been established for various models, for example, non-null Wishart matrices \cite{nonnull1,nonnull3,nonnull2}, double Wishart matrices (classical Jacobi Ensemble) \cite{johnstone}, sample covariance matrices with general populations \cite{bao,manova,isotropic,LS14b,wang_rate}, separable sample covariance matrices \cite{separable}, Gram type random matrices with general variance profiles \cite{gram}, Fisher matrices \cite{fisher,manova2}, sample canonical correlation matrices \cite{bao2,yang2}, and Kendall's tau~\cite{kendall tau}.  We believe the methods developed in this paper can be extended to generalized models to establish the corresponding quantitative edge universality.

{\it Organization of the paper:}  In Section~\ref{sec:preliminary}, we summarize some preliminaries that will be used for the proofs. In Section \ref{sec:main result}, we prove our main result Theorem~\ref{kol_dist} based on the quantitative Green function comparison theorem at the upper edge stated in Theorem \ref{green_comparison}. Before proving Theorem \ref{green_comparison}, we consider a simpler version of the theorem in Proposition~\ref{GCT_mn} to illustrate the main ideas of the proof. The proof of Proposition~\ref{GCT_mn} is summarized in Section~\ref{sec:toy} with the details carried out in Sections~\ref{sec:expand} and~\ref{sec:type0}. In Section~\ref{sec:Fneq1}, we then give the full proof of the Green function comparison in Theorem~\ref{green_comparison}.

{\it Notations:} 
Throughout the paper, many quantities depend on $N$ and for notational simplicity we often omit this dependence. We use~$c$ and~$C$ to denote strictly positive constants that are independent of~$N$, but their values may change from line to line. We use the standard Big-O and little-o notations for large~$N$. For $X,Y \in \R$, we write $X \ll Y$ if there exists a small $c>0$ such that $|X| \leq N^{-c} |Y|$ for large~$N$. Moreover, we write $X \sim Y$ if there exist constants $c, C>0$ such that $c |Y| \leq |X| \leq C |Y|$ for large $N$.

The following definition of stochastic domination from~\cite{Erdos+Knowles+Yau} is well-suited for high-probability estimates. 
\begin{definition}\label{definition of stochastic domination}
	Let $\mathcal{X}\equiv \mathcal{X}^{(N)}$ and $\mathcal{Y}\equiv \mathcal{Y}^{(N)}$ be two sequences of nonnegative random variables. We say that $\mathcal{Y}$ stochastically dominates~$\mathcal{X}$ if, for all (small) $\tau>0$ and (large)~$\Gamma>0$,
	\begin{align}\label{prec}
		\P\big(\mathcal{X}^{(N)}>N^{\tau} \mathcal{Y}^{(N)}\big)\le N^{-\Gamma},
	\end{align}
	for sufficiently large $N\ge N_0(\tau,\Gamma)$, and we write $\mathcal{X} \prec \mathcal{Y}$ or $\mathcal{X}=O_\prec(\mathcal{Y})$.
\end{definition}
We often use the notation $\prec$ also for deterministic quantities, then~\eqref{prec} holds with probability one. Useful properties of stochastic domination can be found in Lemma \ref{dominant_prop} below.

For any matrix $A \in \C^{m \times n}$, the matrix norm induced by the Euclidean vector norm is denoted by $\|A\|_{2}:=\sigma_{\max}(A)$, where $\sigma_{\max}(A)$ denotes the largest singular value of $A$. We denote the max norm of the matrix by $\|A\|_{\max}:=\max_{i,j}|A_{ij}|$.  Moreover, we denote the upper half-plane by $\C^+\deq\{z\in\C\,:\,\Im z>0\}$, and the non-negative numbers by $\R^+\deq\{x\in\R\,:\,x \geq 0\}$.  

Finally, we use double brackets to denote the index sets, \ie
$$\llbracket n_1, n_2 \rrbracket:=[n_1,n_2] \cap \mathbb{Z}, \qquad n_2,n_2 \in \R.$$

\section{Preliminaries}\label{sec:preliminary}
In this section, we collect some basic notations, results and tools required in the proofs of this paper.

\subsection{Local Marchenko--Pastur laws and eigenvalue rigidity}  

For a probability measure $\mu$ on $\R$, denote by $m_\mu$ its Stieltjes transform, i.e.
\begin{align}
 m_\mu(z)\deq\int_\R\frac{\dd\mu(x)}{x-z}\,,\qquad z\in\C^+\,.\nonumber
\end{align}
Note that $m_{\mu}\,:\C^+\rightarrow\C^+$ is analytic and can be analytically continued to the real line outside the support of $\mu$. Moreover, $m_{\mu}$ satisfies $\lim_{\eta\nearrow\infty}\ii\eta {m_{\mu}}(\ii \eta)=-1$. 

Consider an $N \times N$ sample covariance matrix $X^*X$, where $X$ is an $M \times N$ random matrix satisfying the moment conditions in (\ref{sample_condition_1}), (\ref{sample_condition_2}) and (\ref{moment_condition_sample}). The Stieltjes transform of the empirical spectral measure of $X^*X$ is given by
\begin{equation}\label{m_sample}
m(z) \equiv m_{N}(z):= \frac{1}{N} \Tr R(z),\qquad R(z):=(X^*X-zI)^{-1}, \qquad z \in \C^+,
\end{equation}
where $R$ is the resolvent of the matrix $X^*X$. We denote the resolvent of the accompanying $M \times M$ matrix $XX^*$ and its normalized trace by
 \begin{equation}\label{mathcalH}
 \mathcal{R}(z):=(XX^*-z)^{-1},\qquad \mr(z):= \frac{1}{M} \Tr  \mathcal{R} (z), \qquad z \in \C^+.
\end{equation}
It is straightforward to check that the eigenvalues of the $N \times N$ matrix $X^*X$ differ from the eigenvalues of the accompanying $M \times M$ matrix $XX^*$ by $|M-N|$ zeros. Hence we have
\begin{equation}\label{gamma_N}
m(z)=\varrho \mr(z)+\frac{\varrho-1}{z}.
\end{equation}
Without loss of generality, we assume $M \geq N$ and study the non-trivial eigenvalues of the sample covariance matrix of the form $X^*X$. Then from (\ref{ratio}) and (\ref{limit_rho}) we assume that
 \begin{equation}\label{condition}
 \varrho \equiv \varrho_N=M/N \geq 1, \qquad  \lim_{N \rightarrow \infty} \varrho_N= \varrho_0  \in [1,\infty).
\end{equation}

Following \cite{isotropic, LS14b}, we use the linearization of the $M \times N$ rectangular matrix $X$,
\begin{align}\label{hermitization}
H(z):=\begin{pmatrix}
-zI_N & X^*\\
X & -I_M
\end{pmatrix} \in \C^{(N+M) \times (N+M)}, \qquad z \in \C^+,
\end{align}
where $I_N \in \R^{N \times N}$ and $I_M \in \R^{M \times M}$ stand for the identity matrices. Though $H(z)$ is not self-adjoint, its inverse exists for any $z \in \C^+$, see (\ref{hermitization_inverse}) below. We denote its inverse matrix by $G \equiv G(z)$, and refer to $G$ as the Green function of the linearization matrix $H(z)$. Using the Schur decomposition/Feshbach formula, it is straightforward to check that
\begin{align}\label{hermitization_inverse}
G(z)=\begin{pmatrix}
-zI_N & X^*\\
X & -I_M
\end{pmatrix}^{-1}=
\begin{pmatrix}
R &  X^{*} \mathcal{R}  \\
X   {R}   & z \mathcal{R}
\end{pmatrix}, \qquad z \in \C^+,
\end{align}
with $R \equiv R(z)$ and $\mathcal{R} \equiv \mathcal{R}(z)$ given in (\ref{m_sample}) and (\ref{mathcalH}). In order to study the normalized trace of the resolvent of the matrix $X^*X$ in (\ref{m_sample}), it suffices to estimate the average of the first $N$ diagonal entries of the Green function $G(z)$ in (\ref{hermitization_inverse}).

Next, we recall some useful properties of the Green function $G$, see Lemma 4.6 in \cite{isotropic} for reference.
\begin{lemma}
\begin{enumerate}
	\item (Deterministic bound for Green function entries)
	If $z=E+\ii \eta \in \C^+$ satisfies $|z|<C$ for some constant $C>0$, then there exists a constant $C'>0$ such that
	\begin{equation}\label{deter_bound}
		\max_{1 \leq \mathfrak{i}, \mathfrak{j} \leq N+M} |G_{\mathfrak i \mathfrak j}(z)| \leq \|G(z)\|_2 \leq \frac{C'}{\eta}.
	\end{equation}
\item (Generalized Ward identities)
If $z=E+\ii \eta \in \C^+$ satisfies $c <|z|<C$ for some constants $c,C>0$, then there exists a constant $C'>0$ such that for any $1\leq b\leq N$ and $N+1 \leq \alpha \leq N+M$,
\begin{align}\label{ward}
     \sum_{\mathfrak{k}=1}^{N} |G_{b \mathfrak{k}}(z)|^2 = \frac{ \Im G_{bb}(z)}{\eta}; \qquad \qquad \sum_{\mathfrak{k}=N+1}^{N+M} |G_{b \mathfrak{k}}(z)|^2 \leq C' \|X^*X\|_2  \sum_{\mathfrak{k}=1}^{N} |G_{b \mathfrak{k}}(z)|^2 ;\nonumber\\
    \sum_{ \mathfrak{k}=N+1}^{N+M} |G_{\alpha \mathfrak{k}}(z)|^2 \leq \frac{C' \|X^*X\|_2 \Im G_{\alpha \alpha }(z)}{\eta}+2 ; \qquad  \sum_{\mathfrak{k}=1}^{N} |G_{\alpha \mathfrak{k}}(z)|^2 \leq C' \|X^*X\|_2 \sum_{ \mathfrak{k}=N+1}^{N+M} |G_{\alpha \mathfrak{k}}(z)|^2.
\end{align}
\end{enumerate}
	
\end{lemma}

Before we state the local Marchenko--Pastur law for the sample covariance matrix $X^*X$, we recall that the probability density function of the Marchenko-Pastur law (finite $N$ version) is given by
$$\dd\mu_{\mathrm{MP}, \varrho}(x):=\frac{1}{2 \pi \varrho} \sqrt{\frac{(x-E_-)(E_+-x)}{x^2}} \one_{[E_-,E_+]}\dd x,\qquad \varrho \equiv \varrho_N \geq 1,$$
 with the two spectral edge points
 \begin{align}\label{endpoint}
 E_{\pm}=(1 \pm \sqrt{\varrho})^2.
 \end{align}
The Stieltjes transform of the Marchenko--Pastur distribution, denoted by $ \wt m$ for short, is then characterized as the unique solution~to the equation
\begin{equation}\label{mplawgamma-1}
z \wt m^2(z)+(z+1-\varrho) \wt m(z)  +1=0,
\end{equation}
such that $\Im \wt m(z)>0$, $z\in \C^+$. We remark that $\wt m$ and $E_\pm$ depend on the matrix dimension $N$ via the aspect ratio $\varrho \equiv \varrho_N$ in (\ref{ratio}). The following lemma summarizes some quantitative properties of $\wt m$.

For any fixed small $\epsilon>0$ and small $0<c<1$, we introduce the spectral domain
\begin{equation}\label{ddd}
S \equiv S(\epsilon,c):=\Big\{z=E+\ii \eta:  |z| \geq c,  \kappa \leq c^{-1}, N^{-1+\epsilon} \leq \eta \leq  1 \Big\},
\end{equation}
where $\kappa \equiv \kappa(E)$ denotes the distance to the two spectral edges in (\ref{endpoint}), \ie
$$\kappa \equiv \kappa(E):=\min \{ |E_+-E|, |E_--E| \}.$$

\begin{lemma}[Theorem 3.1\cite{bao}; Lemma 3.3 \cite{Alex+Erdos+Knowles+Yau+Yin}]\label{sample_m}
	\begin{enumerate}
		\item For $z \in S$ and sufficiently large $N$ (depending on the convergence rate in (\ref{condition})), we have 
		\begin{equation}\label{1+tm}
			| \wt m(z)| \sim 1; \qquad  |1+   \wt m(z)| \sim 1.
		\end{equation}	
		\item	For $z \in S$ and sufficiently large $N$, we have 
		\begin{equation}\label{im_m_fc}
			|\Im \wt m(z)|  \sim \begin{cases}
				\sqrt{\kappa+\eta}, & \mbox{if } E \in [E_-,E_+], \\
				\frac{\eta}{\sqrt{\kappa+\eta}}, & \mbox{otherwise}.
			\end{cases}
		\end{equation}
	\end{enumerate}
\end{lemma}

Next, we introduce the deterministic control parameter
\begin{equation}\label{control_sample}
\Psi(z):=\sqrt{ \frac{\Im  \wt m(z)}{N \eta}} +\frac{1}{N \eta}\,, \qquad z=E+\ii \eta \in \C^+,
\end{equation}
and the deterministic $(N+M) \times (N+M)$ block matrix
\begin{align}\label{Pi}
\Pi \equiv \Pi(z):=\begin{pmatrix}
\wt m(z)& 0\\
0 & - (1+ \wt m(z))^{-1} 
\end{pmatrix}.
\end{align}

We are now ready to state the (anisotropic) local law for the Green function $G$ in (\ref{hermitization_inverse}) .
\begin{theorem}\label{le theorem local law}(Theorem 2.4 in \cite{Alex+Erdos+Knowles+Yau+Yin}, Theorem 3.6 in \cite{isotropic})\label{locallaw}
	Let $X$ be a random matrix satisfying (\ref{sample_condition_1}), (\ref{sample_condition_2}), (\ref{moment_condition_sample}) and (\ref{condition}). For any deterministic unit vectors $v,w \in \C^{N+M}$, we have uniformly in $z  \in S$ that
	\begin{align}\label{G_0}
	\Big| \langle v,\big( G(z)-\Pi(z)\big)  w \rangle   \Big| \prec \Psi(z).
	\end{align}
	Moreover, the normalized (partial) trace $m_N(z)$ in (\ref{m_sample}) satisfies
	\begin{align}\label{G_average}
	|m_N(z)-\wt m(z)| \prec \frac{1}{N \eta}.
	\end{align}
\end{theorem}

The local law for the Green function $G$ in Theorem~\ref{le theorem local law} implies the following eigenvalue rigidity results for the sample covariance matrix $X^*X$. Recall that the eigenvalues of $X^*X$ are denoted by $(\lambda_j)_{j=1}^N$ arranged in non-decreasing order.

For any $E_1<E_2$ ($E_1,E_2\in\R^+ \cup\{\infty\}$) denote the eigenvalue counting function by
\begin{equation}\label{number_particle}
\mathcal{N}(E_1,E_2):=\# \{j: E_1 \leq \lambda_j \leq E_2\}\,. 
\end{equation}  
We also define the classical location $\gamma_j$ of the $j$-th eigenvalue $\lambda_j$ by
\begin{equation}\label{classical}
\frac{j}{N}=\int_{0}^{\gamma_j}  \dd \mu_{\mathrm{MP},\varrho} (x).
\end{equation}

\begin{theorem}[Rigidity of eigenvalues, Theorem 2.10 \cite{Alex+Erdos+Knowles+Yau+Yin}]
Fix a small $c >0$. Then for any $c \leq E_1<E_2$, 
\begin{equation}\label{rigidity3}
 \Big| \mathcal{N}(E_1,E_2)-N \int_{E_1}^{E_2} \dd \mu_{\mathrm{MP},\varrho} (x)  \Big| \prec 1\,.
\end{equation}
In addition, for any $1\leq j \leq N$ such that $\gamma_{j} \geq c$, we have
\begin{equation}\label{rigidity4}
 |\lambda_j-\gamma_j| \prec N^{-2/3} \Big( \min \{ j, N-j+1\} \Big)^{-1/3}\,.
\end{equation}
In particular, fix some constants $C_1,C_2>0$, then for any small $\tau>0$ and large $\Gamma>0$ we have
\begin{equation}\label{rigidity1}
|\lambda_N-E_+ | \leq  C_1 N^{-2/3+\tau}, \qquad \mathcal{N}(E_+-C_2 N^{-2/3}, \infty) \leq N^{\tau},
\end{equation}
with probability bigger than $1-N^{\Gamma}$, for $N$ sufficiently large.
\end{theorem}

We end this subsection by stating some properties of the stochastic domination defined in Definition~\ref{definition of stochastic domination}, see \eg Proposition 6.5 \cite{book} for reference. We remark that given any product of matrix entries of the Green function $G(z)$ with spectral parameter $z \in S$ given in (\ref{ddd}), the deterministic upper bound condition in statement (3) below is always satisfied by (\ref{deter_bound}). This argument will be used frequently throughout the paper to estimate the expectations of products of Green function entries using stochastic domination. 
\begin{lemma}\label{dominant_prop}
	Let $X$, $X'$, $Y$, $Y'$, $Z$ be non-negative random variables. Then,
	\begin{enumerate}
		\item $X \prec Y$ and $Y \prec Z$ imply $X \prec Z$;
		\item If $X \prec Y$ and $X' \prec Y'$, then $X+X' \prec Y+Y'$ and $XX' \prec YY';$
		\item If $X \prec Y$, $\E[ Y] \geq N^{-c_1}$ and $|X| \leq N^{c_2}$ almost surely with some fixed exponents $c_1$, $c_2>0$, then we have $\E[ X] \prec \E [Y]$.
	\end{enumerate}
\end{lemma}

\subsection{Properties of the white Wishart ensemble}
Recall that a real $(\beta=1)$ or complex $(\beta=2)$ white Wishart matrix is a Gaussian sample covariance matrix $X^*X$ where the rescaled matrix entries $(\sqrt{N} X)_{ij}$,  $1\leq i\leq M, 1\leq j\leq N$ ($M \geq N$) are i.i.d real, respectively complex, standard Gaussian random variables. To distinguish the notations, we will use $W$ to indicate the Gaussian matrix and $W^*W$ to denote the white Wishart ensemble.

In the literature of random matrix theory, the rescaled white Wishart matrix $N W^*W$ is also called Laguerre orthogonal/unitary ensemble, since the joint density of the eigenvalues of $N W^*W$, denoted by $(\nu_j)_{j=1}^N$, is given by
\begin{align}\label{density}
p_{N, \alpha,\beta}(\nu_1, \ldots, \nu_N)=\frac{1}{Z_{N,\alpha,\beta}} \prod_{1 \leq j<k \leq N} |\nu_j-\nu_k|^{\beta} \prod_{j=1}^N \nu_j^{\frac{\alpha \beta}{2}} \ee^{-\frac{\beta \nu_j}{2}} \one_{0 \leq \nu_1 \leq \nu_2 \leq \cdots \leq \nu_N},
\end{align}
where $Z_{N,\alpha,\beta}$ is a normalizing constant, and $\alpha=M-N-1$, $\beta=1$ for the Laguerre orthogonal ensemble (LOE), respectively $\alpha=M-N$, $\beta=2$ for the Laguerre unitary ensemble (LUE).

Let $\{L_k^{\alpha}\}_{k=0}^{\infty}$ be the generalized Laguerre polynomials with parameter $\alpha \geq -1$ which are orthogonal on $\R^+$ with weight function $x^{\alpha} \ee^{-x}$, as defined in \cite{laguerre_poly}. Define the corresponding normalized functions by
$$\psi^{\alpha}_{k}(x) :=\sqrt{\frac{k!}{(k+\alpha)!}}x^{\frac{\alpha}{2}} \ee^{-\frac{x}{2}} L_k^{\alpha}(x), \qquad k\in \N, \quad \alpha \geq -1.$$
Then $\{\psi^{\alpha}_{k}\}_{k=0}^{\infty}$ form an orthonormal basis in $\mathrm{L}^2(\R^+)$.

The  eigenvalue process of the LUE ($\beta=2$) is well known to be a determinantal point process (see \cite{kurt_det, sasha_det}), whose $n$-point correlation function is given by
\begin{align}\label{determinant}
p^{(n)}_{N,2}(\nu_1, \cdots, \nu_n)=\det[ K_{N,2}(\nu_i,\nu_j)]_{1 \leq i,j \leq n},
\end{align}
with the correlation kernel 
\begin{align}\label{K_N2}
K_{N,2}(x,y)=\sum_{k=0}^{N-1} \psi^{\alpha}_{k}(x)\psi^{\alpha}_{k}(y).
\end{align}
From \cite{widom}, the correlation kernel in (\ref{K_N2}) has the following integral representation
\begin{align}\label{K_N2_int}
K_{N,2}(x,y)=\int_{0}^{\infty} \big(\phi_{N,1}(x+z) \phi_{N,2}(y+z) +\phi_{N,2}(x+z) \phi_{N,1}(y+z)\big) \dd z,
\end{align}
where
\begin{align}\label{phi_12}
&\phi_{N,1}(x):=(-1)^N \sqrt{\frac{\sqrt{N(N+\alpha)}}{2}} \psi^{\alpha-1}_N(x)x^{-1/2} \one_{x \geq 0},\nonumber\\
&\phi_{N,2}(x):=(-1)^{N-1} \sqrt{\frac{\sqrt{N(N+\alpha)}}{2}} \psi^{\alpha+1}_{N-1}(x)x^{-1/2} \one_{x \geq 0}.
\end{align}

In the orthogonal case, the eigenvalue process of the LOE ($\beta=1$) is a Pfaffian point process (see \cite{sasha_pf}), whose $n$-point correlation function is given by
\begin{align}\label{pfaffian}
p^{(n)}_{N,1}(\nu_1, \cdots, \nu_n)=\mathrm{pf}[ S_{N,1}(\nu_i,\nu_j)]_{1 \leq i,j \leq n},
\end{align}
with the antisymmetric $2 \times 2$ matrix kernel
\begin{align}\label{K_N1}
S_{N,1}(x,y)=\begin{pmatrix}
K_{N,1}(x,y) & -\frac{\partial }{\partial y}K_{N,1}(x,y)\\
(\varepsilon K_{N,1})(x,y)-\frac{1}{2} \mathrm{sgn}(x-y) & K_{N,1}(y,x)
\end{pmatrix}, \qquad \mbox{for even~} N,
\end{align}
where $\varepsilon$ denotes the convolution operator with kernel $\varepsilon(x,y)=\frac{1}{2} \mathrm{sgn}(x-y)$, and the correlation kernel function $K_{N,1}$ has the following integral representation \cite{widom,Ma}
\begin{align}\label{kernel_orthogonal}
K_{N,1}(x,y)
=&K_{N,2}(x,y)+ \frac{1}{2}\phi_{N,2}(x)(\mathrm{sgn} \star \phi_{N,1})(y),
\end{align}
with $K_{N,2}$ defined in (\ref{K_N2}) and $\phi_{N,1}$, $\phi_{N,2}$ given in (\ref{phi_12}). The case for odd $N$ is slightly more involved and discussed in \cite{odd}.

We recall the centering parameter $\wt \mu_N$ and scaling parameter $\wt \sigma_N$ in (\ref{some_parameter}) for the largest eigenvalue of the real white Wishart ensemble $W^*W$ and normalize the eigenvalues $(\nu_j)_{j=1}^N$ of the rescaled matrix $N W^*W$ in the edge scaling as below
\begin{align}\label{normalize}
\nu_j
=:\wt \mu_N +\wt \sigma_N  l_j.
\end{align}
The corresponding rescaled correlation kernel function for $(l_j)_{j=1}^N$ is then given by
\begin{align}\label{edge}
K^{\mathrm{edge}}_{N,\beta}(x,y):=\wt \sigma_N K_{N,\beta} \big(\wt \mu_N+\wt \sigma_N x, \wt \mu_N+ \wt \sigma_N y\big), \qquad \beta=1,2.
\end{align}
Using the asymptotic results of Laguerre polynomials, it was shown in \cite{peter, kurt_complex, johnstone} (with slightly different centering and scaling parameters) that 
\begin{align}\label{airy}
K^{\mathrm{edge}}_{N,2}(x,y) \rightarrow K_{\mathrm{Airy}}(x,y):= \frac{\Ai(x)\Ai'(y)-\Ai'(x)\Ai(y)}{x-y}\,, \qquad N \rightarrow \infty,
\end{align}
for any $x,y$ in an interval bounded from below, where $\Ai$ is the Airy function of first kind, which is the solution of $\Ai''(x)-x \Ai(x)=0\,,~ x \in \R\,,$
satisfying the boundary condition $\Ai(x) \rightarrow 0$ as $x\rightarrow \infty$. 
 As $x \rightarrow y$, the Airy kernel reduces to
\begin{align}\label{airy_xx}
K_{\mathrm{Airy}}(x,x):=(\Ai'(x))^2-\Ai''(x)\Ai(x)=(\Ai'(x))^2-x(\Ai(x))^2.
\end{align}
In the orthogonal $(\beta=1)$ case \cite{johnstone}, we have similarly
$$K^{\mathrm{edge}}_{N,1}(x,y) \rightarrow K_{\mathrm{Airy}}(x,y)+\frac{1}{2} \Ai(x) \int_{-\infty}^{y} \Ai(t) \dd t, \qquad N \rightarrow \infty,$$
for any $x,y$ in an interval bounded from below.

The following asymptotic results for the Laguerre polynomials are key ingredients in \cite{complex_convergence_rate, Ma} to prove the convergence rate for the fluctuations of the largest eigenvalue of white Wishart matrices.
\begin{lemma}
Recall the functions $\phi_{N,1}$ and $\phi_{N,2}$ defined in (\ref{phi_12}). For any fixed $L_0 \in \R$, there exists a constant $C \equiv C(L_0)>0$ such that for any $x \in [L_0,\infty]$,
\begin{align}\label{laguerre_edge_ortho}
\Big| \wt \sigma_N \phi_{N,1}(\wt \mu_N+\wt \sigma_N x)-\frac{\Ai(x)}{\sqrt{2}}\Big| \leq C N^{-1/3} \ee^{-x},\qquad \Big| \wt \sigma_N \phi_{1}( \wt \mu_N+\wt \sigma_N x)\Big| \leq C \ee^{-x};\nonumber\\
\Big| \wt \sigma_N \phi_{N,2}( \wt \mu_N+ \wt \sigma_N x)-\frac{\Ai(x)}{\sqrt{2}}\Big| \leq C N^{-2/3} \ee^{-x},\qquad \Big| \wt \sigma_N \phi_{1}(\wt \mu_N+\wt \sigma_N x)\Big| \leq C \ee^{-x},
\end{align}
for sufficiently large $N$, with $\wt \mu_N$ and $\wt \sigma_N$ given in (\ref{some_parameter}).
\end{lemma}

Combining the above results with (\ref{K_N2_int}) and (\ref{kernel_orthogonal}), it is straightforward to check the following quantitative convergence rate for the edge kernels in (\ref{edge}) to the their deterministic limits. Similar estimates were also obtained in \cite{convergence_kernel} for general Laguerre ensembles with fixed $\alpha \in \N$.
\begin{proposition}\label{kernel_diff}
For any fixed $L_0\in \R$, there exists a constant $C \equiv C(L_0)>0$ such that for any $x,y \in [L_0, \infty)$,
\begin{align}\label{difference}
\Big|K^{\mathrm{edge}}_{N,2}(x,y)-K_{\mathrm{Airy}}(x,y)\Big| \leq CN^{-1/3} \ee^{-x}\ee^{-y},
\end{align}
and
\begin{align}\label{difference_real}
\Big|K^{\mathrm{edge}}_{N,1}(x,y)-K_{\mathrm{Airy}}(x,y)-\frac{1}{2} \Ai(x) \int_{-\infty}^{y} \Ai(t) \dd t\Big| \leq CN^{-1/3}\ee^{-x} \ee^{-y},
\end{align}
for sufficiently large $N$.
\end{proposition}

Finally, we recall some basic properties of the Airy function and Airy Kernel; see \cite{AGZ} for a reference.
\begin{lemma}\label{lemma_airy_bound}
The Airy function and Airy kernel in (\ref{airy}) have the following integral representations:
\begin{align}
\Ai(x)=\frac{1}{\pi} \int_0^{\infty} \cos(ty+\frac{y^3}{3}) \dd y; \qquad K_{\mathrm{Airy}}(x,y)=\int_0^{\infty} \Ai(x+z) \Ai(y+z) \dd z.\nonumber
\end{align}
For any fixed $L_0\in\R$, there exists a constant $C \equiv C(L_0)>0$, such that for any $x,y \in [ L_0, +\infty)$,
\begin{align}
\Big|  K_{\mathrm{Airy}}(x,y)\Big| \leq C\,, \qquad \Big|  \Ai(x) \int_{-\infty}^{y} \Ai(t) \dd t\Big|  \leq C.\nonumber
\end{align}
\end{lemma}

\subsection{Cumulant expansion formula}

A key tool of this paper is the following cumulant expansion formula, see \eg Lemma 3.1 in \cite{moment} for reference and Lemma 7.1 in there for the complex version.
\begin{lemma}\label{cumulant}
	Let $h$ be a real-valued random variable with finite moments. The $p$-th cumulant of $h$ is given by
	\begin{align}\label{cumulant_k}
	c^{(p)}(h):=(-\ii)^{p} \Big(\frac{\dd}{\dd t} \log \E \ee^{\ii t h} \Big)\Big|_{t=0}.
	\end{align}
	Let $f: \R \longrightarrow \C$ be a smooth function which has bounded derivatives and denote by $f^{(p)}$ its $p$-th derivative. Then for any fixed $l \in \N$, we have
	\begin{align}\label{le first cumulant formula_real}
	\E \big[h f(h)\big]=\sum_{p+1=1}^l \frac{1}{p!} c^{(p+1)}(h)\E[ f^{(p)}(h) ]+R_{l+1}\,,
	\end{align}
	where the error term satisfies
	\begin{align}\label{cumulant_error_real}
	|R_{l+1}| \leq C_l \E \big[ |h|^{l+1}\big] \sup_{|x| \leq M} |f^{(l)}(x)| +C_l \E \big[ |h|^{l+2} 1_{|h|>M}\big] \sup_{x \in \R} |f^{(l)}(x)|,
	\end{align}
	and $M>0$ is an arbitrary fixed cutoff.
\end{lemma}

\section{Proof of Theorem \ref{kol_dist}}\label{sec:main result}

Before we begin with the proof of Theorem \ref{kol_dist}, we first establish the link between the distribution of the rescaled largest eigenvalue of $X^*X$ and the normalized trace of the resolvent in (\ref{m_sample}), following the approach in \cite{rigidity} to prove the edge universality for Wigner matrices.

Fix a small $\epsilon>0$ and introduce a truncation energy for the largest eigenvalue as (see (\ref{rigidity1}))
\begin{equation}\label{E_L}
E_L:=E_++4N^{-2/3+\epsilon},
\end{equation}
with the upper edge $E_+$ given in (\ref{endpoint}). For any $E \leq E_L$, we define the indicator function
\begin{equation}\label{le indicator chi}
\chi_{E}:=\one_{[E,E_L]}\,.
\end{equation}
The eigenvalue counting function for $X^*X$ defined in (\ref{number_particle}) is then written as $\mathcal{N}(E,E_L)=\Tr \chi_E(X^*X)$. 

For $\eta>0$, we define the mollifier $\theta_{\eta}$ by setting
\begin{equation}\label{le mollifier}
\theta_{\eta}(x):=\frac{\eta}{\pi(x^2+\eta^2)}=\frac{1}{\pi} \Im \frac{1}{x-\i \eta}.
\end{equation}
We can then relate $\Tr \chi_{E} \star \theta_{\eta}(X^*X)$ to the normalized trace of the resolvent of $X^*X$, $m_N$ in (\ref{m_sample}), as 
\begin{equation}\label{approx}
	\Tr \chi_{E} \star \theta_{\eta}(X^*X)=\frac{N}{\pi} \int \chi_E(y) \Im m_N(y+\i \eta) \dd y=\frac{N}{\pi} \int_{E}^{E_L} \Im m_N(y+\i \eta) \dd y\,.
\end{equation}

The following two lemmas assure that $\Tr \chi_E(X^*X)$ can be sufficiently well approximated by $\Tr \chi_E\star \theta_\eta(X^*X)$ for $\eta \ll N^{-2/3}$, and hence can be linked to the normalized trace $m_N$, in view of (\ref{approx}). These arguments were used first in \cite{rigidity} to prove the edge universality of Wigner matrices, where $\eta$ is chosen slightly smaller than the typical edge
eigenvalue spacing $N^{-2/3}$. In order to obtain a quantitative convergence rate, we aim to choose $\eta$ here much smaller with $\eta \gg N^{-1}$. Similar arguments were used in \cite{Bourgade extreme,wang_rate}. The proofs of the following two lemmas can be found in the Appendix of \cite{Schnelli+Xu}. 

\begin{lemma}\label{lemma2}
Let $E, l_1$ and $\eta$ be scale parameters satisfying $N^{-1}\ll\eta \ll l_1 \ll E_L-E  \leq C N^{-2/3+\epsilon}$ for some constant $C>0$. Then, for any $\Gamma>0$, 
\begin{equation}\label{approx1}
\Big|\Tr \chi_E(X^*X)- \Tr \chi_{E} \star \theta_{\eta}(X^*X)\Big| \leq C'\Big( \mathcal{N}(E-l_1,E+l_1)+\frac{\eta}{l_1} N^{2 \epsilon}\Big),
\end{equation}
holds with probability bigger than $1-N^{-\Gamma}$, for $N$ sufficiently large.
\end{lemma}

 Let $F\,:\,\R\longrightarrow\R$ be a smooth cut-off function such that
\begin{equation}\label{q_function}
F(x)=1, \quad \mbox{if} \quad |x| \leq 1/9; \qquad F(x)=0, \quad \mbox{if} \quad |x| \geq 2/9,
\end{equation}
and we assume that $F(x)$ is non-increasing for $x \geq 0$. Then one obtains the following estimates from Lemma \ref{lemma2} and the eigenvalue rigidity in (\ref{rigidity1}).
\begin{lemma}\label{lemma1}
Fix a small $\epsilon>0$. Set  $l_1=N^{3\epsilon} \eta$ and $l=N^{3\epsilon}l_1$ such that $N^{-1} \ll \eta \ll l_1 \ll l \ll E_L-E  \leq C N^{-2/3+\epsilon}$. Then for any $\Gamma>0$, we have
\begin{equation*}
 \Tr \chi_{E+l} \star \theta_{\eta}(X^*X) -N^{-\epsilon} \leq \mathcal{N}(E, \infty) \leq  \Tr \chi_{E-l} \star \theta_{\eta}(H) +N^{-\epsilon},
\end{equation*}
with probability bigger than $1-N^{-\Gamma}$, for $N$ sufficiently large. Furthermore, we have
\begin{equation}\label{approx2}
\E \Big[ F\Big( \Tr \chi_{E-l} \star \theta_{\eta}(X^*X) \Big)\Big]-N^{-\Gamma} \leq \P \Big( \mathcal{N}(E, \infty) =0\Big) \leq   \E \Big[F\Big( \Tr \chi_{E+l} \star \theta_{\eta}(X^*X) \Big)\Big]+N^{-\Gamma},
\end{equation}
where $F(x)$ is the cut-off function given in (\ref{q_function}).
\end{lemma}
Hence, recalling~\eqref{approx}, we have established in (\ref{approx2}) the desired link between the distribution function of the rescaled largest eigenvalue of $X^*X$ and the normalized trace of the resolvent of $X^*X$ using a cleverly chosen observable from \cite{rigidity}. Theorem \ref{kol_dist} hence follows from the Green function comparison, Theorem \ref{green_comparison} below, where we compare this observable for any sample covariance matrix $X^*X$ with the corresponding white Wishart matrix $W^*W$. We use $\P^{W}$ and $\E^{W}$ to denote the probability and expectation with respect to the Gaussian matrix $W$.

\begin{theorem}[Green function comparison theorem at the upper edge $E_+$]\label{green_comparison}
Consider a random matrix $X$ satisfying the moment conditions in (\ref{sample_condition_1}), (\ref{sample_condition_2}), (\ref{moment_condition_sample}) and the aspect ration condition in (\ref{condition}).  Let $F$ be a smooth function with uniformly bounded derivatives. For any fixed small $\epsilon>0$ and fixed constants $C_1,C_2>0$, choose $\eta$, $\kappa_1$ and $\kappa_2$ such that $N^{-1+\epsilon} \leq \eta \leq N^{-2/3-\epsilon}$ and $-C_1N^{-2/3} \leq \kappa_1< \kappa_2 \leq C_2 N^{-2/3+\epsilon}$.  Then for any small $\tau>0$, 
\begin{align}\label{green_difference}
\Big| \big(\E-\E^{W}\big)  \Big[ F\Big( N \int_{\kappa_1}^{\kappa_2} \Im m_N( E_++x+\ii \eta) \dd x\Big) \Big]\Big| \leq  N^{-\frac{1}{3}+\tau},
\end{align}
for sufficiently large $N \geq N_0(C_1,C_2,\epsilon, \tau)$. The results hold true for both the real and complex case.
\end{theorem}

Admitting Theorem \ref{green_comparison}, we are ready to prove the quantitative Tracy--Widom laws in Theorem \ref{kol_dist}. Similar arguments for Wigner matrices can be found in \cite{Bourgade extreme, Schnelli+Xu}. We will only consider the real sample covariance matrices ($\beta=1$), the complex case $(\beta=2)$ can be handled similarly.

\begin{proof}[Proof of Theorem \ref{kol_dist}]

Recall the aspect ratio in (\ref{condition}) and the centering and scaling parameters $\mu_N$, $\sigma_N$ in (\ref{old_parameter}). Set then
\begin{align}\label{some_new_parameter}
E_+=(1+\sqrt{\varrho})^2 \sim 1, \qquad \gamma_+=\frac{\varrho^{1/6}}{(1+\sqrt{\varrho})^{4/3}} \sim 1.
\end{align}
In order to study the distribution of the centered and scaled largest eigenvalue $\sigma^{-1}_N(N \lambda_N- \mu_N)$, it is equivalent to study the distribution of $ \gamma_+ N^{2/3}\big(\lambda_N-E_+)$.

Using the rigidity of the eigenvalues in (\ref{rigidity1}),  one easily verifies that, for any fixed small $\epsilon>0$ and $\Gamma>2/3$,
\begin{equation}\label{upper}
\sup_{r > \gamma_+ N^{\epsilon}} \Big|\P\Big(\gamma_+ N^{2/3} (\lambda_N- E_+ )<r \Big)-\P^{W}\Big(\gamma_+N^{2/3} (\lambda_N-E_+ )<r\Big) \Big| \leq  N^{-\Gamma},
\end{equation}
for sufficiently large $N$. Hence it suffices to focus on the regime $r_0 \leq r \leq \gamma_+ N^{\epsilon}$. 

For any $r_0 \leq r \leq \gamma_+ N^{\epsilon}$, let, as in (\ref{E_L}),
$$E=E_++\gamma_+^{-1}N^{-2/3} r, \qquad \mbox{and } \quad E_{L}=E_++4N^{-2/3+\epsilon}.$$ 
Set $\eta=N^{-1+\epsilon}$ and $l=N^{-1+ 7\epsilon}$ as in Lemma \ref{lemma1}, where we choose $0<\epsilon<\frac{1}{21}$ such that $l \ll N^{-2/3}$. From (\ref{approx}) and (\ref{approx2}), we can relate the distribution of the rescaled largest eigenvalue $\gamma_+ N^{2/3}\big(\lambda_N-E_+)$ to the normalized trace of the Green function as follows,
\begin{align}\label{im_identity}
\E \Big[ F\Big( N \int_{\gamma_+^{-1} N^{-2/3}r-l}^{4N^{-2/3+\epsilon}}  \Im & m_N(E_++x+\ii \eta) \dd x\Big) \Big]-N^{-\Gamma} \leq \P\Big( \gamma_+ N^{2/3} (\lambda_N- E_+ )<r \Big)\nonumber\\
 \leq    & \E \Big[ F\Big( N \int_{\gamma_+^{-1} N^{-2/3}r+ l}^{4N^{-2/3+\epsilon}} \Im m_N(E_++x+\ii \eta) \dd x\Big)\Big]+N^{-\Gamma}.
\end{align}
Shifting the value of $r$ in the second inequality of (\ref{im_identity}) and combining with the first one, we obtain
\begin{align}
 \P\Big( \gamma_+ N^{2/3} (\lambda_N- E_+ )<r -2\gamma_+ N^{2/3} l\Big) -N^{-\Gamma} \leq &\E \Big[ F\Big( N \int_{\gamma_+^{-1} N^{-2/3}r-l}^{4N^{-2/3+\epsilon}} \Im m_N(E_++x+\ii \eta) \dd x\Big) \Big]\nonumber\\
  \leq &\P\Big( \gamma_+ N^{2/3} (\lambda_N- E_+ )<r \Big)+N^{-\Gamma}.
\end{align}

Similar bounds can be obtained if we replace $-l$ in the lower integral domain with $+l$. We remark that the above inequalities hold true for any sample covariance matrices, including the white Wishart ensemble. Using Theorem \ref{convergence_gaussian} for the real white Wishart matrices where the centering and scaling parameters satisfy $\frac{\wt \mu_N}{N}=E_++O(N^{-1})$ and $\frac{N}{\wt \sigma_N}=\gamma_+ N^{2/3}+N^{-1/3}$, and that the Tracy--Widom laws have smooth and uniformly bounded density, we find
\begin{align}\label{rhs_gue}
\sup_{r_0 \leq r \leq \gamma_+ N^{\epsilon}}\Big| \E^{W} \Big[ F\Big( N \int_{ \gamma_+^{-1}N^{-2/3}r \pm l}^{4N^{-2/3+\epsilon}} \Im m_N(E_++x+\ii \eta) \dd x\Big) \Big] -\mathrm{TW}_1(r) \Big| =O(N^{-1/3+7\epsilon}).
\end{align}
Choosing $0<\tau<\epsilon$ in the Green function comparison (\ref{green_difference}) in Theorem \ref{green_comparison}, we have
\begin{align}\label{rhs_wigner}
\sup_{r_0 \leq r\leq \gamma_+ N^{\epsilon}}\Big| \Big(\E- \E^{W} \Big)\Big[ F\Big( N \int_{ \gamma_+^{-1} N^{-2/3}r \pm l}^{4N^{-2/3+\epsilon}} \Im m_N(E_++x+\ii \eta) \dd x\Big) \Big] \Big| \leq  N^{-1/3+\epsilon},
\end{align}
for sufficiently large $N$. Combining (\ref{rhs_gue}) and (\ref{rhs_wigner}) with (\ref{im_identity}), we choose $0<\epsilon<\min\{\frac{\omega}{7},\frac{1}{21}\}$ in the setting of Theorem \ref{kol_dist} and obtain
\begin{align}
\sup_{r_0 \leq r \leq \gamma_+ N^{\epsilon}}\Big|\P\Big( \gamma_+ N^{2/3} (\lambda_N- E_+ )<r \Big)-\mathrm{TW}_1(r) \Big| \leq N^{-1/3+\omega},
\end{align}
for sufficiently large $N$. Together with (\ref{upper}), we have hence completed the proof of Theorem \ref{kol_dist}.
\end{proof}

\section{A special case $F(x)=x$: estimates on $\E[\Im m_N]$}\label{sec:toy}
In this section, we prove Theorem \ref{green_comparison} for the special choice $F(x)=x$. It then suffices to compare the expectations of the normalized trace of the resolvent in (\ref{m_sample}) between any sample covariance matrix $X^*X$ and the corresponding white Wishart matrix $W^*W$. 

\begin{proposition}\label{GCT_mn}
Consider a random matrix $X$ satisfying (\ref{sample_condition_1}), (\ref{sample_condition_2}), (\ref{moment_condition_sample}) and (\ref{condition}), and the corresponding Gaussian matrix $W$ which is independent of $X$. Consider the interpolating matrix flow $X(t)$ in~(\ref{dyson_flow}) below and define the normalized trace of the resolvent of $X(t)^*X(t)$, $m_N(t,z)$, as in (\ref{m_sample_t}).  

For any fixed small $\epsilon>0$ and fixed $C_1,C_2>0$, define the following domain of the spectral parameter~$z$ near the upper edge, 
\begin{align}\label{S_edge}
S_{\mathrm{edge}} \equiv& S_{\mathrm{edge}}(\epsilon,C_1,C_2)\nonumber\\
:=&\{ z=E+\ii \eta \in S: -C_1 N^{-2/3} \leq E-E_+ \leq C_2  N^{-2/3+\epsilon}, N^{-1+\epsilon} \leq \eta \leq N^{-2/3-\epsilon}\}\,,
\end{align}
with $S$ given in (\ref{ddd}) and $E_+$ defined in (\ref{endpoint}). Then for any $\tau>0$, we have 
\begin{align}\label{compare_mn}
\Big| \E[ m_N(t,z)]- \E^{\mathrm{W}}[  m_N(z)] \Big| \leq N^{-1/3-\epsilon+\tau},
\end{align}
uniformly in $z \in S_{\mathrm{edge}}$ and $t \geq 0$, for sufficiently large $N \geq N_0(C_1,C_2, \epsilon, \tau)$. Furthermore, there exists a constant $C>0$ depending on $C_1$ and $C_2$ such that
\begin{align}\label{img_sample}
\E[\Im m_N(t, z)] \leq C N^{-1/3},
\end{align}
uniformly in $z \in S_{\mathrm{edge}}$ and $t \geq 0$, for sufficiently large $N$.
\end{proposition}

In the rest of this section we prove Proposition~\ref{GCT_mn}; its proof is split into several parts organized in four subsections. Key results are Proposition \ref{unmatch_lemma}, Proposition \ref{lemma_expand_type} and Proposition \ref{lemma_trace_sample}, whose proofs are presented in Section \ref{sec:expand} and Section \ref{sec:type0} respectively.  We will consider the real sample covariance matrices, the complex case can be proved analogously. 

\subsection{Interpolation between $X^*X$ and $W^*W$}\label{sec:interpolate}

Consider the interpolating matrix flow 
\begin{equation}\label{dyson_flow}
X(t):=\mathrm{e}^{-\frac{t}{2}}X +\sqrt{1-\mathrm{e}^{-t}} W \in \R^{M \times N}\,, \qquad t \in \R^+,
\end{equation}
where $X$ is a real-valued matrix satisfying (\ref{sample_condition_1}), (\ref{moment_condition_sample}) and (\ref{condition}), and $W$ is the corresponding Gaussian matrix chosen independently from $X$.

For any $t \in \R^+$, $z \in \C^{+}$, we define the resolvent of the time-dependent sample covariance matrix $X^*(t)X(t) \in \R^{N \times N}$ and the resolvent of the accompanying matrix $X(t)X^*(t) \in \R^{M \times M}$ as
\begin{align}\label{le time dependent G}
R(t,z):=&\big( X^*(t) X(t)-zI_N \big)^{-1}; \qquad  \mathcal{R}(t,z):=\big( X(t) X^*(t)-zI_M \big)^{-1} , \qquad z \in \C^{+}.
\end{align}
Similarly as in (\ref{hermitization}) and (\ref{hermitization_inverse}), we linearize the rectangular matrix $X(t)$ using the following block matrix
\begin{align}\label{block}
H(t,z):=\begin{pmatrix}
-zI_N  &  X^*(t)  \\
X(t)   & -I_M
\end{pmatrix}\in \C^{(N+M) \times (N+M)}, \qquad t \in \R^+,~z \in \C^{+}.
\end{align}
To distinguish the indices with respect to the block structure, we use Latin letters for indices taking values in $\llbracket 1, N \rrbracket$, Greek letters for indices taking values in $\llbracket N+1, N+M \rrbracket$. We also use calligraphic letters, \eg $\mathfrak{i}, \mathfrak{j}$, to denote the indices ranging from $1$ to $N+M$. We denote the matrix entries of $H(t,z)$ by $h_{\mathfrak{ij}} \equiv h_{\mathfrak{ij}}(t,z)$.

Define the Green function of the linearization matrix $H(t,z)$ by
\begin{align}\label{blockG}
 G(t,z):=H(t,z)^{-1}
=\begin{pmatrix}
R(t,z) &  X^*(t)  \mathcal{R}(t,z) \\
 X(t)  R(t,z)  & z \mathcal{R}(t,z)
\end{pmatrix}, \qquad t \in \R^+, z \in \C^+.
\end{align}

As in (\ref{m_sample}), the normalized trace of the resolvent of $X^*(t)X(t)$ is then given by
\begin{align}\label{m_sample_t}
m(t,z):=\frac{1}{N} \Tr {R}(t,z)=\frac{1}{N}\sum_{v=1}^{N} G_{vv}(t,z); \qquad t \in \R^+,~z \in \C^+.
\end{align}
For notational simplicity, we introduce the partial normalized traces
\begin{align}\label{ggu}
\ud{G}(t,z):=\frac{1}{N}\sum_{v=1}^{N} G_{vv}(t,z)=m(t,z); \quad \ud{\G}(t,z):=\frac{1}{M}\sum_{\nu =N+1}^{N+M} G_{\nu \nu}(t,z), \qquad t \in \R^+,~z \in \C^+.
\end{align}
In the following, we often ignore the parameters and write for short
$$H\equiv H(t,z), \quad G \equiv G(t,z), \quad m \equiv m(t,z), \quad \ud{G} \equiv \ud{G}(t,z), \quad \ud{\G}\equiv \ud{\G}(t,z), \qquad t \in \R^+,~z \in \C^+.$$ 

We remark that, by a simple continuity argument in time parameter, the local law in Theorem \ref{locallaw} still holds for the time dependent Green function $G(t,z)$ for any $t \in \R^+$ and $z \in S$ given in (\ref{ddd}), \ie
\begin{align}\label{G}
\max_{\substack{ 1\leq v \leq N \\N+1 \leq \nu \leq N+M\\1\leq  \mathfrak{i} \neq \mathfrak{j} \leq N+M}}\Big\{ \big|\ud{G}-\wt m\big|\,,\big|\ud{\G}+\frac{1}{1+\wt m} \big|\,,  \big|G_{vv}-\wt m\big|\,, \big|G_{\nu \nu}+\frac{1}{1+\wt m} \big|\,, \big|G_{\mathfrak{ij}}\big| \Big\} \prec \Psi,
\end{align}
where $\wt m$ is the deterministic function defined in (\ref{mplawgamma-1}), and the control parameter $\Psi$ is given in (\ref{control_sample}).

Differentiating $\E[m_N(t,z)]$ with respect to $t$ and using the block structure in (\ref{block}) and (\ref{blockG}), we have
\begin{align}\label{step00}
\frac{\dd }{\dd t}\E [  m_N(t,z)]=&\E \Big[ \frac{1}{N} \sum_{v=1}^N \frac{\dd }{\dd t}G_{vv}(t,z)\Big]=\E \Big[ \frac{1}{N} \sum_{v,b=1}^N \sum_{\alpha=N+1}^{N+M} -2 \dot{h}_{ \alpha b}(t) G_{v b} G_{\alpha v}  \Big].
\end{align}
It is not hard to check from (\ref{dyson_flow}) that
\begin{align}\label{diff_entry}
	(\dot{h}_{\alpha b}(t))_{\alpha b}=-\frac{\ee^{-t/2}}{2} X+\frac{\ee^{-t}}{2\sqrt{1-\ee^{-t}}} W, \qquad b \in \llbracket 1, N \rrbracket, \quad \alpha \in \llbracket N+1, N+M \rrbracket.
\end{align}
Note that the second order cumulants of matrix entries of $X$ and $W$ are identical to $N^{-1}$, and all the third and higher order cumulants of the Gaussian entries of $W$ are vanishing. 

Applying the cumulant expansion formulas in Lemma \ref{cumulant} with respect to $\dot{h}_{\alpha b}(t)$ in (\ref{diff_entry}) on the right side of (\ref{step00}) and that the matrix entries of $X$ and $W$ are all independent random variables, we observe that the second order cumulant terms are canceled precisely and we stop the expansions at the fourth order, \ie
\begin{align}\label{step0}
\frac{\dd }{\dd t}\E [ m_N(t,z)]=&\frac{1}{N} \sum_{v,b=1}^N \sum_{\alpha=N+1}^{N+M}  \sum_{p+1\geq 3}^{4} \frac{1}{p!} \frac{s^{(p+1)}_{\alpha b}(t)}{N^{\frac{p+1}{2}}}  \E \Big[ \frac{ \partial^{p}  G_{v b} G_{\alpha v}}{\partial h^p_{ \alpha b} } \Big]+O_{\prec}\big(\frac{1}{\sqrt{N}}\big)\nonumber\\
=&-\frac{1}{2N} \sum_{v,b=1}^N \sum_{\alpha=N+1}^{N+M} \sum_{p+1\geq 3}^{4} \frac{1}{p!} \frac{s^{(p+1)}_{ \alpha b}(t)}{N^{\frac{p+1}{2}}} \E \Big[ \frac{ \partial^{p+1}  G_{v v} }{\partial h^{p+1}_{\alpha b} } \Big]+O_{\prec}\big(\frac{1}{\sqrt{N}}\big),
\end{align}
where $s^{(p+1)}_{\alpha b}(t)$ is the $(p+1)$-th cumulant of the normalized matrix entries $(\sqrt{N} h_{\alpha b}(t))$ given by 
\begin{align}\label{t_cumu}
s^{(p+1)}_{\alpha b} \equiv s^{(p+1)}_{\alpha b}(t):=s^{(p+1)}_{\alpha b}(0)\ee^{-\frac{(p+1)t}{2}}, \qquad p+1 \geq 3.
\end{align}
In (\ref{step0}), we truncate the cumulant expansions at the fourth order and the error given in (\ref{cumulant_error_real}) is estimated using the local law in (\ref{G}), properties of stochastic domination in Lemma \ref{dominant_prop}, and that the normalized matrix entries $(\sqrt{N} h_{\alpha b}(t))$ have finite moments from the condition in (\ref{moment_condition_sample}).

To compute the partial derivatives on the right side of (\ref{step0}), we are using the differential rule for the Green function entries
\begin{equation}\label{dH}
\frac{\partial G_{\mathfrak{i}\mathfrak{j}}}{ \partial h_{ \alpha b}}=-G_{\mathfrak{i} \alpha} G_{b \mathfrak{j}}-G_{\mathfrak{i} b} G_{\alpha \mathfrak{j}},\qquad \mathfrak{i}, \mathfrak{j} \in \llbracket 1,  N+M \rrbracket, \quad b \in \llbracket 1,  N \rrbracket, \quad \alpha \in \llbracket N+1,  N+M \rrbracket.
\end{equation}
Then the resulting terms on the right side of (\ref{step0}) can be written as a linear combination of averaged products of the Green function entries. We give two examples of the averaged product of Green function entries below for $p+1=3$ and $p+1=4$ respectively,
\begin{align}\label{example_simple}
	 \frac{1}{N^\frac{3}{2}} \sum_{v,b=1}^N \sum_{\alpha=N+1}^{N+M}  s^{(3)}_{ \alpha b}(t) \E \Big[ G_{v \alpha} G_{b v} G_{\alpha \alpha} G_{bb} \Big], \quad \frac{1}{N^2} \sum_{v,b=1}^N \sum_{\alpha=N+1}^{N+M} s^{(4)}_{ \alpha b}(t) \E \Big[ G_{v \alpha} G_{\alpha v} G_{\alpha \alpha} (G_{bb})^2 \Big].
\end{align}

In general, we introduce an abstract form of averaged products of Green function entries as follows. For fixed integers $m_1, m_2$, we use the Latin letters in $\mathcal{I}_N:=\{v_j\}_{j=1}^{m_1}$ (may include the indices $v,b$ in (\ref{step0})) to denote the summation indices taking values in $\llbracket 1, N \rrbracket$, and use the Greek letters in $\mathcal{I}_M:=\{{\nu}_j\}_{j=1}^{m_2}$ (may include $\alpha$) to denote the indices taking values in $\llbracket N+1, N+M \rrbracket$. We also set $\mathcal{I}:=\mathcal{I}_N \cup \mathcal{I}_M:=\{\mathfrak{v}_j\}_{j=1}^{m}$, $m=m_1+m_2$, where each element in $\mathcal{I}$, denoted by $\mathfrak{v}_j$, is from either $\mathcal{I}_N$ or $\mathcal{I}_M$. We use $\prod^{n^o}_{i=1} G_{x_i y_i}$ with~$n^o \in \N$ to denote a general product of the matrix entries of the Green function $G$, where each row index~$x_i$ and column index $y_i$ of the Green function entries represents an element in $\mathcal{I}$. For later purposes of expansions explained in the next section, we also include the centered diagonal Green function entries $\ud G-\wt m$ and $\ud \G+\frac{1}{1+\wt m}$ in the product. We will use $n^{g},n^{\mathfrak{g}} \in \N$ to denote the powers of these two centered diagonal Green function factors. Then we write the abstract form as
\begin{align}\label{form}
&\frac{1}{N^{m_1+m_2}}  \sum_{v_1,\cdots v_{m_1}=1}^N \sum_{\nu_1,\cdots \nu_{m_2}=N+1}^{N+M}  c_{v_1, \ldots, v_{m_1},\nu_1, \ldots, \nu_{m_2}}(t,z) \Big(  \prod^{n^o}_{i=1} G_{x_{i} y_{i}} \Big) \Big( \ud{G} -\wt m\Big)^{n^{g}}  \Big( \ud{\G} +\frac{1}{1+\wt m}\Big)^{n^{\mathfrak{g}} }\nonumber\\
=:&\frac{1}{N^{\#\mathcal{I}} }\sum_{\mathcal{I}} c_{\mathcal{I}} (t,z) \Big(\prod^{n^o}_{i=1} G_{x_i y_i}\Big) \Big(\ud{G}-\wt m\Big)^{n^g} \Big( \ud{\G}+\frac{1}{1+\wt m}\Big)^{n^{\mathfrak{g}}}, \qquad t \in \R^+,  z \in \C^+,
\end{align}
 where the coefficients $\{c_{\mathcal{I}} \equiv c_{\mathcal{I}}(t,z)\}$ are complex-valued deterministic functions of $t,z$, which are uniformly bounded for any $t \in \R^{+}, z \in S$ given in (\ref{ddd}) and are order one in $N \in \N$.

We denote the total number of Green function entries (including the centered diagonal Green function entries $\ud G-\wt m$ and $\ud \G+\frac{1}{1+\wt m}$) in the averaged product in (\ref{form}) by
\begin{align}\label{n_number}
n:=n^o+n^{g}+n^{\mathfrak{g}},
\end{align}
and denote the number of off-diagonal Green function entries in the product as
 \begin{align}\label{d_o}
 d^o:=\#\{ 1 \leq i\leq n^o: x_i \neq y_i\} \leq n^{o}.
 \end{align}
We further define the degree, denoted by $d$, to be the number of off-diagonal Green function entries plus the powers of the centered diagonal Green function factor $\ud G-\wt m$ and $\ud \G+\frac{1}{1+\wt m}$, \ie
\begin{align}\label{degree_0}
d:=d^o+n^g+n^{\mathfrak{g}} \leq n.
\end{align}
We use $\mathcal{Q}_d \equiv \mathcal{Q}_d(t,z)$ to denote the collection of the averaged products of Green function entries of the form in (\ref{form}) with degree $d$. For any $Q_d \equiv Q_d(t,z) \in \mathcal{Q}_d$, it directly follows from the local law in (\ref{G}) that, for any $t \geq 0$ and $z \in S$ given in (\ref{ddd}),
\begin{align}\label{localaw_0}
|Q_d(t,z)| \prec \Psi^d+N^{-1}.
\end{align}
We will often omit the parameters $z$ and $t$ for notational simplicity. 

In the following, we use the symbol $\mathfrak{v}_j$ to denote the free summation indices in $\mathcal{I}$ taking values in either $\llbracket 1,N \rrbracket$ or $ \llbracket N+1,N+M \rrbracket$, and the letters $x_i$, $y_i$ as the row and column indices of the Green function entries. In order to avoid confusion, we clarify that $x_i=y_i=\mathfrak{v}_j$ means that both $x_i$ and $y_i$ stand for the same element $\mathfrak{v}_j \in \mathcal{I}$. Further we write $x_i \neq y_i$, if $x_i$ and $y_i$ represent two distinct summation indices in~$\mathcal{I}$. They could have the same value when summing over $\mathcal{I}$.

\subsection{Unmatched terms}
From (\ref{dH}), we observe that the third order terms corresponding to $p+1=3$ in (\ref{step0}) can be written as averaged products of Green function entries of the form in (\ref{form}) up to a factor $\sqrt{N}$, with $\mathcal{I}=\{v,b,\alpha\}$, $n^o=4$, and $n^{g}=n^{\mathfrak{g}}=0$. Some examples of these terms are
\begin{align}\label{third_term}
\sqrt{N} \frac{1}{N^2}\sum_{v,\alpha, b}   s^{(3)}_{ \alpha b}(t) \E \Big[ G_{v \alpha} G_{b v} G_{\alpha \alpha} G_{bb} \Big], \quad \sqrt{N} \frac{1}{N^2}\sum_{v,\alpha, b}   s^{(3)}_{ \alpha b}(t) \E \Big[ G_{v \alpha} G_{\alpha v} G_{\alpha b} G_{bb} \Big].
\end{align}
Note that the indices $\alpha$ and $b$ appear three times as the row or column index in the product of Green function entries and thus all the resulting third order terms are unmatched as defined next. The fourth order terms for $p+1=4$ are then matched and are discussed in the next subsection.

\begin{definition}[Unmatched terms, unmatched indices]\label{unmatch_def}
Consider a term of the form in (\ref{form}) with degree $d$, denoted by $Q_d \in \mathcal{Q}_d$. For any summation index $\mathfrak{v}_j \in \mathcal{I}$, let $\nn(\mathfrak{v}_j)$ be the total number of appearances of the index $\mathfrak{v}_j$ as the row or column index in the product of Green function entries, \ie
\begin{align}\label{nu_number}
\nn(\mathfrak{v}_j):=&\#\{ 1 \leq i \leq n^o: x_i=\mathfrak{v}_j \}+\#\{ 1 \leq i \leq n^o: y_i=\mathfrak{v}_j \}.
\end{align}
If a summation index $\mathfrak{v}_j \in \mathcal{I}$ appears an odd number of times as the row or column index in the product of Green function entries, then we say that the summation index $\mathfrak{v}_j$ is unmatched. The set of all unmatched summation indices is defined as 
\begin{align}\label{unmatch_set}
\mathcal{I}^o:=\{\mathfrak{v}_j \in \mathcal{I} :  \nn(\mathfrak{v}_j) \mbox{ is odd} \}.
\end{align}
If $\mathcal{I}^o=\emptyset$, then we say $Q_d$ is a matched term. Otherwise, $Q_d$ is called an unmatched term, denoted by $Q^o_d$. The collection of unmatched terms of the form in (\ref{form}) with degree $d$ is denoted by $\mathcal{Q}_d^o \subset \mathcal{Q}_d$. 
\end{definition}

The following proposition asserts that the expectation of any unmatched term is much smaller than its naive size obtained by power counting using the local law as in (\ref{localaw_0}). 
\begin{proposition}\label{unmatch_lemma}
Consider an unmatched term $Q^o_d \in \mathcal{Q}_d^o$ of the form in (\ref{form}) with fixed $n \in \N$ given in (\ref{n_number}). For any fixed integer $D \geq d+1$, we have
\begin{align}\label{unmatch_lemma equation}
|\E[Q^o_d(t,z)]|=O_{\prec}\big(N^{-1}+\Psi^D\big)\,,
\end{align}
uniformly in $t \in \R^+$ and $z \in S$ given in (\ref{ddd}). 
\end{proposition}

Armed with Proposition \ref{unmatch_lemma}, the third order terms on the right side of (\ref{step0}) are unmatched terms of the form in (\ref{form}) with $\nn(\alpha)=\nn(b)=3$, $n=5$, up to a factor $\sqrt{N}$, and hence are bounded by $O_{\prec}\big(\sqrt{N} (N^{-1}+\Psi^D)\big)$. By choosing $D$ sufficiently large depending on $\epsilon$ in (\ref{ddd}) to make $\Psi^D \leq N^{-1}$ for large $N$ (see (\ref{control_sample})), the third order terms are bounded by $O_{\prec}(N^{-1/2})$. We can hence rewrite (\ref{step0}) as
\begin{align}\label{step1}
\frac{\dd }{\dd t}\E [m_N(t,z)]=&-\frac{1}{12N^3}  \sum_{v,\alpha,b}  s^{(4)}_{ \alpha b} \ee^{-2t} \E \Big[ \frac{ \partial^{4}  G_{v v} }{\partial h^{4}_{\alpha b} } \Big]+O_{\prec}\big(\frac{1}{\sqrt{N}}\big).
\end{align}

It then suffices to estimate the remaining fourth order terms on the right side of (\ref{step1}) which are matched terms from Definition \ref{unmatch_def}.

\subsection{Matched terms}
We now write out by (\ref{dH}) the fourth order terms on the  right side of (\ref{step1}) and observe that they are of the form in (\ref{form}) with $\mathcal{I}=\{v,b,\alpha\}$, $n^o=5$, and $n^{g}=n^{\mathfrak{g}}=0$. For example, we find terms as below
 \begin{align}\label{fourth_term}
\frac{1}{N^2} \sum_{v,\alpha, b} s^{(4)}_{ \alpha b}(t) \E \Big[ G_{v \alpha} G_{\alpha v} G_{\alpha \alpha} (G_{bb})^2 \Big], \quad \frac{1}{N^2} \sum_{v,\alpha, b} s^{(4)}_{ \alpha b}(t) \E \Big[ G_{v b} G_{bv} (G_{\alpha \alpha})^2 G_{bb} \Big].
\end{align}
In all the resulting fourth order terms, the indices $\alpha$ and $b$ appear four times as the row or column index in the product of the Green function entries, \ie $\nn(\alpha)=\nn(b)=4$, and $\nn(v)=2$. These terms are referred to as type-$\alpha b$ terms as defined next. 

\begin{definition}[Type-$\alpha b$ terms, type-$b$ terms, and type-0 terms]\label{def_type_AB}
For fixed integers $n^o, n^g, n^{\mathfrak{g}}\in \N$ and a free summation index set $\mathcal{I}:=\{\mathfrak{v}_j\}_{j=1}^m~(m \in \N)$, consider an averaged product of Green function entries as in (\ref{form}), with the two summation indices $\alpha,b$ singled out, of the form
\begin{align}\label{form_ab}
\frac{1}{N^{2+\#\mathcal{I}} }  \sum_{\alpha=N+1}^{N+M} \sum_{b=1}^N \sum_{\mathcal{I}} c_{\mathcal{I},\alpha,b}(t,z) \Big(\prod^{n^o}_{i=1} G_{x_i y_i}\Big) \Big(\ud{G}-\wt m\Big)^{n^{g}} \Big( \ud{\G}+\frac{1}{1+\wt m}\Big)^{n^{\mathfrak{g}}}, \qquad t \in \R^+,  z \in \C^+\,,
\end{align}
where each row index $x_i$ and column index $y_i$ of the Green function entries represents $\alpha$, $b$ or any summation index in $\mathcal{I}$, and the coefficients $\{c_{\mathcal{I}, \alpha,b} \equiv c_{\mathcal{I}, \alpha,b}(t,z)\}$ are uniformly bounded complex functions for any $t\in \R^+, z\in S$ given in (\ref{ddd}) and are order one in $N \in \N$. The number of all the Green function entries in the product including the centered entries $\ud G-\wt m$ and $\ud \G+\frac{1}{1+\wt m}$ is denoted by $n$, as defined in~(\ref{n_number}). The number of off-diagonal Green function entries in the product, $d^o$, is given in (\ref{d_o}). The degree of such a term, denoted by $d$, is defined as in (\ref{degree_0}). Recall that the number of appearances of any summation index $\mathfrak{v}_j \in \mathcal{I}$ as a row or column index of the Green function entries, denoted by $\nn(\mathfrak{v}_j)$, is defined in (\ref{nu_number}). We also define $\nn(\alpha)$ and $\nn(b)$ in the same way for the two special indices $\alpha$ and $b$.

A {\it type-$\alpha b$ term}, denoted by $P_d^{\alpha b}$, is of the form in (\ref{form_ab}) of degree $d$, with $\nn(\alpha)=\nn(b)=4$ for these two special indices, and $\nn(\mathfrak{v}_j)=2$ for any $\mathfrak{v}_j \in \mathcal{I}$. Moreover, there are no diagonal Green function entries in the product $\prod_{i=1}^{n^o} G_{x_i y_i}$ other than $G_{\alpha \alpha}$ and $G_{bb}$. In other words, if $x_i= y_i~(1\leq i\leq n^{o})$, then $x_i=y_i=\alpha$ or $x_i=y_i=b$. We denote by $\mathcal{P}_d^{\alpha b} \equiv \mathcal{P}_d^{\alpha b}(t,z)$ the collection of all type-$\alpha b$ terms of degree~$d$. We remark that type-$\alpha b$ terms are matched in the sense of Definition ~\ref{unmatch_def}.

Similarly, a {\it type-$b$ term}, denoted by $P_d^{b}$, is of the form in (\ref{form_ab}) of degree $d$, where $\nn(\alpha)=0$, $\nn(b)=4$, and $\nn(\mathfrak{v}_j)=2$ for all $\mathfrak{v}_j \in \mathcal{I}$. Moreover, we assume that $x_i \neq y_i~(1\leq i\leq n^{o})$ unless $x_i=y_i=b$. We denote by $\mathcal{P}_d^{\alpha b} \equiv \mathcal{P}_d^{\alpha b}(t,z)$ the collection of all type-$b$ terms of degree~$d$. We remark that, though the index $\alpha$ will no longer appear as the row or column index in the product of Green function entries, we still keep it in the notation in order to emphasize the inheritance from the original form in~(\ref{form_ab}).

Finally, a {\it type-0 term}, denoted by $P_d$, is of the form in (\ref{form_ab}) of degree $d$, where $\nn(\alpha)=\nn(b)=0$ and $\nn(\mathfrak{v}_j)=2$ for any $\mathfrak{v}_j \in \mathcal{I}$. We also assume that there are no diagonal Green function entries in the product, \ie $x_i \neq y_i~(1\leq i\leq n^{o})$.

We denote the collection of all type-0 terms of degree~$d$ by $\mathcal{P}_d \equiv \mathcal{P}_d(t,z)$. We remark that the indices $\alpha$ and $b$ will not make appearances in the product of Green function entries, but we still keep them in the notation in order to emphasize the inheritance from the original form in~(\ref{form_ab}).

\end{definition}

The next proposition claims that, in expectation, any type-$\alpha b$ term of degree $d$ can be expanded into a linear combination of type-0 terms of degrees at least $d$ up to an error $O_{\prec}(N^{-1}+\Psi^D)$, for any $D\geq d+1$. 
\begin{proposition}\label{lemma_expand_type}
Consider any type-$\alpha b$ term $P_d^{\alpha b} \in \mathcal{P}^{\alpha b}_d$ of the form in (\ref{form_ab}) of degree $d$, with fixed $n \in \N$ given in (\ref{n_number}) and $d^o \in \N$ given in (\ref{d_o}). Then for any fixed integer $D \geq d+1$, we have
\begin{align}\label{reduce}
\E [  P_d^{\alpha b}(t, z)]=\sum_{\substack{P_{d'} \in \mathcal{P}_{d'} \\ d \leq d' < D}} \E [ P_{d'}(t,z)]+O_{\prec}\Big(\frac{1}{N}+\Psi^D\Big),
\end{align}
uniformly in $t\in \R^+$ and $z \in S$ given in (\ref{ddd}), where the sum above contains at most $(32(n+4D))^{2D}$ type-0 terms of the form in (\ref{form_ab}) of degrees $d'$ satisfying $d \leq d' < D$. Moreover, the number of the Green function entries in the product in each type-0 term, is bounded by $n+4D$, and the number of off-diagonal Green function entries in each term is at least $d^o$.
\end{proposition}

Returning to the right side of (\ref{step1}), the resulting terms by (\ref{dH}) are finitely many type-$\alpha b$ terms of the form in (\ref{form_ab}) of degrees $d \geq 2$, with $n=5$ and $d^o \geq 2$ and $\nn(\alpha)=\nn(b)=4$. Using Proposition \ref{lemma_expand_type}, we can expand these type-$\alpha b$ terms into type-0 terms of degrees at least two, \ie for any fixed $D \geq 3$, 
\begin{align}\label{final}
\frac{\dd }{\dd t}\E [  m_N(t,z)]=\sum_{ \substack{P_{d} \in \mathcal{P}_{d} \\ 2 \leq d \leq D-1}} \E [ P_d(t, z) ]+O_{\prec}\Big(\frac{1}{\sqrt{N}}+\Psi^D\Big)\,,
\end{align}
uniformly in $z \in S$ and $t\in \R^+$, 
where the sum above contains at most $(CD)^{cD}$ type-0 terms of the form in (\ref{form_ab}) for some numerical constants $C,c>0$, the number of Green function entries in each term is bounded by $CD$, and the number of off-diagonal Green function entries in each term is at least two.

\subsection{Size of type-0 terms}

We next estimate the size of an arbitrary type-0 term of the form in (\ref{form}) of degree $d\geq 2$, with the number of the off-diagonal Green function entries in the product, $d^o$, at least two, in the edge scaling, \ie when the spectral parameter $z$ is chosen in the domain $S_{\mathrm{edge}}$ given by~\eqref{S_edge}. 
\begin{proposition}\label{lemma_trace_sample}
Consider any type-0 term $P_d \in \mathcal{P}_d$ of the form in (\ref{form_ab}) of degree $d \geq 2$, with fixed $n$ given in (\ref{n_number}) and $d^o \geq 2$ given in (\ref{d_o}). Then we have the estimate
\begin{align}\label{tracek_sample}
|\E [P_d(t,z)]| =O_{\prec}(N^{-1/3-\epsilon}),
\end{align}
uniformly in $t \in \R^+$ and $z \in S_{\mathrm{edge}}$. 
\end{proposition}

Integrating (\ref{final}) over $ [t_0, T]$ with $T:=8 \log N$ and $0 \leq t_0\leq T$, and combining with the estimates in Proposition \ref{lemma_trace_sample}, we find that
\begin{align}\label{integrate_imn}
\E[m_N(T,z)]-\E[m_N(t_0,z)]=&\sum_{ \substack{ P_{d} \in \mathcal{P}_d \\2\leq d < D}} \int_{t_0}^{T} \E[ P_d(t,z)] \dd t +O_{\prec}\big(\log N (\frac{1}{\sqrt{N}}+\Psi^D)\big)=O_{\prec}(N^{-\frac{1}{3}-\epsilon})\,,
\end{align}
by choosing $D$ sufficiently large, depending on $\epsilon$. For $t \geq 8 \log N$, from (\ref{dyson_flow}) and (\ref{blockG}), $G(t,z)$ is close to the $G^{W}:=(H^{W})^{-1}$ with $H^{W}$ defined as in (\ref{block}) by replacing $X(t)$ with the Gaussian matrix $W$, \ie
\begin{align}\label{approxxxx}
\|G(t,z) -G^{\mathrm{W}}(z)\|_{\mathrm{max}} \leq &  \|G(t,z) -G^{\mathrm{W}}(z)\|_{\mathrm{2}} \leq \|G(t,z)(H^{W}(z)-H(t,z))G^{\mathrm{W}}(z)\|_2 \nonumber\\
&\leq \frac{C{N}}{\eta^2} \|H^{W}(z)-H(t,z))\|_{\mathrm{max}}  \prec \frac{1}{N^3\eta^2},
\end{align}
where in the first step we used the matrix norm inequality $\|A\|_{\mathrm{max}} \leq \|A\|_{2} $ for $A\in \C^{m \times n}$, in the second step we used the second resolvent identity, and in the third step we used the deterministic bound of $\|G\|_2$ in (\ref{deter_bound}) and that $\|A\|_{2} \leq \sqrt{m n}\|A\|_{\mathrm{max}} $ for $A\in \C^{m \times n}$. Thus we have
$$\big| \E[m_N(T,z)] -\E^{W}[m_N(z)] \big|=O_{\prec}(N^{-1}).$$
Combining with the estimate in (\ref{integrate_imn}), we complete the proof of the comparison in (\ref{compare_mn}). The estimate in (\ref{img_sample}) hence follows from the comparison in (\ref{compare_mn}) and the corresponding estimate (\ref{traceimg}) of Lemma \ref{lemma_trace} for white Wishart ensemble shown below.

\section{Proof of Proposition~\ref{unmatch_lemma} and~\ref{lemma_expand_type}}\label{sec:expand}

Before giving the proofs of Proposition~\ref{unmatch_lemma} and Proposition~\ref{lemma_expand_type}, we first introduce the expansion mechanism applied to an averaged product of Green function entries of the form in (\ref{form}).

Recall that we use $\mathcal{I}_N$ to denote the summation indices that take values in $\llbracket 1, N \rrbracket$ and we use Latin letters to denote the indices in $\mathcal{I}_N$.  We use $\mathcal{I}_M$ to denote  summation indices that take values in $\llbracket N+1, N+M \rrbracket$ and we use Greek letters to denote the elements in $\mathcal{I}_M$. We set $\mathcal{I}=\mathcal{I}_N \cup \mathcal{I}_M$ and use the calligraphic letters, \eg $\mathfrak{i}, \mathfrak{j}$, $\mathfrak{v}$, to denote the summation indices ranging from $1$ to $N+M$. 

Using the definition of the Green function $G(t,z)$ in (\ref{blockG}) and that $X(t)$ in (\ref{block}) is real-valued, we have the following symmetry in the indices for the Green function 
\begin{align}\label{symmetry}
G_{\mathfrak{ij}}(t,z)=G_{\mathfrak{ji}}(t,z), \qquad  \mathfrak{i}, \mathfrak{j} \in \llbracket 1, N+M \rrbracket.
\end{align}
For convenience, we will always put the index used for an expansion in the row position of the Green function entries. In combination with (\ref{block}) and (\ref{le time dependent G}), we obtain the useful identities
\begin{align}\label{id_3}
G_{\alpha \mathfrak{i}}= \sum_{k=1}^{N} H_{\alpha k} G_{k \mathfrak{i}}-\delta_{\alpha \mathfrak{i}}, ~ \alpha \in \llbracket N+1, N+M \rrbracket; \qquad zG_{b \mathfrak{i}}=\sum_{\gamma=N+1}^{N+M} H_{b \gamma} G_{\gamma \mathfrak{i}}-\delta_{b\mathfrak{i}}, ~ b \in \llbracket 1, N \rrbracket.
\end{align}
In these formulas, we will refer to $k$ and $\gamma$ as fresh summation indices, and the spectral parameter $z$ in this section will always be chosen in the domain~$S$ given in (\ref{ddd}).

Now, we are ready to introduce the mechanism to expand the averaged product of Green function entries in (\ref{form}) by combining the identities in~(\ref{id_3}) and the cumulant expansion formulas in Lemma \ref{cumulant}.

\subsection{Expansion mechanism I}\label{sec:expand_1}

In this subsection, we expand a Green function entry in a term in~(\ref{form}) using indices taking values in $\llbracket N+1, N+M \rrbracket$. Corresponding expansions using indices in $\llbracket 1, N \rrbracket$ are given in the next subsection.

Consider an averaged product of Green function entries of the form in (\ref{form}) of degree $d$, denoted by $Q_d \in \mathcal{Q}_d$. Let $n$, defined in (\ref{n_number}), be the total number of Green function entries in the product and $d^o$, given in~(\ref{d_o}), be the number of off-diagonal Green function entries in the product. Recall from (\ref{nu_number}) that the number of appearances of a free summation index $\mathfrak{v}_j \in \mathcal{I}=\mathcal{I}_N \cup \mathcal{I}_M$ as the row or column index in the product of Green function entries is denoted by $\nn(\mathfrak{v}_j)$. We further define the number of appearances of the index $\mathfrak{v}_j$ as the row or column index of off-diagonal Green function factors in the product as
\begin{align}\label{nu_number_0}
\nn^o(\mathfrak{v}_j):=&\#\{ 1 \leq i \leq n^o: x_i=\mathfrak{v}_j, y_i \neq \mathfrak{v}_j \mbox{~or~}x_i\neq \mathfrak{v}_j, y_i=\mathfrak{v}_j \}.
\end{align}
We remark that $\nn(\mathfrak{v}_j)$ and $\nn^o(\mathfrak{v}_j)$ are either simultaneously even or odd.

Let $\alpha \in \mathcal{I}_M$ be a summation index in $Q_d$, with $\nn(\alpha) \geq 1$. Then we split the discussion in two cases.

{\bf Case 1$\alpha$: Eliminating two $\alpha$'s from $G_{\alpha \alpha}$.}  In this first case, we expand a diagonal Green function entry $G_{\alpha \alpha}$ in a product of the form (\ref{form}). We may assume $G_{x_1y_1}=G_{\alpha \alpha}$. Using the first identity in~(\ref{id_3}) on $G_{\alpha \alpha}$ and applying the cumulant expansions in Lemma \ref{cumulant} on the resulting $\{H_{\alpha k}\}$, we have
\begin{align}\label{temp_1}
\E[Q_d(t,z)]=&\frac{1}{N^{\#\mathcal{I}} }\sum_{\mathcal{I}} c_{\mathcal{I}}\E \Big[  G_{\alpha \alpha}\Big(\prod^{n^o}_{i=2} G_{x_i y_i}\Big) \Big(\ud{G}-\wt m\Big)^{n^g} \Big( \ud{\G}+\frac{1}{1+\wt m}\Big)^{n^{\mathfrak{g}}}\Big]\nonumber\\
=&\frac{1}{N^{\#\mathcal{I}} }\sum_{\mathcal{I}} c_{\mathcal{I}} \sum_{k=1}^N \E \Big[  H_{\alpha k}G_{k \alpha}  \Big(\prod^{n^o}_{i=2} G_{x_i y_i}\Big) \Big(\ud{G}-\wt m\Big)^{n^g} \Big( \ud{\G}+\frac{1}{1+\wt m}\Big)^{n^{\mathfrak{g}}}\Big]\nonumber\\
&-\frac{1}{N^{\#\mathcal{I}} }\sum_{\mathcal{I}} c_{\mathcal{I}}  \E \Big[ \Big(\prod^{n^o}_{i=2} G_{x_i y_i}\Big) \Big(\ud{G}-\wt m\Big)^{n^g} \Big( \ud{\G}+\frac{1}{1+\wt m}\Big)^{n^{\mathfrak{g}}}\Big]\nonumber\\
=&\frac{1}{N^{1+\#\mathcal{I}} }\sum_{\mathcal{I},k} c_{\mathcal{I}} \E \bigg[  \frac{ \partial G_{k \alpha}  \Big(\prod^{n^o}_{i=2} G_{x_i y_i}\Big) \Big(\ud{G}-\wt m\Big)^{n^g} \Big( \ud{\G}+\frac{1}{1+\wt m}\Big)^{n^{\mathfrak{g}}}}{\partial h_{\alpha k}}\bigg]\nonumber\\
&+ \frac{1}{\sqrt{N} }\frac{1}{2 N^{1+\#\mathcal{I}} }\sum_{\mathcal{I},k} c_{\mathcal{I}} s^{(3)}_{\alpha k}(t) \E \bigg[  \frac{ \partial^2 G_{k \alpha}  \Big(\prod^{n^o}_{i=2} G_{x_i y_i}\Big) \Big(\ud{G}-\wt m\Big)^{n^g} \Big( \ud{\G}+\frac{1}{1+\wt m}\Big)^{n^{\mathfrak{g}}}}{\partial h^2_{\alpha k}}\bigg]\nonumber\\
&-\frac{1}{N^{\#\mathcal{I}} }\sum_{\mathcal{I}} c_{\mathcal{I}} \E \Big[ \Big(\prod^{n^o}_{i=2} G_{x_i y_i}\Big) \Big(\ud{G}-\wt m\Big)^{n^g} \Big( \ud{\G}+\frac{1}{1+\wt m}\Big)^{n^{\mathfrak{g}}}\Big]+O_{\prec}(N^{-1}),
\end{align}
uniformly for any $t \in \R^+$ and $z \in S$ given in (\ref{ddd}), where $\{s^{(3)}_{\alpha k}(t)\}$ are the third order cumulants of the normalized matrix entries $\sqrt{N} h_{\alpha k}(t)$ given in (\ref{t_cumu}), and the error $O_{\prec}(N^{-1})$ stems from truncating the cumulant expansions at the third order.

Using the differentiation rule in (\ref{dH}), the third order terms in the cumulant expansions above can be written as at most $4n(n+1)$ terms of the form in (\ref{form}) with an additional factor $\frac{1}{\sqrt{N}}$ in front, where the new summation index set $\mathcal{I}'=\{\mathcal{I},k\}$, $n'=n+2$, and the new coefficients $c_{\mathcal{I}'}=\frac{1}{2}c_{\mathcal{I}} s^{(3)}_{\alpha k}$ are uniformly bounded. Most importantly, the number of appearances of the fresh index $k$ in the product of Green function entries is $\nn(k)=3$. Thus these third order terms are unmatched terms under Definition \ref{unmatch_def}. Since $k$ is a fresh index, the degrees of these terms, $d'$, satisfy $d' \geq d$. We use 
\begin{align}\label{higher}
\frac{1}{\sqrt{N}} \sum_{Q^o_{d'} \in \mathcal{Q}^o_{d'}; d' \geq d} \E [Q^o_{d'}],
\end{align}
to denote the finite sum of these unmatched terms from the third order expansion.

We next look at the second order terms in the cumulant expansions, \ie the first group of terms on the right side of (\ref{temp_1}). Using (\ref{dH}), the resulting $2n$ terms  are also of the form in (\ref{form}), with the new summation index set $\mathcal{I}'=\{\mathcal{I},k\}$, $n'=n+1$, and the new coefficients $c_{\mathcal{I}'} \equiv -c_{\mathcal{I}}$. It is straightforward to check that the fresh index $k$ has the number of appearances $\nn(k)=2$ and the number of appearances of any original summation index in $\mathcal{I}$ (including the index $\alpha$) remains the same as that of the original term~$Q_d$.  Among them, the leading term of degree $d$ comes from $\frac{\partial}{\partial h_{\alpha k}}$ acting on $G_{k\alpha}$, \ie
\begin{align}\label{temp_2}
-\frac{1}{N^{1+\#\mathcal{I}} }\sum_{\mathcal{I}} \sum_{k=1}^N c_{\mathcal{I}} \E \Big[ G_{kk} G_{\alpha \alpha}\Big(\prod^{n^o}_{i=2} G_{x_i y_i}\Big) \Big(\ud{G}-\wt m\Big)^{n^g} \Big( \ud{\G}+\frac{1}{1+\wt m}\Big)^{n^{\mathfrak{g}}}\Big]\nonumber\\
=-\E[\ud{G} Q_d]=-\wt m \E[Q_d]-\E[(\ud{G}-\wt m) Q_d],
\end{align}
where $\wt m$ is the deterministic function defined in (\ref{mplawgamma-1}). The first term on the right side of (\ref{temp_2}) can be absorbed into the left side of (\ref{temp_1}) by considering $(1+\wt m) \E[Q_d]$. In addition, the second term on the right side of (\ref{temp_2}) gains an additional centered factor $\ud{G}-\wt m$. Thus its degree is increased to $d+1$, and the number of off-diagonal Green function entries in the product remains the same as $d^o$.

We next discuss the remaining $2n-1$ terms from the first group on the right side of (\ref{temp_1}), whose degrees are higher than $d$. Since $k$ is a fresh index, the terms from taking $\frac{\partial}{\partial h_{\alpha k}}$ on $G_{k\alpha}\prod_{i=2}^{n^o}G_{x_i y_i}$, except the one in (\ref{temp_2}), have degrees $d' \geq d+1$ and the number of off-diagonal Green function entries in each resulting term is also increased by at least one, \ie $(d^o)' \geq d^o+1$. Else if $\frac{\partial}{\partial h_{\alpha j}}$ acts on the centered factors $(G-\wt m)^{n^g}$ or $\Big( \ud{\G}+\frac{1}{1+\wt m}\Big)^{n^{\mathfrak{g}}}$, then we obtain the following two terms
\begin{align}\label{temp_3}
-\frac{2}{N^{2+\#\mathcal{I}} }\sum_{\mathcal{I},k,v} c_{\mathcal{I}}   G_{k \alpha} G_{v k} G_{ \alpha v }  \Big(\prod^{n^o}_{i=2} G_{x_i y_i}\Big) \Big(\ud{G}-\wt m\Big)^{n^g-1} \Big( \ud{\G}+\frac{1}{1+\wt m}\Big)^{n^{\mathfrak{g}}},\nonumber\\
-\frac{2}{ N^{2+\#\mathcal{I}} }\sum_{\mathcal{I},k,\nu} \frac{c_{\mathcal{I}} }{\varrho}  G_{k \alpha} G_{\nu k} G_{ \alpha \nu }  \Big(\prod^{n^o}_{i=2} G_{x_i y_i}\Big) \Big(\ud{G}-\wt m\Big)^{n^g} \Big( \ud{\G}+\frac{1}{1+\wt m}\Big)^{n^{\mathfrak{g}}-1},
\end{align}
with the aspect ratio $\varrho$ given in (\ref{ratio}), which stems from the definition of $\ud \G$ in (\ref{ggu}). Compared with the original term $Q_d$, the degrees of these two terms are increased by two, \ie $d'=d+2$, and the number of off-diagonal Green function entries is increased by three, \ie $(d^o)' = d^o+3$. We remark that for the above two terms, we have created additional free summation indices $v \in \llbracket 1, N \rrbracket$ and $\nu \in \llbracket N+1, N+M \rrbracket$ with $\nn(v)=\nn(\nu)=2$, which comes from taking $\frac{\partial}{\partial h_{\alpha k}}$ on $\ud{G}$ and $\ud{\G}$ given in (\ref{ggu}). To avoid confusion, we will not record these fresh summation indices created in this way in the notation. 

For short, we denote the above $2n-1$ terms of degrees at least $d+1$, together with the second term on the right side of (\ref{temp_2}) of degree $d+1$, by
\begin{align}\label{second}
 \sum_{ \substack{Q_{d'} \in \mathcal{Q}_{d'}  \\ d' \geq d+1 }} \E[Q_{d'}].
\end{align}

Therefore, after moving the leading term in (\ref{temp_2}) to the left side of (\ref{temp_1}) and dividing both sides of (\ref{temp_1}) by $1+\wt m \sim 1$ (see (\ref{1+tm})), together with the shorthand notations in (\ref{higher}) and (\ref{second}), we obtain that
\begin{align}\label{temp_4}
\E[Q_d(t,z)]=&- \frac{1}{1+\wt m} \frac{1}{N^{\#\mathcal{I}} }\sum_{\mathcal{I}} c_{\mathcal{I}} \E \Big[ \Big(\prod^{n^o}_{i=2} G_{x_i y_i}\Big) \Big(\ud{G}-\wt m\Big)^{n^g} \Big( \ud{\G}+\frac{1}{1+\wt m}\Big)^{n^{\mathfrak{g}}}\Big]\nonumber\\
&+\frac{1}{1+\wt m}  \sum_{ \substack{Q_{d'} \in \mathcal{Q}_{d'}  \\ d' \geq d+1 }} \E[Q_{d'}]+\frac{1}{1+\wt m} \frac{1}{\sqrt{N}} \sum_{ \substack{Q^o_{d'} \in \mathcal{Q}^o_{d'}\\ d' \geq d}} \E [Q^o_{d'}]+O_{\prec}(N^{-1}),
\end{align}
uniformly for any $t \in \R^+$ and $z \in S$, where the first term of degree $d$ on the right side is obtained from the original term $Q_d$ by replacing the factor $G_{\alpha \alpha}$ by the deterministic function $-\frac{1}{1+\wt m}$. We use $Q_d\big( G_{\alpha \alpha} \rightarrow -\frac{1}{1+\wt m}\big)$ to denote such a term, and write (\ref{temp_4}) in short as
\begin{align}\label{case_1a}
\E[Q_d]=&\E \Big[  Q_d\big( G_{\alpha \alpha} \rightarrow -\frac{1}{1+\wt m}\big)\Big]+\sum_{ \substack{Q_{d''} \in \mathcal{Q}_{d''}  \\ d'' \geq d+1 }} \E[Q_{d''}]+\frac{1}{\sqrt{N}} \sum_{\substack{ Q^o_{d'} \in \mathcal{Q}^o_{d'}\\ d' \geq d}} \E [Q^o_{d'}]+O_{\prec}(N^{-1}),
\end{align}
where we have eliminated the diagonal Green function entry $G_{\alpha \alpha}$ for the leading term, and thus the number of appearances of the index $\alpha$ in the product of Green function entries is reduced to $\nn(\alpha)-2$. The second group of terms on the right side of (\ref{case_1a}), collected from the second order terms in the cumulant expansions, are at most $2n$ terms of the form in (\ref{form}) of degrees at least $d+1$, where the number of Green function entries in the product in each term is $n+1$ and the number of the off-diagonal Green function entries in each term is at least $d^o$. 
Moreover, the number of appearances of any original summation indices (including $\alpha$) is the same as that of $Q_d$, and the number of appearances of the fresh indices created in the expansions (\eg the index $k$ in (\ref{temp_1}), $v$ and $\nu$ in (\ref{temp_3})) is exactly two. 
Lastly, the third group of terms are at most $4(n+1)^2$ unmatched terms of the form in (\ref{form}) up to a factor $\frac{1}{\sqrt{N}}$, which are collected from the third order terms in the cumulant expansions.

{\bf Case 2$\alpha$: Eliminating two $\alpha$'s from two off-diagonal Green function entries.}  In this second case, we expand an off-diagonal Green function entry $G_{x_i y_i}~(1\leq i \leq n^o)$ with $\alpha \in \{x_i,y_i\}$ in the product in (\ref{form}). We may assume $G_{x_1y_1}=G_{\alpha y_1}$ with $y_1 \neq \alpha$.  Using the first identity in (\ref{id_3}) on $G_{\alpha y_1}$ and applying the cumulant expansions in Lemma \ref{cumulant}, we have
\begin{align}\label{temp_11}
\E[Q_d(t,z)]=&\frac{1}{N^{\#\mathcal{I}} }\sum_{\mathcal{I}} c_{\mathcal{I}} \E \Big[  G_{\alpha y_1}\Big(\prod^{n^o}_{i=2} G_{x_i y_i}\Big) \Big(\ud{G}-\wt m\Big)^{n^g} \Big( \ud{\G}+\frac{1}{1+\wt m}\Big)^{n^{\mathfrak{g}}}\Big]\nonumber\\
=&\frac{1}{N^{\#\mathcal{I}} }\sum_{\mathcal{I}} c_{\mathcal{I}} \sum_{k=1}^N \E \Big[  H_{\alpha k}G_{k y_1}  \Big(\prod^{n^o}_{i=2} G_{x_i y_i}\Big) \Big(\ud{G}-\wt m\Big)^{n^g} \Big( \ud{\G}+\frac{1}{1+\wt m}\Big)^{n^{\mathfrak{g}}}\Big]\nonumber\\
&+\frac{1}{N^{\#\mathcal{I}} }\sum_{\mathcal{I}} c_{\mathcal{I}} \sum_{k=1}^N \E \Big[  \delta_{\alpha y_1} \Big(\prod^{n^o}_{i=2} G_{x_i y_i}\Big) \Big(\ud{G}-\wt m\Big)^{n^g} \Big( \ud{\G}+\frac{1}{1+\wt m}\Big)^{n^{\mathfrak{g}}}\Big]\nonumber\\
=&\frac{1}{N^{1+\#\mathcal{I}} }\sum_{\mathcal{I},k} c_{\mathcal{I}} \E \bigg[  \frac{ \partial G_{k y_1}  \Big(\prod^{n^o}_{i=2} G_{x_i y_i}\Big) \Big(\ud{G}-\wt m\Big)^{n^g} \Big( \ud{\G}+\frac{1}{1+\wt m}\Big)^{n^{\mathfrak{g}}}}{\partial h_{\alpha k}}\bigg]\nonumber\\
&+\frac{1}{\sqrt{N}} \frac{1}{2 N^{1+\#\mathcal{I}} }\sum_{\mathcal{I},k} c_{\mathcal{I}} s^{(3)}_{\alpha k}(t) \E \bigg[ \frac{ \partial^2 G_{k y_1}  \Big(\prod^{n^o}_{i=2} G_{x_i y_i}\Big) \Big(\ud{G}-\wt m\Big)^{n^g} \Big( \ud{\G}+\frac{1}{1+\wt m}\Big)^{n^{\mathfrak{g}}}}{\partial h^2_{\alpha k}}\bigg]+O_{\prec}(N^{-1}),
\end{align}
uniformly for any $t \in \R^+$ and $z \in S$, where the error $O_{\prec}(N^{-1})$ comes from the truncation of the cumulant expansions at the third order and the case of index coincidence when $y_1 \equiv \alpha$. The third order terms with $\{s^{(3)}_{\alpha k}(t)\}$, are at most $4n(n+1)$ unmatched terms of the form in (\ref{form}) with an additional factor $\frac{1}{\sqrt{N}}$ in front, similarly as estimated in (\ref{higher}).

Using (\ref{dH}), the resulting $2n$ second order terms, \ie from the first group of terms on the right side of~(\ref{temp_11}) are also of the form in (\ref{form}), where $\mathcal{I}'=\{\mathcal{I},k\}$, $n'=n+1$, and $c_{\mathcal{I}'} \equiv -c_{\mathcal{I}}$. The fresh index~$k$ has the number of appearances $\nn(k)=2$, and the number of appearances of any original summation index in $\mathcal{I}$ (including the index $\alpha$) remains the same as that of $Q_d$. One of the leading terms of degree $d$ from $\frac{\partial}{\partial h_{\alpha k}}$ acting on $G_{k y_1}$ is estimated as in (\ref{temp_2}), where the first term can be absorbed into the left side of~(\ref{temp_11}) by considering $(1+\wt m) \E[Q_d]$.

Similarly, if $\frac{\partial}{\partial h_{\alpha k}}$ acts on the centered factors $(G-\wt m)^{n^g}$ or $\Big( \ud{\G}+\frac{1}{1+\wt m}\Big)^{n^{\mathfrak{g}}}$, as discussed in (\ref{temp_3}), each resulting term has degree $d' =d+1$ and the number of off-diagonal Green function entries in the product is given by $(d^o)'=d^o+2$. In addition, the other terms from taking $\frac{\partial}{\partial h_{\alpha k}}$ on $G_{\alpha y_1}\prod_{i=2}^{n^o}G_{x_i y_i}$ have degrees $d' \geq d+1$ and $(d^o)' \geq d^o+1$,  except the leading ones from taking $\frac{\partial}{\partial h_{\alpha k}}$ on an off-diagonal Green function entry $G_{x_i y_i}~(2\leq i\leq n^o)$ with $\alpha \in \{x_i,y_i\}$. It is not hard to check from (\ref{nu_number_0}) that the number of these leading terms of degree $d$ is given by $\nn^o(\alpha)-1$. Using (\ref{dH}), they are obtained from the original $Q_d$ by replacing one index $\alpha$ from the expanded entry $G_{\alpha y_1}$ and a second index $\alpha$ from an off-diagonal factor $G_{x_i y_i}~(2\leq i\leq n^o)$ with a fresh index $k$, and adding a factor $G_{\alpha \alpha}$ for the replaced pair of the index $\alpha$. Hence each of these leading terms still contains $d^o$ off-diagonal Green function entries. 

For example, we may assume that $G_{x_2 y_2}=G_{\alpha y_2}$ with $y_2 \neq \alpha$. Then the resulting leading term of degree $d$ is given by
\begin{align}\label{example_1}
(*):=-\frac{1}{N^{1+\#\mathcal{I}} }\sum_{\mathcal{I},k} c_{\mathcal{I}}  \E \Big[  G_{k y_1} G_{\alpha \alpha} G_{ ky_2}\Big(\prod^{n^o}_{i=3} G_{x_i y_i}\Big) \Big(\ud{G}-\wt m\Big)^{n^g} \Big( \ud{\G}+\frac{1}{1+\wt m}\Big)^{n^{\mathfrak{g}}}\Big].
\end{align}
Note that this leading term has a diagonal factor $G_{\alpha \alpha}$, which has been discussed in Case 1$\alpha$ above. Using the expansion in (\ref{case_1a}), one can replace $G_{\alpha \alpha}$ by the deterministic function $-\frac{1}{1+\wt m}$ and expand $(*)$ as
\begin{align}\label{example_2} 
(*)=&\frac{1}{1+\wt m}\frac{1}{N^{1+\#\mathcal{I}} }\sum_{\mathcal{I},k} c_{\mathcal{I}}  \E \Big[ G_{k y_1}G_{ ky_2}\Big(\prod^{n^o}_{i=3} G_{x_i y_i}\Big) \Big(\ud{G}-\wt m\Big)^{n^g} \Big( \ud{\G}+\frac{1}{1+\wt m}\Big)^{n^{\mathfrak{g}}}\Big]\nonumber\\
&+\sum_{ \substack{Q_{d'} \in \mathcal{Q}_{d'}  \\ d' \geq d+1} } \E[Q_{d'}]+\frac{1}{\sqrt{N}} \sum_{\substack{Q^o_{d'} \in \mathcal{Q}^o_{d'}\\ d' \geq d}} \E [Q^o_{d'}]+O_{\prec}(N^{-1}).
\end{align}
We hence have replaced one pair of the index $\alpha$ from two off-diagonal Green function entries $G_{\alpha y_1}$ and $G_{\alpha y_2}$ with the fresh index $k$ for this leading term. In general, compared with the original term $Q_d$, the leading terms are obtained by replacing one index~$\alpha$ of the expanded entry $G_{\alpha y_1}$ and a second index~$\alpha$ from another off-diagonal Green function entry $G_{x_i y_i}~(2\leq i\leq n^o)$ with the fresh index $k$, up to a deterministic factor $\frac{1}{1+\wt m}$. We denote these leading terms by 
\begin{align}\label{short_temp}
\frac{1}{1+\wt m}\sum_{\substack{2\leq i \leq n^o\\ x_i=\alpha,y_i \neq \alpha}} \E\Big[Q_{d} \big( x_1,x_i=\alpha \rightarrow k \big)\Big]+\frac{1}{1+\wt m}\sum_{\substack{2\leq i \leq n^o\\x_i \neq \alpha,y_i = \alpha}} \E\Big[Q_{d} \big( x_1,y_i =\alpha\rightarrow k \big)\Big].
\end{align}

Therefore, after moving one leading term as estimated in (\ref{temp_2}) to the left side of (\ref{temp_11}) and dividing both sides of (\ref{temp_11}) by $1+\wt m \sim 1$, together with the expansions such as in (\ref{example_2}) and the third order terms denoted as in (\ref{higher}), we obtain the analogue of (\ref{case_1a}), 
\begin{align}\label{case_2a}
\E[Q_d]=&\frac{1}{(1+\wt m)^2}\sum_{\substack{2\leq i \leq n^o\\x_i=\alpha,y_i \neq \alpha}} \E \Big[ Q_{d} \big( x_1 ,x_i=\alpha \rightarrow k \big)\Big]+\frac{1}{(1+\wt m)^2}\sum_{\substack{2\leq i \leq n^o\\ x_i\neq \alpha,y_i = \alpha}} \E \Big[ Q_{d} \big( x_1, y_i =\alpha\rightarrow k \big)\Big]\nonumber\\
&+\sum_{ \substack{Q_{d''} \in \mathcal{Q}_{d''}  \\ d'' \geq d+1} } \E[Q_{d''}]+\frac{1}{\sqrt{N}} \sum_{\substack{Q^o_{d'} \in \mathcal{Q}^o_{d'}\\ d' \geq d}} \E [Q^o_{d'}]+O_{\prec}(N^{-1}),
\end{align}
uniformly in $t \in \R^+$ and $z \in S$, where the first line are $\nn^o(\alpha)-1$ leading terms of the form in (\ref{form}) of degree $d$, and the number of appearances of the index $\alpha$ in the off-diagonal Green function entries is reduced to $\nn^o(\alpha)-2$. The first group of terms on the last line of (\ref{case_2a}), collected from the second order terms in the cumulant expansions, are at most $2(n+2)^2$ terms of the form in (\ref{form}) of degrees at least~$d+1$, where the number of Green function entries in the product in each term is at most $n+2$ and the number of off-diagonal Green function entries in each term is at least $d^o$. 
Moreover, the number of appearances of any original summation indices is the same as that of $Q_d$, and the number of appearances of the fresh indices created in the expansions is exactly two.
Lastly, the second group of terms on the last line of (\ref{case_2a}) contains at most $4(n+2)^3$ unmatched terms of the form in (\ref{form}) up to a factor $\frac{1}{\sqrt{N}}$, which are collected from the third order terms in the cumulant expansions.

Now we have eliminated two $\alpha$'s from either one diagonal Green function entry $G_{\alpha \alpha}$ or two off-diagonal Green function entries in a product in (\ref{form}) for the leading terms in the expansions, and thus the number of appearances of the index $\alpha$ in the product of Green function entries is reduced to $\nn(\alpha)-2$.

\subsection{Expansion mechanism II}\label{sec:expand_2}

In this subsection, we discuss the similar mechanism to expand a term in (\ref{form}) using indices taking values in $\llbracket 1, N \rrbracket$. Consider an index $b \in \mathcal{I}_N$ and recall that $\nn(b)$ and $\nn^o(b)$ are defined in (\ref{nu_number}) and (\ref{nu_number_0}) for the index $b$. Similarly as in Subsection \ref{sec:expand_1}, we divide the discussion in two cases.

{\bf Case 1b: Eliminating two $b$'s from $G_{bb}$.}  In the first case, we expand a diagonal Green function entry~$G_{bb}$ in the product in (\ref{form}). We may assume $G_{x_1y_1}=G_{bb}$. Using the second identity in (\ref{id_3}) on~$G_{bb}$ and applying the cumulant expansions, we obtain the analogue of (\ref{temp_1}),
\begin{align}\label{temp_13}
z\E[Q_d]=&\frac{1}{N^{\#\mathcal{I}} }\sum_{\mathcal{I}} c_{\mathcal{I}} \E \Big[ z G_{bb}\Big(\prod^{n^o}_{i=2} G_{x_i y_i}\Big) \Big(\ud{G}-\wt m\Big)^{n^g} \Big( \ud{\G}+\frac{1}{1+\wt m}\Big)^{n^{\mathfrak{g}}}\Big]\nonumber\\
=&\frac{1}{N^{1+\#\mathcal{I}} }\sum_{\mathcal{I},\gamma} c_{\mathcal{I}}  \E \bigg[ \frac{ \partial G_{\gamma b}  \Big(\prod^{n^o}_{i=2} G_{x_i y_i}\Big) \Big(\ud{G}-\wt m\Big)^{n^g} \Big( \ud{\G}+\frac{1}{1+\wt m}\Big)^{n^{\mathfrak{g}}}}{\partial h_{b \gamma}}\bigg]\nonumber\\
&-\frac{1}{N^{\#\mathcal{I}} }\sum_{\mathcal{I}} c_{\mathcal{I}} \E \Big[ \Big(\prod^{n^o}_{i=2} G_{x_i y_i}\Big) \Big(\ud{G}-\wt m\Big)^{n^g} \Big( \ud{\G}+\frac{1}{1+\wt m}\Big)^{n^{\mathfrak{g}}}\Big]\nonumber\\
&+\mbox{third order terms}+O_{\prec}(N^{-1}),
\end{align}
uniformly in $t \in \R^+$ and $z \in S$, where the third order terms are given by at most $4n(n+1)$ unmatched terms of the form in (\ref{form}) up to a factor $\frac{1}{\sqrt{N}}$, with $\mathcal{I}'=\{\mathcal{I}, \gamma\}$, $n'=n+2$ and $\nn(\gamma)=3$, similarly as denoted in~(\ref{higher}). 

Using (\ref{dH}), the resulting terms from the second order terms, \ie the first group of terms on the right side of (\ref{temp_13}), are also of the form in (\ref{form}), with $\mathcal{I}'=\{\mathcal{I},\gamma\}$, $n'=n+1$ and $\nn(\gamma)=2$. The number of appearances of any original summation index in $\mathcal{I}$ (including the index $b$) remains the same as that of~$Q_d$. The leading term of degree $d$ stems from $\frac{\partial}{\partial h_{b \gamma}}$ acting on $G_{b \gamma}$, \ie
\begin{align}\label{temp_23}
-\frac{1}{N^{1+\#\mathcal{I}} }\sum_{\mathcal{I},\gamma} c_{\mathcal{I}} \E \Big[ G_{\gamma \gamma} G_{bb}  \Big(\prod^{n^o}_{i=2} G_{x_i y_i}\Big) \Big(\ud{G}-\wt m\Big)^{n^g} \Big( \ud{\G}+\frac{1}{1+\wt m}\Big)^{n^{\mathfrak{g}}}\Big]\nonumber\\
=-\varrho \E[\ud{\G} Q_d]= \frac{\varrho}{1+\wt m} \E[Q_d] -\varrho  \E \Big[ \Big( \ud{\G} +\frac{1}{1+\wt m}\Big) Q_d\Big],
\end{align}
with $\varrho$ given in (\ref{ratio}), which comes from the definition of $\G$ in (\ref{ggu}). The first term on the right side of~(\ref{temp_23}) can be moved to the left side of (\ref{temp_13}) and we hence obtain 
$$\Big(z-\frac{\varrho}{1+\wt m} \Big) \E[Q_d]=-\frac{1}{ \wt m} \E[Q_d],$$
using the relation in (\ref{mplawgamma-1}). The second term on the right side of (\ref{temp_23}) gains an additional centered factor $\ud{\G} +\frac{1}{1+\wt m}$. Thus its degree is increased to $d+1$, and the number of off-diagonal Green function entries in the product remains the same as $d^o$. 

Since $\gamma$ in (\ref{temp_13}) is a fresh index, the remaining terms from the second order expansions have degrees $d' \geq d+1$, and the number of off-diagonal Green function entries in each term, $(d^o)'$, satisfies $(d^o)' \geq d^o+1$. Therefore, after absorbing the leading term in (\ref{temp_23}) into the left side of (\ref{temp_13}) by considering $-\frac{1}{ \wt m} \E[Q_d]$ and multiplying both sides by $-\wt m$, we obtain the analogue of (\ref{case_1a}),
\begin{align}\label{case_1b}
\E [Q_d]=&\E \Big[  Q_d\big( G_{bb} \rightarrow \wt m \big)\Big]+\sum_{ \substack{Q_{d''} \in \mathcal{Q}_{d''} \\ d'' \geq d+1 }} \E[Q_{d''}]+\frac{1}{\sqrt{N}} \sum_{\substack{Q^o_{d'} \in \mathcal{Q}^o_{d'}\\ d' \geq d}} \E [Q^o_{d'}]+O_{\prec}(N^{-1}),
\end{align}
where the leading term of degree $d$ is obtained from $Q_d$ by replacing the diagonal factor~$G_{bb}$ with the deterministic function $\wt m$, and the remaining terms are described similarly as below (\ref{case_1a}).

{\bf Case 2b: Eliminating two $b$'s from two off-diagonal Green function entries.}  In the second case, we expand an off-diagonal Green function entry $G_{x_i y_i}~(1\leq i \leq n^o)$ with $b \in \{x_i,y_i\}$ in the product in (\ref{form}). We may assume $G_{x_1y_1}=G_{b y_1}$ with $y_1 \neq b$.  Using the second identity in (\ref{id_3}) on $G_{b y_1}$ and applying the cumulant expansions, we have the analogue of (\ref{temp_11})
\begin{align}\label{temp_113}
z\E[Q_d]=&\frac{1}{N^{\#\mathcal{I}} }\sum_{\mathcal{I}} c_{\mathcal{I}} \E \Big[zG_{b y_1}\Big(\prod^{n^o}_{i=2} G_{x_i y_i}\Big) \Big(\ud{G}-\wt m\Big)^{n^g} \Big( \ud{\G}+\frac{1}{1+\wt m}\Big)^{n^{\mathfrak{g}}}\Big]\nonumber\\
=&\frac{1}{N^{1+\#\mathcal{I}} }\sum_{\mathcal{I},\gamma} c_{\mathcal{I}}  \E \bigg[ \frac{ \partial G_{\gamma y_1}  \Big(\prod^{n^o}_{i=2} G_{x_i y_i}\Big) \Big(\ud{G}-\wt m\Big)^{n^g} \Big( \ud{\G}+\frac{1}{1+\wt m}\Big)^{n^{\mathfrak{g}}}}{\partial h_{b \gamma}}\bigg]\nonumber\\
&+\mbox{third order terms}+O_{\prec}(N^{-1}),
\end{align}
uniformly in $t \in \R^+$ and $z \in S$, where the third order terms can be estimated similarly as in (\ref{higher}) and the error comes from the truncation of the cumulant expansions and the case when $b \equiv y_1$. 

Using (\ref{dH}), the resulting terms from the second order terms are also of the form in (\ref{form}), with $\mathcal{I}'=\{\mathcal{I},\gamma\}$, $n'=n+1$ and $\nn(\gamma)=2$. The number of appearances of any original summation index in $\mathcal{I}$ (including the index $b$) remains the same as that of $Q_d$. One of the leading terms of degree $d$ from $\frac{\partial}{\partial h_{b \gamma}}$ acting on $G_{\gamma y_1}$ is estimated as in (\ref{temp_23}), where the first term can be absorbed into the left side of (\ref{temp_113}) by considering $-\frac{1}{ \wt m} \E[Q_d]$. Similarly as discussed in Case $2\alpha$, since $\gamma$ is a fresh index, the other terms have degrees $d' \geq d+1$ and $(d^o)' \geq d^o+1$, except the ones from taking $\frac{\partial}{\partial h_{b \gamma}}$ on an off-diagonal Green function entry $G_{x_i y_i}~(2\leq i\leq n^o)$ with $b \in \{x_i,y_i\}$. The number of these leading terms of degree $d$ is then given by $\nn^o(b)-1$. For example, we may assume that $G_{x_2 y_2}=G_{b y_2}$ with $y_2 \neq b$. Then the resulting leading term of degree $d$ is given by
\begin{align}\label{example_14}
(**):=-\frac{1}{N^{1+\#\mathcal{I}} }\sum_{\mathcal{I},\gamma} c_{\mathcal{I}}  \E \Big[  G_{\gamma y_1} G_{b b} G_{ \gamma y_2}\Big(\prod^{n^o}_{i=3} G_{x_i y_i}\Big) \Big(\ud{G}-\wt m\Big)^{n^g} \Big( \ud{\G}+\frac{1}{1+\wt m}\Big)^{n^{\mathfrak{g}}}\Big].
\end{align}
Note that $(**)$ contains a factor $G_{bb}$ which has been previously discussed in Case $1b$.  One can further replace $G_{bb}$ by the deterministic function $\wt m$ and expand $(**)$  using (\ref{case_1b}). We hence have eliminated one pair of the index $b$ from two off-diagonal Green function entries for this leading term. In general, compared with the original term $Q_d$, all the leading terms are obtained by replacing one index $b$ from the expanded entry $G_{b y_1}$ and a second index $b$ from another off-diagonal factor $G_{x_i y_i}~(2\leq i\leq n^o)$ with a fresh index $\gamma$, up to a deterministic factor $-\wt m $. We denote these leading terms by 
\begin{align}\label{example_25}
-\wt m \sum_{\substack{2\leq i \leq n^o\\x_i=b,y_i \neq b}} \E \Big[ Q_{d} \big(  x_1,x_i=b \rightarrow \gamma \big)\Big]- \wt m\sum_{\substack{2\leq i \leq n^o\\x_i \neq b,y_i = b}} \E \Big[ Q_{d} \big( x_1,y_i =b\rightarrow \gamma \big)\Big].
\end{align}

Therefore, after moving one leading term in (\ref{temp_23}) to the left side of (\ref{temp_113}) and multiplying $-\wt m$ on both sides, we obtain the analogue of (\ref{case_2a}), 
\begin{align}\label{case_2b}
\E[Q_d]=&\wt m ^2\sum_{\substack{2\leq i \leq n^o\\x_i=b,y_i \neq b}} \E \Big[ Q_{d} \big( x_1,x_i=b \rightarrow \gamma \big)\Big]+\wt m^2\sum_{\substack{2\leq i \leq n^o\\x_i \neq b,y_i = b}}  \E \Big[ Q_{d} \big( x_1,y_i =b\rightarrow \gamma \big)\Big]\nonumber\\
&+\sum_{ \substack{Q_{d''} \in \mathcal{Q}_{d''}  \\d'' \geq d+1} } \E[Q_{d''}]+\frac{1}{\sqrt{N}} \sum_{\substack{Q^o_{d'} \in \mathcal{Q}^o_{d'}\\ d' \geq d}} \E [Q^o_{d'}]+O_{\prec}(N^{-1}),
\end{align}
where the first line contains $\nn^o(b)-1$ leading terms of the form in (\ref{form}) of degree $d$ and the number of appearances of the index $b$ in the off-diagonal Green function entries in each term is reduced to $\nn^o(b)-2$. Terms on the second line are described similarly as below (\ref{case_2a}).

To sum up, we have eliminated one pair of the index $b$ from either one diagonal Green function entry~$G_{bb}$ or two off-diagonal Green function entries in a product of the form (\ref{form}) for the leading terms in the expansions, and thus the number of appearances of the index $b$ in the product of Green function entries is reduced to $\nn(b)-2$.

\subsection{Proof of Proposition~\ref{unmatch_lemma}}
We are now prepared to prove Proposition~\ref{unmatch_lemma} using the above expansion mechanism iteratively.

\begin{proof}[Proof of Proposition \ref{unmatch_lemma}]
 Consider an arbitrary unmatched term $Q_d^{o} \in \mathcal{Q}_d^o$ of the form in (\ref{form}). We may first assume that an index $\alpha \in \mathcal{I}_M$ belongs to the unmatched index set $\mathcal{I}^o$ defined in (\ref{unmatch_set}). We will prove the proposition by iteratively using the expansion in (\ref{case_2a}) via the unmatched index~$\alpha$. Similar arguments also apply if there exists an unmatched index in $\mathcal{I}_N$, by using the expansion in (\ref{case_2b}) iteratively. 
 
Recall that the number of Green function entries in $Q^o_d$, $n$, is defined in (\ref{n_number}). We also recall $\nn(\alpha)$ and $\nn^o(\alpha)$, for the unmatched index $\alpha$, are defined in (\ref{nu_number}) and (\ref{nu_number_0}), with both $\nn(\alpha)$ and $\nn^o(\alpha)$ odd. The key observation is that in one expansion step in (\ref{case_2a}), the number of appearances of the index $\alpha$ in the off-diagonal Green function entries in each leading term of degree $d$ is reduced to $\nn^o(\alpha)-2$. Once $\nn^o(\alpha)$ is reduced to one, there will be no more terms of degree $d$ on the right side of (\ref{case_2a}). We hence improve the estimate of the size of $\E[Q_d^o]$ by one order in the power counting, compared to the initial estimate in (\ref{localaw_0}). We then use $\nn^o(\alpha)$ to record the iteration step and use $n$ to control the number of expansion terms. We begin the iteration for a given unmatched term $Q_d^o$ by setting the initial numbers to be
\begin{align}\label{s_00}
\nn^o(\alpha)=s_0, \qquad n=n_0,
\end{align}
where $s_0 \geq 1$ is an odd integer and $n_0 \geq 1$.
 
Since $\alpha$ is an unmatched index in $\mathcal{I}^o$, there exists at least one off-diagonal Green function entry $G_{x_i,y_i}~(1\leq i \leq n^o)$ with $\alpha \in \{x_i,y_i\}$ in the expression of $Q_d^o$. We may set $G_{x_1y_1}=G_{\alpha y_1}$ with $y_1 \neq \alpha$. Expanding this off-diagonal Green function factor $G_{\alpha y_1}$ by using (\ref{case_2a}), we obtain
\begin{align}\label{unmatch_temp_1}
\E[Q^o_d(t,z)]=&\frac{1}{N^{\#\mathcal{I}} }\sum_{\mathcal{I}} c_{\mathcal{I}} \E \Big[  G_{\alpha y_1}\Big(\prod^{n^o}_{i=2} G_{x_i y_i}\Big) \Big(\ud{G}-\wt m\Big)^{n^g} \Big( \ud{\G}+\frac{1}{1+\wt m}\Big)^{n^{\mathfrak{g}}}\Big]\nonumber\\
=&\frac{1}{(1+\wt m)^2}\sum_{\substack{2\leq i \leq n^o\\x_i=\alpha,y_i \neq \alpha}} \E\Big[ Q_{d} \big( x_1 ,x_i=\alpha \rightarrow k \big)\Big]+\frac{1}{(1+\wt m)^2}\sum_{\substack{2\leq i \leq n^o\\ x_i\neq \alpha,y_i = \alpha}} \E \Big[ Q_{d} \big( x_1, y_i =\alpha\rightarrow k \big)\Big]\nonumber\\
&+\sum_{ \substack{Q_{d''} \in \mathcal{Q}_{d''}  \\ d'' \geq d+1} } \E[Q_{d''}]+\frac{1}{\sqrt{N}} \sum_{\substack{Q^o_{d'} \in \mathcal{Q}^o_{d'}\\ d' \geq d}} \E [Q^o_{d'}]+O_{\prec}(N^{-1}),
\end{align}
uniformly in $t\in \R^+$ and $z \in S$ given in (\ref{ddd}), where the second line above are $s_0-1$ unmatched terms of the form in (\ref{form}) of degree $d$, with $\alpha \in \mathcal{I}^o$, $\nn^o(\alpha)=s_0-2$ and $n=n_0$.  The first group of terms on the last line are at most $2(n_0+2)^2$ unmatched terms of the form in (\ref{form}) of degrees at least $d+1$, with $\alpha \in \mathcal{I}^o$, $\nn^o(\alpha)=s_0$ and $n \leq n_0+2$. The second group of terms on the last line are at most $4(n_0+2)^3$ unmatched terms of the form in (\ref{form}) up to a factor $\frac{1}{\sqrt{N}}$. Therefore, we write (\ref{unmatch_temp_1}) in short as
\begin{align}\label{case_12}
\E[Q_d^o]=\sum_{\substack{Q^o_{d_1} \in \mathcal{Q}^o_{d_1} \\ d_1=d}} \E [Q^o_{d_1}]+ \sum_{\substack{Q^o_{d'_1} \in \mathcal{Q}^o_{d'_1}\\ d'_1 \geq d+1}} \E [Q^o_{d'_1}]+\frac{1}{\sqrt{N}} \sum_{\substack{Q^o_{d_1''} \in \mathcal{Q}^o_{d_1''}\\ d_1'' \geq d}} \E [Q^o_{d_1''}]+O_{\prec}\big(N^{-1}\big),
\end{align}
where we use the subscript $1$ in the degrees $d_1,d_1', d_1''$ to indicate the first iteration step. 

The good news is that the leading terms of degree $d$ still have the unmatched summation index $\alpha$, and thus these leading terms can be further expanded using (\ref{case_2a}). Then in the second step, we apply again~(\ref{case_2a}) on each resulting term $Q^o_{d_1}$ on the right side of (\ref{case_12}) with $\nn^o(\alpha)=s_0-2$ and $n = n_0$, to obtain
\begin{align}\label{case_13}
\E[Q_{d_1}^o]=\sum_{\substack{Q^o_{d_2} \in \mathcal{Q}^o_{d_2}\\ d_2=d_1}} \E [Q^o_{d_2}]+ \sum_{\substack{ Q^o_{d'_2} \in \mathcal{Q}^o_{d'_2}\\ d'_2 \geq d_1+1}} \E [Q^o_{d'_2}]+\frac{1}{\sqrt{N}} \sum_{\substack{Q^o_{d_2''} \in \mathcal{Q}^o_{d_2''}\\ d_2'' \geq d_1}} \E [Q^o_{d_2''}]+O_{\prec}\big(N^{-1}\big),
\end{align}
uniformly in $t\in \R^+$ and $z \in S$, where the leading terms of degree $d$ are $s_0-3$ unmatched terms with $\alpha \in \mathcal{I}^o$, $\nn^o(\alpha)=s_0-4$ and $n=n_0$, and we use the subscript $2$ to indicate the second step. The second group of terms on the right side of (\ref{case_13} )are at most $2(n_0+2)^2$ unmatched terms of degrees at least~$d+1$, with $\alpha \in \mathcal{I}^o$, $\nn^o(\alpha)=s_0-2$ and $n \leq n_0+2$. The third group of terms are at most $4(n_0+2)^3$ unmatched terms up to a factor $\frac{1}{\sqrt{N}}$.

We continue to expand each of the resulting terms of degree $d$ on the right side of (\ref{case_13}) using (\ref{case_2a}). In general, in the $s$-th step of the iteration, we have
\begin{align}\label{step_s}
\E[Q_{d_{s-1}}^o]=&\sum_{\substack{Q^o_{d_{s}} \in \mathcal{Q}^o_{d_{s}}\\ d_{s}=d_{s-1}}} \E [Q^o_{d_{s}}]+ \sum_{\substack{Q^o_{d'_{s}} \in \mathcal{Q}^o_{d'_{s}}\\ d'_{s} \geq d_{s-1}+1}} \E [Q^o_{d'_{s}}]+\frac{1}{\sqrt{N}} \sum_{\substack{ Q^o_{d_{s}''} \in \mathcal{Q}^o_{d_{s}''}\\ d_{s}'' \geq d_{s-1}}} \E [Q^o_{d_{s}''}]+O_{\prec}\big(N^{-1}\big),
\end{align}
uniformly in $t\in \R^+$ and $z \in S$, where the first group of terms are $s_0-2s+1$ unmatched terms of degree~$d$, with $\alpha \in \mathcal{I}^o$, $\nn^o(\alpha)=s_0-2s$ and $n=n_0$. The second group of terms are at most $2(n_0+2)^2$ unmatched terms of the form in (\ref{form}) of higher degrees, with $\alpha \in \mathcal{I}^o$, $\nn^o(\alpha)=s_0-2s+2$ and $n \leq n_0+2$. The third group of terms are at most $4(n_0+2)^3$ unmatched terms of the form in (\ref{form}) up to a factor $\frac{1}{\sqrt{N}}$.

We stop the iterations at the step $s=\frac{s_0+1}{2}$. Then the term $Q_{d_{s-1}}^o$ on the left side of (\ref{step_s}) has $\nn^o(\alpha)$ being reduced to one and there will be no more terms of degree $d$ on the right side of (\ref{step_s}). Then we end up with finitely many unmatched terms of degrees at least $d+1$ generated in the iterations, plus all the unmatched terms with a factor $\frac{1}{\sqrt{N}}$, collected from the third order terms in the cumulant expansions. To sum up, for any unmatched $Q_d^o \in \mathcal{Q}_d^o$, we have obtained the following expansion,
\begin{align}\label{third_unmatch2}
\E[Q_d^o(t,z)]=\sum_{\substack{ Q^o_{d'} \in \mathcal{Q}^o_{d'} \\ d' \geq d+1}}  \E[Q^o_{d'}(t,z)]+\frac{1}{\sqrt{N}} \sum_{\substack{Q^o_{d''} \in \mathcal{Q}^o_{d''}\\ d'' \geq d}} \E [Q^o_{d''}(t,z)]+O_{\prec}\big(N^{-1}\big)\,,
\end{align}
uniformly in $t\in \R^+$ and $z \in S$, where the number of unmatched terms on the right side above is bounded by $(Cn_0)^{cn_0}$, and the number of the Green function entries in  each term is bounded by $Cn_0$ for some numerical constants~$C,c>0$. The error term $O_{\prec}(N^{-1})$ stems from truncating the cumulant expansions and from the diagonal cases, \ie when two distinct summation indices coincide in the sums over $\mathcal{I}$.

We finally iterate the expansions in (\ref{third_unmatch2}) for $D-d$ times. Then the unmatched terms from the first group on the right side of (\ref{third_unmatch2}) have degrees increased to at least $D$, and the unmatched terms with a factor $\frac{1}{\sqrt{N}}$ from the second group have degrees increased to at least $D-1$. In addition, the number of these terms generated in the iterations is bounded by $\big((C^Dn_0)^{c^D n_0} \big)^D$, and the number of the Green function entries in each term is bounded by $C^{D}n_0$. We hence obtain from the local law in~(\ref{G}) by power counting that
\begin{align}\label{third_unmatch3}
|\E[Q_d^o(t,z)]|=O_{\prec}\Big(\Psi^D+\frac{\Psi^{D-1}}{\sqrt{N}}+\frac{1}{N}\Big)=O_{\prec}(\Psi^D+N^{-1})\,,
\end{align}
uniformly in $t\in \R^+$ and $z \in S$. We have finished the proof of Proposition \ref{unmatch_lemma}.

.
\end{proof}

\subsection{Proof of Proposition~\ref{lemma_expand_type}}
Next, we give the proof of Proposition \ref{lemma_expand_type} using the expansion mechanism in Subsections \ref{sec:expand_1}-\ref{sec:expand_2} and Proposition \ref{unmatch_lemma}.

\begin{proof}[Proof of Proposition \ref{lemma_expand_type}]
We consider a type-$\alpha b$ term of degree $d$, denoted by $P_d^{\alpha b} \in \mathcal{P}_d^{\alpha b}$, from Definition~\ref{def_type_AB}, with fixed initial integers $n_0$ defined in (\ref{n_number}) and $d^o_0$ given in (\ref{d_o}), \ie
\begin{align}
\frac{1}{N^{2+\#\mathcal{I}} }  \sum_{\alpha=N+1}^{N+M} \sum_{b=1}^N \sum_{\mathcal{I}} c_{\mathcal{I},\alpha,b}\Big(\prod^{n^o}_{i=1} G_{x_i y_i}\Big) \Big(\ud{G}-\wt m\Big)^{n^{g}} \Big( \ud{\G}+\frac{1}{1+\wt m}\Big)^{n^{\mathfrak{g}}}, 
\end{align}
with $\nn(\alpha)=\nn(b)=4$ given in (\ref{nu_number}), $\nn(\mathfrak{v}_j)=2$ for any $\mathfrak{v}_j \in \mathcal{I}$, and $x_i \neq y_i~(1\leq i\leq n^{o})$ unless $x_i=y_i=\alpha$ or $x_i=y_i=b$. 

We first expand the type-$\alpha b$ term $P_d^{\alpha b}$ with $\nn(\alpha)=4$ by the index $\alpha$ into finitely many type-$b$ terms with $\nn(\alpha)=0$, using the expansions in (\ref{case_1a}) and (\ref{case_2a}) twice. We split the discussion in two cases.

\textbf{Case 1:} If there is a diagonal factor $G_{\alpha \alpha}$ in the product of Green function entries, we apply expansions on $G_{\alpha \alpha}$ as discussed in Case $1\alpha$ in Subsection \ref{sec:expand_1} and obtain from (\ref{case_1a}) that
\begin{align}\label{remove_a}
\E[P_d^{\alpha b}(t,z)]=&\frac{1}{N^{2+\#\mathcal{I}} }\sum_{\mathcal{I},\alpha,b}  c_{\mathcal{I},\alpha,b} \E\Big[G_{\alpha \alpha}\Big(\prod^{n^o}_{i=2} G_{x_i y_i}\Big) \Big(\ud{G}-\wt m\Big)^{n^g} \Big( \ud{\G}+\frac{1}{1+\wt m}\Big)^{n^{\mathfrak{g}}}\Big] \nonumber\\
=&\E \Big[  P_d^{\alpha b} \big( G_{\alpha \alpha} \rightarrow -\frac{1}{1+\wt m}\big)\Big]+\sum_{ \substack{Q_{d'} \in \mathcal{Q}_{d'}  \\ d' \geq d+1 }} \E[Q_{d'}]+\frac{1}{\sqrt{N}} \sum_{\substack{ Q^o_{d'} \in \mathcal{Q}^o_{d'}\\ d' \geq d}} \E [Q^o_{d'}]+O_{\prec}(N^{-1}),
\end{align}
uniformly in $t\in \R^+$ and $z \in S$ given in (\ref{ddd}), where the leading term of degree $d$ with $\nn(\alpha)=2$ is obtained from the original type-$\alpha b$ term $P_d^{\alpha b}$ by replacing the factor $G_{\alpha \alpha}$ with the deterministic function $-\frac{1}{1+\wt m}$. In addition, the third group of terms on the right side of (\ref{remove_a}) contains at most $4(n_0+1)^2$ unmatched terms up to a factor $\frac{1}{\sqrt{N}}$, and thus can be bounded by $O_{\prec}(N^{-1/2}\Psi^D+N^{-3/2})$ using Proposition \ref{unmatch_lemma}. 

 We next discuss in detail the second group of terms on the right side of (\ref{remove_a}), which are collected from the second order terms in the cumulant expansions. They are at most $2n_0$ terms of the form in~(\ref{form_ab}) of degrees at least $d+1$, with $n=n_0+1$ and $d^o \geq d_0^o$. Moreover, we have $\nn(\alpha)=4$, $\nn(b)=4$ and $\nn(\mathfrak{v}_j)=2$ for any original indices $\mathfrak{v}_j \in \mathcal{I}$. The number of appearances of any fresh index created in the expansions is exactly two. It is not hard to check using the differentiation rule (\ref{dH}) that there are no diagonal Green function entries, except possible factors of $G_{\alpha \alpha}$ and $G_{bb}$, in the product of each resulting term. Thus these terms are also type-$\alpha b$ terms from Definition \ref{def_type_AB}.

\textbf{Case 2:} Else if there is no diagonal factor $G_{\alpha\alpha}$ in the product of Green function entries, then we have $\nn(\alpha)=\nn^o(\alpha)=4$; see (\ref{nu_number_0}). We may assume that $G_{x_1y_1}=G_{\alpha y_1}$ with $y_1 \neq \alpha$. Applying expansions on~$G_{\alpha y_1}$ as outlined in Case $2\alpha$ in Subsection \ref{sec:expand_2}, we obtain from (\ref{case_2a}) that
\begin{align}\label{remove_a_1}
\E[P_d^{\alpha b}]=&\frac{1}{N^{2+\#\mathcal{I}} }\sum_{\mathcal{I},\alpha,b}  c_{\mathcal{I},\alpha,b} \E\Big[G_{\alpha y_1} \Big(\prod^{n^o}_{i=2} G_{x_i y_i}\Big) \Big(\ud{G}-\wt m\Big)^{n^g} \Big( \ud{\G}+\frac{1}{1+\wt m}\Big)^{n^{\mathfrak{g}}}\Big] \nonumber\\
=&\frac{1}{(1+\wt m)^2}\sum_{\substack{2\leq i \leq n^o\\x_i=\alpha,y_i \neq \alpha}} \E \Big[ Q_{d} \big( x_1 ,x_i=\alpha \rightarrow k \big)\Big]+\frac{1}{(1+\wt m)^2}\sum_{\substack{2\leq i \leq n^o\\ x_i\neq \alpha,y_i = \alpha}} \E \Big[ Q_{d} \big( x_1, y_i =\alpha\rightarrow k \big)\Big]\nonumber\\
&+\sum_{ \substack{Q_{d''} \in \mathcal{Q}_{d''}  \\ d'' \geq d+1} } \E[Q_{d''}]+\frac{1}{\sqrt{N}} \sum_{\substack{Q^o_{d'} \in \mathcal{Q}^o_{d'}\\ d' \geq d}} \E [Q^o_{d'}]+O_{\prec}(N^{-1}),
\end{align}
uniformly in $t\in \R^+$ and $z \in S$, where the second line contains three leading terms of degree $d$ with $\nn(\alpha)=2$, which are obtained from the original type-$\alpha b$ term $P_d^{\alpha b}$ by replacing one pair of the index $\alpha$ from two off-diagonal Green function entries with a fresh index $k$, up to a factor $\frac{1}{(1+\wt m)^2}$. Similarly as in~(\ref{remove_a}), the first group of terms on the last line above contains at most $8(n_0+2)$ type-$\alpha b$ terms of the form in (\ref{form_ab}) of degrees at least $d+1$, with $n\leq n_0+2$ and $d^o \geq d_0^o$. The second group of unmatched terms with a factor~$\frac{1}{\sqrt{N}}$ on the last line can be bounded by $O_{\prec}(N^{-1/2}\Psi^D+N^{-3/2})$ using Proposition~\ref{unmatch_lemma}.

Combining (\ref{remove_a_1}) with (\ref{remove_a}), we have eliminated one pair of the index $\alpha$ for the leading terms of degree $d$ in one expansion step. We further apply the above arguments  on these leading terms with $\nn(\alpha)=2$ and eliminate another pair of the index $\alpha$ to obtain type-$b$ terms. Therefore, any type-$\alpha b$ term can be expanded using the index $\alpha$ as
 \begin{align}\label{alpha}
\E[P_d^{\alpha b}]=&\sum_{ \substack{P^{b}_{d'} \in \mathcal{P}^{b}_{d'}  \\ d'=d }} \E [  P_{d'}^{b}]+\sum_{ \substack{P^{\alpha b}_{d''} \in \mathcal{P}^{\alpha b}_{d''}  \\ d'' \geq d+1 }} \E[P^{\alpha b}_{d''}]+O_{\prec}(N^{-1}+N^{-1/2}\Psi^{D}),
\end{align}
uniformly in $t\in \R^+$ and $z \in S$, where the first group of terms with degree $d$ are at most three type-$b$ terms of the form in (\ref{form_ab}), with $\nn(\alpha)=0$, $\nn(b)=4$, and the second group of terms are at most $32(n_0+2)$ type-$\alpha b$ terms of the form in (\ref{form_ab}) of degrees at least $d+1$, with $n \leq n_0+2$ and $d^o \geq d^o$. 

Iterating the expansion procedure in (\ref{alpha}) $D-d$ times, the degrees of the resulting type-$\alpha b$ terms on the right side of (\ref{alpha}) are increased to at least $D$. Using the local law in (\ref{G}), we expand an arbitrary type-$\alpha b$ term $P^{\alpha b}_d\in \mathcal{P}^{\alpha b}_d$ as a finite sum of type-0 terms of degrees at least~$d$, up to $O_{\prec}(\Psi^D+N^{-1})$, \ie
\begin{align}\label{alpha_final}
\E [ P_d^{\alpha b}(t,z)]=\sum_{d \leq d' < D} \sum_{P^b_{d'} \in \mathcal{P}^b{d'}} \E [ P^b_{d'}(t,z)]+O_{\prec}\big(N^{-1}+\Psi^D\big)\,,
\end{align}
uniformly in $t\in \R^+$ and $z \in S$, where the sum above contains at most $(32(n_0+2D))^D$ type-$b$ terms of the form in (\ref{form}) of degrees at least $d$, with $n\leq n_0+2D$ and $d^o \geq d_0^o$.

Next in the second step, we further eliminate two pairs of the index $b$ and expand the resulting type-$b$ terms on the right side of (\ref{alpha_final}) into type-0 terms. For any type-$b$ terms of the form in (\ref{form_ab}) with fixed~$n_0$ given in (\ref{n_number}) and $d_0^o$ defined in (\ref{d_o}), using the expansions in (\ref{case_1b}), (\ref{case_2b}) and similar arguments for the index $\alpha$ as above, we obtain the  analogue of (\ref{alpha}), \ie
 \begin{align}\label{b}
\E[P_d^{ b}]=&\sum_{ \substack{P_{d'} \in \mathcal{P}_{d'}  \\ d'=d }} \E [  P_{d'}]+\sum_{ \substack{P^{ b}_{d''} \in \mathcal{P}^{ b}_{d''}  \\ d'' \geq d+1 }} \E[P^{ b}_{d''}]+O_{\prec}(N^{-1}+N^{-1/2}\Psi^{D}),
\end{align}
uniformly in $t\in \R^+$ and $z \in S$, where the first group of terms of degree $d$ are at most three type-$0$ terms of the form in (\ref{form_ab}) with $\nn(\alpha)=\nn(b)=0$, and the second group of terms are at most $32(n_0+2)$ type-$b$ terms of the form in (\ref{form_ab}) of degrees at least $d+1$, with $n \leq n_0+2$ and $d^o \geq d_0^o$. Iterating the expansion procedure in (\ref{b}) $D-d$ times and we obtain the analogue of (\ref{alpha_final}),
\begin{align}\label{b_final}
\E [ P_d^{ b}(t,z)]=\sum_{d \leq d' < D} \sum_{P_{d'} \in \mathcal{P}{d'}} \E [ P_{d'}(t,z)]+O_{\prec}\big(N^{-1}+\Psi^D\big)\,,
\end{align}
uniformly in $t\in \R^+$ and $z \in S$, where the sum above contains at most $(32(n_0+2D))^D$ type-$0$ terms of the form (\ref{form}) of degrees at least $d$, with $n\leq n_0+2D$ and $d^o \geq d_0^o$. 

To sum up, combining (\ref{alpha_final}) with (\ref{b_final}), we conclude the proof of Proposition \ref{lemma_expand_type}.
\end{proof}

\section{Proof of Proposition \ref{lemma_trace_sample}}\label{sec:type0}

We start with the following lemma by considering the white Wishart ensemble $W^*W$, whose proof is postponed to the end of this section. 
\begin{lemma}\label{lemma_trace} Consider the real white Wishart ensemble $W^*W$ with $W$ be a Gaussian matrix satisfying (\ref{sample_condition_1}) and (\ref{condition}). For any fixed small $\epsilon>0$ and fixed $C_1,C_2>0$, recall the domain $S_{\mathrm{edge}}\equiv S_{\mathrm{edge}}(\epsilon,C_1,C_2)$ defined in~\eqref{S_edge}. Then there exists a constant $C>0$, depending on $C_1$ and $C_2$, such that the normalized trace of the resolvent of $W^*W$ satisfies
\begin{equation}\label{traceimg}
\E^{\mathrm{W}} [ \Im m (z)] \leq CN^{-1/3}, 
\end{equation}
uniformly for all $z \in S_{\mathrm{edge}}$, for sufficiently large $N \geq N_0(C_1,C_2, \epsilon)$. 

Furthermore, consider any type-0 term $P_d \in \mathcal{P}_d$ ($d\geq 2$) of the form in (\ref{form}) with the number of off-diagonal Green function factors given in (\ref{d_o}) satisfying $d^o \geq 2$. Then, for any $\tau>0$, we have
\begin{align}\label{tracek}
|\E^{\mathrm{W}} [P_d(z)]| \leq N^{-1/3-\epsilon+\tau},
\end{align}
uniformly for all $z\in S_{\mathrm{edge}}$, for sufficiently large $N \geq N_0(C_1,C_2, \epsilon, \tau)$.
\end{lemma}
We remark that in this section, we will always choose $z \in S_{\mathrm{edge}} \subset S$, unlike to Section \ref{sec:expand} when $z \in S$.

Armed with Proposition \ref{lemma_trace}, we are now ready to prove Proposition \ref{lemma_trace_sample}  using recursive comparisons based on Propositions \ref{unmatch_lemma} and \ref{lemma_expand_type}.

\begin{proof}[Proof of Proposition \ref{lemma_trace_sample}]
Consider a type-0 term $P_d \in \mathcal{P}_d$ of the form in (\ref{form_ab}) of degree $d \geq 2$ with fixed $n$ given in (\ref{n_number}) and $d^o \geq 2$ given in (\ref{d_o}).  If $d \geq D$ for some large $D$ depending on $\epsilon$, we can prove~(\ref{tracek_sample}) using the initial estimate in (\ref{localaw_0}) by power counting. Else if $d$ is smaller, we estimate $\E[P_d]$ by iteratively using comparisons and the corresponding estimates for the white Wishart ensemble in~(\ref{tracek}).

We start the iteration by denoting a given type-0 term $P_d$ ($d \geq 2$) of the form in (\ref{form_ab}) as $P_d\equiv P^{(1)}_{d_1}$, where the superscript $(1)$ and degree $d\equiv d_1$ will be used to indicate the initial step. We hence consider a term of the form
\begin{align}\label{P_d_s_1}
P^{(1)}_{d_1} \equiv P^{(1)}_{d_1} (t,z): \qquad  \frac{1}{N^{\# \mathcal I_1+2}} \sum_{\mathcal{I}_1,\alpha_1,b_1} c_{\alpha_1,b_1,\mathcal{I}_1} \Big(\prod^{n_1^o}_{i=1} G_{x_i y_i}\Big) \Big(\ud{G}-\wt m\Big)^{n_1^{g}} \Big( \ud{\G}+\frac{1}{1+\wt m}\Big)^{n_1^{\mathfrak g}}\,,
\end{align}
where $\mathcal{I}_1$ is a set of free summation indices, the coefficients $\{c_{\alpha_1,b_1,\mathcal{I}_1}\}$ are uniformly bounded, each $x_i, y_i$ represents a free summation index in $\mathcal{I}_1$, with $x_i \neq y_i$ and $\nn(\mathfrak{v}_j)=2$ for any $\mathfrak{v}_j \in \mathcal{I}_1$. We also set
$$n_1:=n_1^o+n_1^{g}+n_1^{\mathfrak g}; \qquad d^o_1:=\#\{ 1 \leq i\leq n_1^o: x_i \neq y_i\}$$ 
as in (\ref{n_number}) and (\ref{d_o}). In particular, we have $d_1 \geq d^o_1 \geq 2$.

We next compute the derivative of $\E[P^{(1)}_{d_1}]$ with respect to time under the interpolation flow in~\eqref{dyson_flow}, similarly as in (\ref{step00}) and (\ref{step0}). We then find that
\begin{align}\label{late_0}
\frac{\dd}{\dd t}\E[P^{(1)}_{d_1}]=&\sum_{\alpha_2=N+1}^{N+M} \sum_{b_2=1}^N \E\Big[ \dot{h}_{\alpha_2 b_2} \frac{\partial P^{(1)}_{d_1}}{\partial h_{\alpha_2 b_2}} \Big]=-\frac{1}{2} \sum_{\alpha_2,b_2}  \sum_{p+1\geq 3}^{4} \frac{1}{p!} \frac{s^{(p+1)}_{ \alpha_2 b_2}}{N^{\frac{p+1}{2}}} \E \Big[ \frac{ \partial^{p+1}  P^{(1)}_{d_1} }{\partial h^{p+1}_{\alpha_2 b_2} } \Big]+O_{\prec}(N^{-1/2}),
\end{align} 
where $\alpha_2,b_2$ are fresh summation indices (the subscript $2$ indicates the iteration step), and $\{s^{(p+1)}_{ \alpha_2 b_2}\}$ are the uniformly bounded cumulants of rescaled matrix entries given in (\ref{t_cumu}). 

From (\ref{dH}), all the third order terms for $p+1=3$ on the right side of (\ref{late_0}) are of the form in (\ref{form}) up to a factor $\sqrt{N}$. Since the fresh indices have $\nn(\alpha_2)=\nn(b_2)=3$, they are unmatched from Definition~\ref{unmatch_def}. Using Proposition~\ref{unmatch_lemma}, these term are hence bounded by $O_{\prec}(N^{-1/2}+\sqrt{N} \Psi^D)$. 

Next, it suffices to estimate the fourth order terms for $p+1=4$ on the right side of (\ref{late_0}). Since $\alpha_2$ and $b_2$ are freshly added indices, using the differentiation rule in (\ref{dH}), the fourth order terms are at most $(n_1+4)^4$ matched terms of degrees at least $d_1+1$, with $\nn(\alpha_2)=\nn(b_2)=4$, $\nn(\alpha_1)=\nn(b_1)=0$, and $\nn(\mathfrak{v}_j)=2$ for any $\mathfrak{v}_j \in \mathcal{I}_1$. The number of Green function entries in the product of each term is $n_1+4$ and the number of off-diagonal Green function entries in each term is at least $d^o_1+1$.

To be precise, these fourth order terms can be written out in the abstract form
\begin{align}\label{type_ab_p_s}
  \frac{1}{N^{\# \mathcal I_2+4}} \sum_{\mathcal{I}_2,a_1,b_1, a_2,b_2} c_{a_1,b_1,a_2,b_2,\mathcal{I}_2} \Big(\prod^{n_2^o}_{i=1} G_{x_i y_i}\Big) \Big(\ud{G}-\wt m\Big)^{n_2^{g}} \Big( \ud{\G}+\frac{1}{1+\wt m}\Big)^{n_2^{\mathfrak g}}\,,
\end{align}
where $\mathcal{I}_2$ is a set of free summation indices, the coefficients $\{c_{a_1,b_1,a_2,b_2,\mathcal{I}_2}\}$ are uniformly bounded, $x_i, y_i$ stand for a free summation index $\alpha_2$, $b_2$ or some element in $\mathcal{I}_2$, with $\nn(\alpha_2)=\nn(b_2)=4$, $\nn(\mathfrak{v}_j)=2$ for any $\mathfrak{v}_j \in \mathcal{I}_2$. Moreover, $x_i \neq y_i$ for any $ 1 \leq i\leq n^o_2$ unless $x_i=y_i=\alpha_2$ or $x_i=y_i=b_2$. The number of Green function entries and the number of off-diagonal Green function entries in a term in (\ref{type_ab_p_s}) are denoted by $n_2$ and $d^o_2$, respectively, as in (\ref{n_number}) and (\ref{d_o}). The degree of such a term, denoted by $d_2$, is defined as in (\ref{degree_0}).

Recall the definition of type-$\alpha b$, type-$b$ and type-0 terms from Definition \ref{def_type_AB} for the form in (\ref{form_ab}). The definitions can be adapted naturally with respect to the fresh indices $\alpha_2$ and $b_2$ for the form given in (\ref{type_ab_p_s}). In particular, the corresponding type-0 term is defined to be of the form in (\ref{type_ab_p_s}), with $x_i \neq y_i~(1\leq i\leq n_2^o)$ and each $x_i$, $y_i$ represents an element in $\mathcal{I}_2$ with $\nn(\mathfrak{v}_j)=2$ for any $\mathfrak{v}_j \in \mathcal{I}_2$. We remark that though $\nn(\alpha_1)=\nn(b_1)=\nn(\alpha_2)=\nn(b_2)=0$, we keep these summation indices to record the iteration step and emphasize the inheritance from the form in (\ref{P_d_s_1}). We use $\mathcal{P}^{(2)}_{d_2}$ to denote the collection of the corresponding type-0 terms of the form in~(\ref{type_ab_p_s}). 

Thus the fourth order terms for $p+1=4$ on the right side of (\ref{late_0}) consist of at most $4(n_1+1)^4$ type-$\alpha b$ terms of the form in (\ref{type_ab_p_s}) of degrees $d_2 \geq d_1 +1$, with $n_2= n_1+4 $ and $d^o_2 \geq d^o_1+1$. Using Proposition~\ref{lemma_expand_type}, we expand each of these type-$\alpha b$ terms as a sum of finitely many type-0 terms of the form in (\ref{type_ab_p_s}) of degrees at least $d_1+1$, and rewrite (\ref{late_0}) in short as
\begin{align}\label{late3}
\frac{\dd}{\dd t}\E[P^{(1)}_{d_1}]= \sum_{\substack{ P^{(2)}_{d_{2}} \in \mathcal{P}^{(2)}_{d_{2}}\\d_1+1\leq d_{2} < D}} \E[ P^{(2)}_{d_{2}}]+O_{\prec}\big(N^{-1/2}+\sqrt{N}\Psi^D\big)\,,
\end{align}
uniformly in $t \geq 0$ and $z \in S_{\mathrm{edge}}$, where the sum above is over finitely many (depending on $D$ and $n_1$) type-0 terms, with $n_2 \leq n_1+4D$, $d_2 \geq d_1 +1$ and $d^o_2 \geq d^o_1+1$. We now choose $D$ sufficiently large, depending on $\epsilon$ in (\ref{S_edge}), such that $\sqrt{N}\Psi^D \leq N^{-1/2}$.

Integrating (\ref{late3}) over $ [t', T]$ for any $0\leq t' \leq T=8 \log N$ as in (\ref{integrate_imn}) and applying the estimates in~(\ref{localaw_0}) on the resulting type-0 terms on the right side, we find that
\begin{align}\label{previous_d_1}
\big|\E[P^{(1)}_{d_{1}}(T,z)]-\E[P^{(1)}_{d_{1}}(t',z)]\big|=&O_{\prec}\big(\log N (\Psi^{d_1+1}+N^{-1/2})\big)\,,
\end{align}
uniformly in $t'\in [0,T]$ and $z \in S_{\mathrm{edge}}$.

For $T=8 \log N$, the Green function $G(T,z)$ is close to the corresponding quantity for the White Wishart ensemble, denoted by $G^{W}$; see (\ref{approxxxx}). In combination with the estimate in (\ref{tracek}), we find that
$$\big|\E[P^{(1)}_{d_1}(T,z)]\big|=O_{\prec}(N^{-1/3-\epsilon}), \qquad d^o_1 \geq 2,$$
uniformly for any $z \in S_{\mathrm{edge}}$. Therefore, we have from (\ref{previous_d_1}) that
\begin{align}\label{previous_d_2}
\big|\E[P^{(1)}_{d_{1}}(t',z)]\big|=&O_{\prec}\big( \log N \Psi^{d_1+1}+N^{-1/3-\epsilon}\big)\,,
\end{align}
uniformly in $t'\in [0,T]$ and $z \in S_{\mathrm{edge}}$. In this way, we improve the estimate on the size of $\E[P^{(1)}_{d_{1}}]$ by one order in power counting, compared with the initial estimate in (\ref{localaw_0}). We can further apply the above arguments on the resulting type-0 terms on the right side of (\ref{late3}) to get a finer estimate. We next discuss the detailed iteration process.

For any $s \geq 1$, we define a type-0 term in the $s$-th iteration step, denoted by $P^{(s)}_{d_s}$, of the form
\begin{align}\label{P_ds}
 \frac{1}{N^{\# \mathcal I_s+2s}} \sum_{\mathcal{I}_s, \alpha_1,b_1, \ldots, \alpha_s, b_s} c_{\alpha_1,b_1, \ldots, \alpha_s,b_s,\mathcal{I}_s}  \Big(\prod^{n_s^o}_{i=1} G_{x_i y_i}\Big) \Big(\ud{G}-\wt m\Big)^{n_s^{g}} \Big( \ud{\G}+\frac{1}{1+\wt m}\Big)^{n_s^{\mathfrak g}},
 \end{align}
where $\mathcal{I}_s$ is a set of free summation indices, the coefficients $\{c_{\alpha_1,b_1, \ldots, \alpha_s,b_s,\mathcal{I}_s}\}$ are uniformly bounded, each $x_i, y_i~(1 \leq i\leq n_2^o)$ represents a free summation index in $\mathcal{I}_s$, with $x_i \neq y_i$, and $\nn(\mathfrak{v}_j)=2$ for any $\mathfrak{v}_j \in \mathcal{I}_s$. We remark that, though $\nn(\alpha_l)=\nn(b_l)=0$ for any $1\leq l\leq s$, we keep $\alpha_l$ and $b_l$ in the notation to emphasize the inheritance from $\alpha_l$ and $b_l$ in the $l$-th iteration step. The degree of a term in (\ref{P_ds}), $d_s$, is defined as in (\ref{degree_0}) and $n_s$, $d^o_s$ are defined as in (\ref{n_number}) and (\ref{d_o}). We use $\mathcal{P}^{(s)}_{d_s}$ to denote the collection of the corresponding type-0 terms of the form in (\ref{P_ds}) in the $s$-th iteration step. 

Next, we take the time derivative of $\E[P^{(s)}_{d_s}]$ for any $P^{(s)}_{d_s} \in \mathcal{P}^{(s)}_{d_s}~(s\geq 1)$ with $d_s \geq d^o_s \geq 2$, similarly as in (\ref{late_0}). That is, 
\begin{align}\label{rhs2}
\frac{\dd}{\dd t}\E[P^{(s)}_{d_s}]=&\sum_{\alpha_{s+1}=N+1}^{N+M} \sum_{b_{s+1}=1}^N \E\Big[ \dot{h}_{\alpha_{s+1} b_{s+1}} \frac{\partial P^{(s)}_{d_s}}{\partial h_{\alpha_{s+1} b_{s+1}}} \Big]\nonumber\\
=&-\frac{1}{2} \sum_{\alpha_{s+1},b_{s+1}}  \sum_{p+1\geq 3}^{4} \frac{1}{p!} \frac{s^{(p+1)}_{ \alpha_{s+1} b_{s+1}}}{N^{\frac{p+1}{2}}} \E \Big[ \frac{ \partial^{p+1}  P^{(s)}_{d_s} }{\partial h^{p+1}_{\alpha_{s+1} b_{s+1}} } \Big]+O_{\prec}(N^{-1/2}),
\end{align}
where $\alpha_{s+1}$ and $b_{s+1}$ are fresh summation indices in the $s$-th iteration step, and $\{s^{(p+1)}_{ \alpha_{s+1} b_{s+1}}\}$ are time-dependent cumulants of matrix entries given in (\ref{t_cumu}). Using Proposition \ref{unmatch_lemma} and Proposition \ref{lemma_expand_type}, we expand the resulting terms on the right side of (\ref{rhs2}), similarly to (\ref{late3}),
\begin{align}\label{late22}
\frac{\dd}{\dd t}\E[P^{(s)}_{d_s}(t,z)]= \sum_{\substack{ P^{(s+1)}_{d_{s+1}} \in \mathcal{P}^{(s+1)}_{d_{s+1}}\\d_s+1\leq d_{s+1} < D}} \E[ P^{(s+1)}_{d_{s+1}}(t,z)]+O_{\prec}\big(N^{-1/2}+\sqrt{N}\Psi^D\big)\,,
\end{align}
uniformly in $t \geq 0$ and $z \in S_{\mathrm{edge}}$, where the sum above are over finitely many (depending on $D$,~$n_s$) type-0 terms of the form in (\ref{P_ds}) in the $(s+1)$-th step, with $n_{s+1} \leq n_s +4D$, $d_{s+1} \geq d_s+1$ and $d^o_{s+1} \geq d^o_s+1$. 

Next, we integrate both sides of (\ref{late22}) over $ [t', T]$ for any $0\leq t' \leq T=8 \log N$. Indeed, from the approximation in (\ref{approxxxx}), the local law in (\ref{G}) together with the estimate in (\ref{tracek}) for the white Wishart ensemble, we have
\begin{align}\label{approx_gue}
\big|\E[P^{(s)}_{d_s}(T,z)]\big|=O_{\prec}(N^{-1/3-\epsilon})\,, \qquad d_s \geq d^o_s \geq 2,
\end{align}
uniformly in $z\in S_{\mathrm{edge}}$. We hence obtain from~(\ref{late22}) and~(\ref{approx_gue}) that, for any $P^{(s)}_{d_s} \in \mathcal{P}^{(s)}_{d_s}$ with $d_s \geq d^o_s \geq 2$,
\begin{align}\label{temp_eq}
\E[P^{(s)}_{d_s}(t',z)]=&\sum_{ \substack{ P^{(s+1)}_{d_{s+1}} \in \mathcal{P}^{(s+1)}_{d_{s+1}} \\d_s+1 \leq d_{s+1} <D}}  \int_{t'}^{T} \E[ P^{(s+1)}_{d_{s+1}}(t,z)] \dd t +O_{\prec}\Big(\log N (N^{-1/2}+\Psi^D)+N^{-1/3-\epsilon}\Big),
\end{align}
uniformly in $t'\in [0,T]$ and $z \in S_{\mathrm{edge}}$.

Now, we are ready to iterate using ~\eqref{temp_eq}. In the first step, we start by $P_{d_1}^{(1)}(t,z)$ in (\ref{P_d_s_1}) with $d_1 \geq d^o_1 \geq 2$ and have the estimate in~(\ref{temp_eq}) for $s=1$. The number of the terms $P^{(2)}_{d_2} \in \mathcal{P}^{(2)}_{d_2}$ with $d_2 \geq d_1+1$ on the right side of (\ref{temp_eq}) is finite and depends on $n_1$ and $D$.  Then we further estimate each of these type-0 terms $P^{(2)}_{d_2}$ using (\ref{temp_eq}) for $s=2$ as the second step. The resulting type-0 terms $P^{(3)}_{d_3} \in \mathcal{P}^{(2)}_{d_3}$ with $d_3 \geq d_2+1 \geq d_1+2$ will be estimated again using (\ref{temp_eq}) for $s=3$ as the third step. Since in each step of using (\ref{temp_eq}), the degrees of the resulting type-0 terms $P^{(s+1)}_{d_{s+1}} \in \mathcal{P}^{(s+1)}_{d_{s+1}} $ on the right side are increased by at least one, we have $d_{s+1} \geq d_1+s$. We hence stop the iteration process at step $s=s_0:=D-d_1$. For any $P^{(s_0)}_{d_{s_0}} \in \mathcal{P}^{(s_0)}_{d_{s_0}}$ with $d_{s_0} \geq D-1$, the resulting terms $P^{(s_0+1)}_{d_{s_0+1}} \in \mathcal{P}^{(s_0+1)}_{d_{s_0+1}}$ on the right side of (\ref{temp_eq}) have degrees $d_{s_0+1} \geq D$. The number of these terms is finite and depends on $D,n_{s_0}$. Using the local law in (\ref{G}), all these terms can be bounded by $O_{\prec}(\Psi^D+N^{-1})$. Therefore, for any $P^{(s_0)}_{d_{s_0}} \in \mathcal{P}^{(s_0)}_{d_{s_0}}$ with $d_{s_0} \geq D-1$ and $d^o_{s_0} \geq 2$, 
$$\big|\E[P^{(s_0)}_{d_{s_0}}(t',z)]\big|=O_{\prec}\Big(\log N (N^{-1/2}+\Psi^D)+N^{-1/3-\epsilon}\Big)\,.$$

We hence apply the above estimate back to the previous step, \ie (\ref{temp_eq}) for $s=s_0-1$. We then obtain a similar estimate for any $P^{(s_0-1)}_{d_{s_0-1}} \in \mathcal{P}^{(s_0-1)}_{d_{s_0-1}}$ with $d_{s_0-1} \geq D-2$ and $d^o_{s_0-1}\geq 2$, \ie
$$\big|\E[P^{(s_0-1)}_{d_{s_0-1}}(t',z)]\big|=O_{\prec}\Big( (\log N)^2 (N^{-1/2}+\Psi^D)+N^{-1/3-\epsilon} \log N\Big)\,,$$
uniformly in $t'\in [0,T]$ and $z \in S_{\mathrm{edge}}$.

Repeating the above process until $s=1$, we hence obtain the estimate of $P^{(1)}_{d_{1}}$ with $d_1 \geq d^o_1 \geq 2$,
$$\big|\E[P^{(1)}_{d_{1}}(t',z)]\big|=O_{\prec}\Big((N^{-1/3-\epsilon}+ \Psi^D) (\log N)^{D}\Big),$$
uniformly in $t' \in[0,T]$ and $z \in S_{\mathrm{edge}}$. By choosing $D$ sufficiently large depending on $\epsilon$, we prove~(\ref{tracek_sample}) for~$t' \in [0,T]$. If $t' \geq T$, a similar estimate can be obtained by using (\ref{approxxxx}) and (\ref{approx_gue}). We have hence finished the proof of Proposition \ref{lemma_trace_sample}.
\end{proof}

To end this section, we give the proof of Lemma \ref{lemma_trace} using the uniform convergence of the correlation kernel of the white Wishart ensemble at the upper edge, along with the the local Marchenko--Pastur law and the eigenvalue rigidity.

\begin{proof}[Proof of Lemma \ref{lemma_trace}]
Let $(\lambda_j)_j$ be the eigenvalues of the white Wishart matrix $W^*W$. Using the spectral decomposition, for any $z=E+\ii \eta \in S_{\mathrm{edge}}$, we write
\begin{align}\label{type_0_temp}
\E^{W}[\Im m(z)]=\frac{\eta}{N} \sum_{j=1}^N \frac{1}{|\lambda_j-E|^2+\eta^2}=\frac{\eta}{N} \E \Big[ \Big(  \sum_{|\lambda_j-E| \leq l} + \sum_{|\lambda_j-E| \geq l} \Big) \frac{1}{|\lambda_j-E|^2+\eta^2}\Big]:=\mathrm{I}+\mathrm{II},
\end{align}
where we choose $l=N^{-2/3}$ for later purposes. We first estimate the term I using the convergence of the correlation kernel at the edge in Proposition \ref{kernel_diff}. From (\ref{pfaffian}), the one-point correlation function of the eigenvalue process $(\nu_j)$ of the rescaled white Wishart matrix $N W^*W$ is given by 
\begin{align}
p_{N,1}^{(1)}(\nu)=\frac{1}{2} \Tr S_{N,1}(\nu,\nu)=K_{N,1}(\nu,\nu).
\end{align}
Then the one-point correlation function of the eigenvalue process $(\lambda_j)$ of $W^*W$ is given by 
$$\wt p_{N,1}^{(1)}(\lambda)=NK_{N,1}(N \lambda,N \lambda),$$
and the term I in (\ref{type_0_temp}) can be written as
$$\mathrm{I}=\frac{\eta}{N} \int_{|\lambda-E| \leq l}  \frac{  \wt p_{N,1}^{(1)}(\lambda)}{(\lambda-E)^2+\eta^2} \dd \lambda=\eta \int_{E-l}^{E+l}  \frac{ K_{N,1}(N \lambda,N \lambda)}{(\lambda-E)^2+\eta^2} \dd \lambda.$$
Changing variable $N \lambda=\wt \mu_N+\wt \sigma_N x$ as in (\ref{normalize}), we write the term I using (\ref{edge}) as
\begin{align}\label{lemma_temp_1}
\mathrm{I}=\frac{N \eta}{\wt \sigma_N^2} \int_{\frac{N}{\wt \sigma_N} \big(E-\frac{\wt \mu_N}{N}-l\big)}^{\frac{N}{\wt \sigma_N} \big(E-\frac{\wt \mu_N}{N}+l\big)}  \frac{ K^{\mathrm{edge}}_{N,1}(x,x)}{\big(x-\frac{N}{\wt \sigma_N} (E-\frac{\wt \mu_N}{N})\big)^2+\big(\frac{N}{\wt \sigma_N} \eta\big)^2} \dd x.
\end{align}
From (\ref{some_parameter}) and (\ref{condition}), we find that $\wt \mu_N \sim N$ and $\wt \sigma_N \sim N^{1/3}$. To be precise, we have
\begin{align}\label{some_constant}
\frac{\wt \mu_N}{N}=\big(1+\sqrt{\varrho}\big)^2+O(N^{-1})=E_++O(N^{-1}); \qquad \frac{N}{\wt \sigma_N}=\frac{N^{2/3}}{(1+\sqrt{\varrho}) \big(1+\frac{1}{\sqrt{\varrho}}\big)^{\frac{1}{3}}}+O(N^{-1/3}).
\end{align}

Since $E-E_+ \geq -C_1 N^{-2/3}$ for $E+\ii \eta \in S_{\mathrm{edge}}(\epsilon,C_1,C_2)$ and $l = N^{-2/3}$, it is not hard to check that the lower integration bound in (\ref{lemma_temp_1}) satisfies $\frac{N}{\wt \sigma_N} (E-\frac{ \wt \mu_N}{N}-l )\geq L_0$ for some constant $L_0 \in \R$. By Proposition~\ref{kernel_diff} and Lemma \ref{lemma_airy_bound}, there hence exists a constant $C \equiv C(L_0)$ such that 
$$\big|K^{\mathrm{edge}}_{N,1}(x,x) \big|  \leq C, \qquad \mbox{ for any } x \in [L_0,\infty). $$
Then using the fact that $K^{\mathrm{edge}}_{N,1}(x,x) \geq 0$, we have
\begin{align}\label{part_1}
\mathrm{I} \leq&\frac{C N \eta}{\wt \sigma_N^2} \int_{L_0}^{\infty}  \frac{1}{\big(x-\frac{N}{\wt \sigma_N} (E-\frac{ \wt \mu_N}{N})\big)^2+\big(\frac{N}{ \wt \sigma_N} \eta\big)^2} \dd x \leq \frac{C N \eta}{\wt \sigma_N^2} \int_{-\infty}^{\infty}  \frac{1}{x^2+\big(\frac{N}{\wt \sigma_N} \eta\big)^2} \dd x \leq \frac{C\pi}{\wt \sigma_N} =O(N^{-\frac{1}{3}}).
\end{align}

Next we estimate the second term II on the right side of (\ref{type_0_temp}), which can be written as
\begin{align}
\mathrm{II}= \frac{\eta}{N} \sum_{\lambda_j \geq E+l} \frac{1}{|\lambda_j-E|^2+\eta^2}+\frac{\eta}{N} \sum_{\lambda_j \leq E-l} \frac{1}{|\lambda_j-E|^2+\eta^2}=:\mathrm{II}_1+\mathrm{II}_2,
\end{align}
with $l=N^{-2/3}$. Using the eigenvalue rigidity in (\ref{rigidity1}) and that $-C_1 N^{-2/3} \leq E-E_+ \leq C_2 N^{-2/3+\epsilon}$, we find that for any small $\tau>0$ and large $\Gamma>0$, $\#\{ j: \lambda_j \geq E+l\} \leq N^{\tau}$ holds with probability bigger than $1-N^{-\Gamma}$. Then the first part $\mathrm{II}_1$ can be bounded by
\begin{align}\label{II_2}
\mathrm{II}_1=\frac{\eta}{N} \sum_{\lambda_j \geq E+l} \frac{1}{|\lambda_j-E|^2+\eta^2} \leq \frac{\eta N^{\tau}}{Nl^2},
\end{align}
with probability bigger than $1-N^{-\Gamma}$. The second part $\mathrm{II}_2$ can be written as
\begin{align}\label{lemma_temp}
\mathrm{II}_2=\frac{\eta}{N} \sum_{\lambda_j \leq E-l} \frac{1}{|\lambda_j-E|^2+\eta^2} \leq  \frac{\eta}{N}  \sum_{k=0}^{\infty}  \frac{1}{3^{2k} l^2 } \mathcal{N}_k,
\end{align}
where 
$$\mathcal{N}_k:=\#\{j: \lambda_j \in \mathcal{I}_k\}, \qquad \mathcal{I}_k:=(E-3^{k+1} l , E-3^{k} l ], \qquad k \geq 0.$$
We next estimate $\mathcal{N}_k$ using the local law in (\ref{G_average}). Note that
\begin{align}\label{append_temp2}
\Im m_N \big(E-2 \cdot 3^{k} l+\ii (3^{k} l) \big)=\frac{1}{N} \sum_{i=1}^N \frac{3^k l }{ |\lambda_i-(E-2 \cdot 3^k l)|^2+(3^k l)^2} \geq \frac{1}{N} \frac{\mathcal{N}_k}{2 \cdot 3^k l}\,.
\end{align}
Using the local law in (\ref{G_average}) to control the left side above, we have, for any small $\tau>0$ and large $\Gamma>0$,
\begin{align}
\Im m_N \big (E-2 \cdot 3^{k} l+\ii (3^{k} l) \big) \leq & \Im \wt m(E-2 \cdot 3^{k} l+\ii 3^{k} l)+\frac{N^{\tau}}{ N 3^{k} l}\nonumber\\
\leq& C \sqrt{3^k l+|E-2 \cdot 3^{k} l-E_+|}+\frac{N^{\tau}}{ N 3^{k} l} \leq C \Big( \sqrt{3^k l}+\frac{N^{\tau}}{ N 3^{k} l}+N^{-1/3+\epsilon}\Big)\,,\nonumber
\end{align}
with probability bigger than $1-N^{-\Gamma}$. We then obtain from (\ref{append_temp2}) that
\begin{align}
\mathcal{N}_{k} \leq C\Big( (3^k l)^{3/2}N+N^{\tau} +3^k l N^{2/3+\epsilon}\Big)\,,\nonumber
\end{align}
with probability bigger than $1-N^{-\Gamma}$. Combining with(\ref{lemma_temp}), we find
\begin{align}
\mathrm{II}_2 \leq & \frac{C \eta}{N}  \sum_{k=0}^{\infty}  \frac{1}{3^{2k} l^2 } \Big( (3^k l)^{3/2}N+N^{\tau} +3^k l N^{2/3+\epsilon}\Big) \leq C' \Big(\frac{\eta}{\sqrt{l}}+\frac{N^{\tau }\eta}{N l^2} +\frac{N^{\epsilon} \eta}{N^{1/3}l}\Big),\nonumber
\end{align}
with probability bigger than $1-N^{-\Gamma}$. Therefore, combining with (\ref{II_2}) and choosing $\tau<\epsilon$, we have
\begin{align}\label{part_2}
\mathrm{II} \leq C' \Big(\frac{\eta}{\sqrt{l}}+\frac{N^{ \tau }\eta}{N l^2} +\frac{N^{\epsilon} \eta}{N^{1/3}l}\Big)=O(N^{-\frac{1}{3}}),
\end{align}
with probability bigger than $1-N^{-\Gamma}$, using that $l=N^{-2/3}$, $N^{-1+\epsilon} \leq \eta \leq N^{-2/3-\epsilon}$. Hence, combining (\ref{part_1}) and (\ref{part_2}), we have proved~(\ref{traceimg}).

Finally, we consider any type-0 term $P_d(z) \in \mathcal{P}_d(z)$ of the form in (\ref{form}) for the white Wishart ensemble, where the number of off-diagonal Green function entries in the product $\prod_{i=1}^{n^o}G_{x_i y_i}$, denoted by $d^o$ as in (\ref{d_o}), is at least two. Using the local law in (\ref{G}) and that the coefficients $\{c_{\mathcal{I}}\}$ are uniformly bounded, we find
\begin{align}\label{rain_1}
|P_d(z) | \prec (\Psi)^{n^g+n^{\mathfrak{g}}} \frac{1}{N^{\#\mathcal{I}} } \sum_{\mathcal{I}} \prod^{n^o}_{i=1} \big| G_{x_i y_i}(z)\big|,
\end{align}
where each $x_i$ and $y_i$ represents some summation index in $\mathcal{I}=\{\mathfrak{v}_j\}$, with $\nn(\mathfrak{v}_j)=2$ and $x_i \neq y_i$. In particular, we have $n^o=d^o \geq 2$. From the local law in (\ref{G}) and Young's inequality, we have
\begin{align}\label{rain_2}
|P_d(z) | \prec \Psi^{d-2} \frac{1}{N^2} \sum_{\mathfrak{i,j}=1}^{N+M} |G_{\mathfrak{ij}}(z)|^2, \qquad d=d^o+n^g+n^{\mathfrak{g}}.
\end{align}
Using the generalized Ward identities in (\ref{ward}), in combination with (\ref{hermitization_inverse}), (\ref{gamma_N}) and that $|m| \prec 1$, $\|X^*X\|_2 \prec 1$, we conclude that, for any $z \in S_{\mathrm{edge}}$,
\begin{align}\label{ward_id}
\frac{1}{N^2} \sum_{\mathfrak{i}, \mathfrak{j}=1}^{N+M} |G_{\mathfrak{ij}}(z)|^2 \prec \frac{\Im m(z)}{N \eta}+\frac{1}{N}.
\end{align}
Then for any $d\geq 2$ and $z \in S_{\mathrm{edge}}$, we have 
$$|P_d(z)| \prec \frac{\Im m_N(z)}{N \eta} +\frac{1}{N}.$$
After taking the expectation, using statement (3) in Lemma \ref{dominant_prop} in combination with (\ref{traceimg}), we have
$$|\E[P_d(z)] |\prec \frac{1}{N \eta }\E[\Im m_N(z)]+\frac{1}{N}=O_{\prec}(N^{-1/3-\epsilon}).$$
We hence complete the proof of Lemma \ref{lemma_trace}.
\end{proof}

\section{Green function comparison theorem: Proof of Theorem \ref{green_comparison}}\label{sec:Fneq1}
In this section, we extend the ideas in the proof of Proposition \ref{GCT_mn} with $F(x)=x$ to prove the Green function comparison theorem, Theorem \ref{green_comparison}, for smooth and uniformly bounded general functions $F$.

\begin{proof}[Proof of Theorem \ref{green_comparison}]

For fixed small $\epsilon>0$ and fixed $C_1,C_2>0$, let 
\begin{align}\label{z_set}
N^{-1+\epsilon} \leq \eta \leq N^{-2/3-\epsilon}, \qquad -C_1N^{-2/3} \leq \kappa_1< \kappa_2 \leq C_2 N^{-2/3+\epsilon}.
\end{align}
For notational simplicity, we introduce for short
\begin{equation}\label{X}
\mathcal{X} \equiv \mathcal{X}(t):=N\int_{\kappa_1}^{\kappa_2} \Im m_N(t,E_++x+\ii \eta) \dd x, \qquad t \in \R^+,
\end{equation}
with $E_+$ given in (\ref{endpoint}). Let $F$ be a smooth function with uniformly bounded derivatives. Differentiating $\E[F(\X(t))]$ with respected to $t$, we obtain (\cf (\ref{step00}))
\begin{align}\label{diff_fx}
\frac{\dd }{\dd t}\E [F(\X)]=&\E \Big[F'(\X)\Im \int_{\kappa_1}^{\kappa_2} \sum_{v=1}^{N} \frac{\dd G_{vv}(t,E_++x+\ii \eta)}{\dd t}  \dd x\Big]
=\E \Big[F'(\X)\Im \int_{\kappa_1}^{\kappa_2}  \Big( \sum_{v, \alpha,b} -2 \dot{h}_{\alpha b} G_{\alpha v} G_{v b}  \Big) \dd x\Big]\nonumber\\
=&\E \Big[F'(\X)\Im \int_{\kappa_1}^{\kappa_2}  \Big( \sum_{ \alpha,b} -2 \dot{h}_{\alpha b}  \frac{\dd }{\dd x} (G_{\alpha b}) \Big) \dd x\Big]=-2\E \Big[ F'(\X) \sum_{\alpha,b}  \dot{h}_{ \alpha b}  \Delta \Im G_{\alpha b}\Big],
\end{align}
where the third step follows from the definition of the Green function $G$ in (\ref{blockG}), and in the last step we abbreviate, for any function $P\,:\,\R^+ \times \C \setminus \R \longrightarrow\C$,
\begin{align}\label{dim}
\Delta \Im P \equiv (\Delta \Im P) (t,z_1,z_2):=\Im P(t,z_2)-\Im P(t,z_1),
\end{align}
with $t \in \R^+$, and
$$z_1:=E_++\kappa_1+\ii \eta \in \C \setminus \R,\qquad z_2:=E_++\kappa_2+\ii \eta \in \C \setminus \R.$$ 
In particular, we have $z_1, z_2 \in S_{\mathrm{edge}}$ given in (\ref{S_edge}), in view of (\ref{z_set}).

Using the cumulant expansion formulas in Lemma \ref{cumulant} as in (\ref{step0}), we obtain from (\ref{diff_fx}) that
\begin{align}\label{derivative_F}
\frac{\dd }{\dd t}\E [F(\X(t))]=&\sum_{\alpha,b}   \sum_{p+1=3}^{4} \frac{1}{p!} \frac{s^{(p+1)}_{\alpha b}(t)}{N^{\frac{p+1}{2}}}\E \Big[\frac{ \partial^{p} F'(\X) \Delta \Im G_{\alpha b}}{\partial h^{p}_{\alpha b}} \Big]+O_{\prec}\big(\frac{1}{\sqrt{N}}\big),
\end{align}
where $\{s^{(p+1)}_{\alpha b}(t)\}$ are the $(p+1)$-th cumulants of the normalized entries $\sqrt{N}h_{\alpha b}(t)$ given in (\ref{t_cumu}).

Recall the differentiation rules in (\ref{dH}), and note that we further have
 \begin{align}\label{int_1}
\frac{ \partial F'(\X) }{ \partial h_{\alpha b} }=-2 F''(\X) \sum_{v=1}^{N} \Im  \Big( \intkappa  G_{b v} G_{v \alpha}(E_++x+\ii \eta) \dd x \Big)=-2 F''(\X) \Delta \Im G_{ b \alpha}\,,
\end{align}
which follows from the block structure of $G$ in (\ref{blockG}), similarly as in (\ref{diff_fx}).

To estimate the resulting terms on the right side of (\ref{derivative_F}), we introduce the analogous form of (\ref{form}) in (\ref{form_F}) below, which are functions of $t \in \R^+$ and $z_1,z_2 \in S_{\mathrm{edge}}$. We remark that the variables $z_1,z_2$ are from the definition of $\Dim$ in (\ref{dim}). Instead of the centered factors $\ud G-\wt m$ and $\ud \G +\frac{1}{1+\wt m}$ included in (\ref{form}), we replace them with $\ud G+\frac{1}{1+\sqrt{\varrho}} $ and $\ud \G+\frac{1+\sqrt{\varrho}}{\sqrt{\varrho}}$, considering that $z_1,z_2 \in S_{\mathrm{edge}}$. To be precise, from (\ref{mplawgamma-1}) and Lemma \ref{sample_m}, it is not hard to check that, for any $z \in S_{\mathrm{edge}}$,
\begin{align}\label{edge_number}
\Big|\wt m(z)+\frac{1}{1+\sqrt{\varrho}}\Big| =O(N^{-\frac{1}{3}+\frac{\epsilon}{2}}),  \qquad \Big|\frac{1}{1+\wt m(z)}-\frac{1+\sqrt{\varrho}}{\sqrt{\varrho}}\Big| =O(N^{-\frac{1}{3}+\frac{\epsilon}{2}}),
\end{align}
and thus from the local law in (\ref{G}), for any $t \in \R^+$ and $z \in S_{\mathrm{edge}}$,
\begin{align}\label{GG}
\Big| \ud G(t,z)+\frac{1}{1+\sqrt{\varrho}} \Big| \prec \Psi, \qquad \Big| \ud \G(t,z)+\frac{1+\sqrt{\varrho}}{\sqrt{\varrho}} \Big| \prec \Psi.
\end{align}
Then, for fixed $i_0 \in \N$, $n_i^{o},n_i^{g},n_i^{\mathfrak{g}} \in \N$ for any $1\leq i\leq i_0$ and a free summation index set $\mathcal{I}:=\{\mathfrak{v}_j\}_{j=1}^{m},~m \in \N$ (which may include $\alpha$ and $b$ in (\ref{derivative_F})), the analogous form of (\ref{form}) for general functions $F$ is given by
\begin{align}\label{form_F}
 &\frac{1}{N^{\# \mathcal{I}}} \sum_{\mathcal{I}} c^{(0)}_{\mathcal{I}}(t) F^{(i_0)}(\X) \prod_{i=1}^{i_0} \Dim \Big( c^{(i)}_{\mathcal{I}}(t) \big( \prod_{l=1}^{n^o_i} G_{x^{(i)}_{l} y^{(i)}_l}\big) \big(\ud{G}+\frac{1}{1+\sqrt{\varrho}}\big)^{n_i^g} \big( \ud{\G}+\frac{1+\sqrt{\varrho}}{\sqrt{\varrho}}\big)^{n_i^{\mathfrak{g}}} \Big),
\end{align}
where the coefficients $\{c^{(i)}_{\mathcal{I}}(t) \in \R\}_{i=0}^{i_0}$ are uniformly bounded deterministic functions of $t \in \R^{+}$, $F^{(i_0)}$ is the $i_0$-th derivative of the smooth function $F$, $\Dim: \R^+ \times (\C \setminus \R)^2 \rightarrow \C$ is defined in (\ref{dim}), and each row index $x^{(i)}_{l}$ and column index $y^{(i)}_{l}$ of the Green function entries represents some element in $\mathcal{I}$.

As the analogue of (\ref{n_number}), we denote the total number of Green function entries (including the centered diagonal Green function factor $\ud G+\frac{1}{1+\sqrt{\varrho}}$ and $\ud \G+\frac{1+\sqrt{\varrho}}{\sqrt{\varrho}}$) in the products of the form in (\ref{form_F}) by 
\begin{align}\label{n_number_F}
n:=\sum_{i=1}^{i_0} (n_i^o+n_i^{g}+n_i^{\mathfrak{g}}),
\end{align}
and denote the number of off-diagonal Green function entries in the products (\cf (\ref{d_o})) as
 \begin{align}\label{d_o_F}
 d^o:=\sum_{i=1}^{i_0} \# \big\{ 1 \leq l \leq n^{o}_i : x^{(i)}_{l} \neq y^{(i)}_{l}\big\}.
 \end{align}
Analogous to (\ref{degree_0}), we further define the degree of such a term, denoted by $d$, to be
 \begin{align}\label{degree_F}
 d:=d^o+\sum_{i=1}^{i_0} n_i^g+\sum_{i=1}^{i_0}n_i^{\mathfrak{g}}.
 \end{align}

We use $\wt {\mathcal{Q}}_d \equiv \wt {\mathcal{Q}}_d(t,z_1,z_2)$ to denote the collection of the averaged products of the Green function entries of the form in (\ref{form_F}) of degree $d$. From the definition of $\Dim$ in (\ref{dim}), the local laws in (\ref{G}) and (\ref{GG}), together with the assumption that $F$ has bounded derivatives, we have, for any $\wt Q_d \equiv \wt Q_d(t,z_1,z_2)  \in \mathcal{\wt Q}_d$, 
$$|\wt Q_d(t,z_1,z_2)| \prec \Psi^d+N^{-1}\,,$$
uniformly in $t\in \R^+$, and $z_1,z_2 \in S_{\mathrm{edge}}$. In the following, we often omit the parameters $t,z_1,z_2$ for notational simplicity.

Definition \ref{unmatch_def} for unmatched terms of the form in (\ref{form}) can be adapted naturally to the general form in~(\ref{form_F}) by setting the number of appearances of any summation index $\mathfrak{v}_j \in \mathcal{I}$ as the row or column index of a Green function entry in the products to be (\cf (\ref{nu_number}))
\begin{align}\label{nu_number_F}
\nn(\mathfrak{v}_j):=&\sum_{i=1}^{i_0}\#\{ 1 \leq l \leq n^{o}_i: x^{(i)}_l=\mathfrak{v}_j \}+\sum_{i=1}^{i_0}\#\{ 1 \leq l \leq n^{o}_i: y^{(i)}_l=\mathfrak{v}_j \}.
\end{align}
We also denote by $\wt{\mathcal{Q}}_d^{o} \subset \wt{\mathcal{Q}}_d$ the collection of unmatched terms in the form (\ref{form_F}) of degree $d$.

Using the differential rules in (\ref{dH}) and (\ref{int_1}), the third order terms for $p+1=3$ on the right side of (\ref{derivative_F}) are unmatched terms of the form in (\ref{form_F}) up to a factor $\sqrt{N}$, with $\mathcal{I}=\{\alpha,b\}$, $\nn(\alpha)=\nn(b)=3$ and $n=3$. The arguments given in the proof of Proposition \ref{unmatch_lemma} still apply to the general form in (\ref{form_F}), using that $\{h_{ij}\}$ commute with $\Dim$ defined in (\ref{dim}), the differentiation rules in (\ref{dH}) and (\ref{int_1}), and the assumption that the function $F$ has bounded derivatives. Therefore, as in Proposition \ref{unmatch_lemma}, given any unmatched term $\wt Q^o_d  \in \wt{\mathcal{Q}}^o_d$ in~(\ref{form_F}) with fixed $n \in \N$, for any fixed integer $D\ge d$, we have
\begin{align}\label{wt_Q_d}
|\E[\wt Q^o_d(t,z_1,z_2)]|=O_{\prec}\big(N^{-1}+\Psi^D\big)\,,
\end{align}
uniformly in $t\in \R^+$ and $z_1,z_2 \in S_{\mathrm{edge}}$. Hence, the third order terms on the right side of (\ref{derivative_F}) can be bounded by $O_{\prec}\big(\sqrt{N}(\frac{1}{N}+ \Psi^D)\big)$ and 
\begin{align}\label{derivative_F_1}
\frac{\dd }{\dd t}\E [F(\X(t))]=&\frac{1}{6 N^2} \sum_{\alpha,b} s^{(4)}_{\alpha b}(t) \E \Big[\frac{ \partial^{3} F'(\X) \Delta \Im G_{\alpha b}}{\partial h^{3}_{\alpha b}} \Big]+O_{\prec}\big(\frac{1}{\sqrt{N}}+\sqrt{N} \Psi^D\big).
\end{align}

It then suffices to estimate the remaining fourth order terms. From (\ref{dH}) and (\ref{int_1}), the resulting terms on the right side above are matched terms of the form in (\ref{form_F}), with $\mathcal{I}=\{\alpha,b\}$, $\nn(\alpha)=\nn(b)=4$ and $n=4$. The definitions for type-$\alpha b$, type-$b$ and type-0 terms in Definition \ref{def_type_AB} can be extended naturally to the general form in (\ref{form_F}). We consider the following form with the two indices $\alpha$ and~$b$ singled out,
\begin{align}\label{form_F_ab}
\frac{1}{N^{2+\# \mathcal{I}}} \sum_{\alpha,b,\mathcal{I}} c^{(0)}_{\alpha,b,\mathcal{I}}(t)  F^{(i_0)}(\X) \prod_{i=1}^{i_0} \Dim \Big( c^{(i)}_{\alpha,b,\mathcal{I}}(t) \big( \prod_{l=1}^{n^o_i} G_{x^{(i)}_{l} y^{(i)}_l}\big) \big(\ud{G}+\frac{1}{1+\sqrt{\varrho}}\big)^{n_i^g} \big( \ud{\G}+\frac{1+\sqrt{\varrho}}{\sqrt{\varrho}}\big)^{n_i^{\mathfrak{g}}} \Big),
\end{align}
where $\{c^{(i)}_{\alpha,b,\mathcal{I}}(t) \in \R \}_{i=0}^{i_0}$ are uniformly bounded deterministic functions of $t\in \R^+$, each row index $x^{(i)}_{l}$ and column index $y^{(i)}_{l}$ of the Green function entries represents either the indices $\alpha$, $b$, or some element in the free summation index set~$\mathcal{I}:=\{\mathfrak{v}_j\}_{j=1}^{m}$. Recall from (\ref{nu_number_F}) the definition of $\nn(\mathfrak{v}_j)$ for any $\mathfrak{v}_j \in \mathcal{I}$. We define $\nn(\alpha)$ and $\nn(b)$ similarly for the special indices $\alpha$ and $b$.

\begin{definition}

A term of degree $d$ in (\ref{form_F_ab}) is referred to as a {\it type-$\alpha b$ term} if
\begin{enumerate}
\item $\nn(\alpha)=\nn(b)=4$, and $\nn(\mathfrak{v}_j)=2$ for any $\mathfrak{v}_j \in \mathcal{I}$;
\item $x^{(i)}_{l}= y^{(i)}_{l}~(1 \leq i\leq i_0, ~1\leq l\leq n^o_i)$, then $x^{(i)}_{l}=y^{(i)}_{l}=\alpha$ or $x^{(i)}_{l}=y^{(i)}_{l}=b$.
\end{enumerate}
Such a term is denoted by $T_d^{\alpha b}$, and the collection of all such type-$\alpha b$ terms is denoted by $\mathcal{T}^{\alpha b}_d$.

Similarly, a term of degree $d$ in (\ref{form_F_ab}) is referred to as a {\it type-$b$ term} if
\begin{enumerate}[label=(\roman*)]
\item $\nn(\alpha)=0$, $\nn(b)=4$, and $\nn(\mathfrak{v}_j)=2$ for all $\mathfrak{v}_j \in \mathcal{I}$;
\item $x^{(i)}_{l} \neq y^{(i)}_{l}~(1 \leq i\leq i_0, ~1\leq l\leq n^o_i)$ unless $x^{(i)}_{l}=y^{(i)}_{l}=b$.
\end{enumerate}
Such a term is denoted by $T_d^{b}$, and the collection of all such type-$b$ terms is denoted by $\mathcal{T}^{b}_d$.

Finally, a term of degree $d$ in (\ref{form_F_ab}) is referred to as a {\it type-$0$ term} if 
\begin{enumerate}[label=\alph*)]
\item $\nn(\alpha)=\nn(b)=0$, $\nn(\mathfrak{v}_j)=2$ for all $\mathfrak{v}_j \in \mathcal{I}$;
\item 
$x^{(i)}_{l} \neq y^{(i)}_{l}$ for any $1 \leq i\leq i_0, ~1\leq l\leq n^o_i$.
\end{enumerate}
Such a term is denoted by $T_d$, and the collection of all such type-$0$ terms is denoted by $\mathcal{T}_d$.
\end{definition}

Returning to the right side of (\ref{derivative_F_1}), the resulting terms by (\ref{dH}) and (\ref{int_1}) are finitely many type-$\alpha b$ terms of the form in (\ref{form_F_ab}) of degrees $d\geq 0$ with $\nn(\alpha)=\nn(b)=4$ and $n = 4$. The proof of Proposition~\ref{lemma_expand_type} can be extended to the general form in (\ref{form_F}) with modifications. We will present below the modified expansion mechanism to eliminate one pair of the index $\alpha$ in the general form (\ref{form_F}), as discussed in Case~$1\alpha$ in Subsection \ref{sec:expand} for the special form  (\ref{form}), and the remaining cases can be modified similarly.

For example, a term of degree zero on the right side of (\ref{derivative_F_1}) is given by
$$T^{\alpha b}_0:=\frac{1}{N^2} \sum_{\alpha, b} s_{\alpha b}^{(4)}(t) F'(\X) \Dim \big( (G_{\alpha \alpha})^2 (G_{bb})^2 \big) .$$
We then eliminate one pair of the index $\alpha$ by expanding one diagonal Green function entry $G_{\alpha \alpha}$ using the identity (\ref{id_3}) and the cumulant expansions in Lemma \ref{cumulant}. Using the first identity in (\ref{id_3}) on $G_{\alpha \alpha}$ and the definition of $\Dim$ in (\ref{dim}), we have
\begin{align}\label{F_temp_1}
\E[T^{\alpha b}_0]=&\frac{1}{N^2} \sum_{\alpha, b,k} s_{\alpha b}^{(4)}(t) \E \Big[ H_{\alpha k} F'(\X) \Dim \big( G_{k \alpha }G_{\alpha \alpha} (G_{bb})^2 \big) \Big]-\frac{1}{N^2} \sum_{\alpha, b} s_{\alpha b}^{(4)}(t)  \E \Big[ F'(\X) \Dim \big( G_{\alpha \alpha} (G_{bb})^2 \big) \Big]\nonumber\\
=&\frac{1}{N^3} \sum_{\alpha, b,k} s_{\alpha b}^{(4)}(t) \E \Big[ \frac{\partial F'(\X) \Dim \big( G_{k \alpha }G_{\alpha \alpha} (G_{bb})^2 \big) }{\partial h_{\alpha k}}\Big]-\frac{1}{N^2} \sum_{\alpha, b} s_{\alpha b}^{(4)}(t)  \E \Big[ F'(\X) \Dim \big( G_{\alpha \alpha} (G_{bb})^2 \big) \Big]\nonumber\\
&+\frac{1}{2 \sqrt{N} N^3} \sum_{\alpha, b,k} s_{\alpha b}^{(4)}(t) s_{\alpha k}^{(3)}(t)    \E \Big[ \frac{\partial^2 F'(\X) \Dim \big( G_{k \alpha }G_{\alpha \alpha} (G_{bb})^2 \big) }{\partial h^2_{\alpha k}}\Big]+O_{\prec}(N^{-1}),
\end{align}
where the error $O_{\prec}(N^{-1})$ stems from truncating the cumulant expansions at the third order. Using the differentiation rules in (\ref{dH}) and (\ref{int_1}), the third order terms with $\{s_{\alpha k}^{(3)}(t) \}$ above are also of the form in (\ref{form_F_ab}) with $\nn(k)=3$ up to a factor $\frac{1}{\sqrt{N}}$, and thus are unmatched terms which can be bounded by $O_{\prec}(N^{-3/2}+N^{-1/2}\Psi^D)$. 

We next look at the second order terms in the cumulant expansions, \ie the first group of terms on the right side of (\ref{F_temp_1}). From the differentiation rules in (\ref{dH}), (\ref{int_1}) and that $\frac{\partial}{\partial h_{\alpha k}}$ commutes with $\Dim$, these terms are also type $\alpha b$ terms of the form in (\ref{form_F_ab}), with $\mathcal{I}=\{\alpha, b, k\}$, $\nn(k)=2$, and $\nn(\alpha)=\nn(b)=4$. Since the index $k$ is fresh, the degrees of these type-$\alpha b$ terms are increased to at least two, except the leading term of degree zero which comes from letting $\frac{\partial}{\partial h_{\alpha k}}$ act on $G_{k\alpha}$. The collection of these terms of degrees at least two is denoted as
\begin{align}\label{F_second}
\sum_{d \geq 2} \E [T^{\alpha b}_d].
\end{align}
The leading term of degree zero corresponding to $\frac{\partial}{\partial h_{\alpha k}}$ acting on $G_{k\alpha}$ is given by
\begin{align}\label{F_temp_2}
-\frac{1}{N^3} \sum_{\alpha, b,k} s_{\alpha b}^{(4)}& \E \Big[ F'(\X) \Dim \big( G_{k k}(G_{\alpha \alpha})^2 (G_{bb})^2 \big) \Big]=\frac{1}{1+\sqrt{\varrho}}\frac{1}{N^2} \sum_{\alpha, b} s_{\alpha b}^{(4)} \E \Big[ F'(\X) \Dim \big( (G_{\alpha \alpha})^2 (G_{bb})^2 \big) \Big]\nonumber\\
 &-\frac{1}{N^2} \sum_{\alpha, b} s_{\alpha b}^{(4)} \E \Big[ F'(\X) \Dim \Big((G_{\alpha \alpha})^2 (G_{bb})^2\big( \ud{G} +\frac{1}{1+\sqrt{\varrho}}\big) \Big) \Big].
\end{align}
 The first term on the right side of (\ref{F_temp_2}) can be absorbed into the left side of (\ref{F_temp_1}) by considering $\frac{\sqrt{\varrho}}{1+\sqrt{\varrho}} \E[T^{\alpha b}_0]$. The second term gains an additional centered factor $\ud G+\frac{1}{1+\sqrt{\varrho}}$ and thus its degree is increased to one.

Therefore, after moving the leading term in (\ref{F_temp_2}) to the left side of (\ref{F_temp_1}) and dividing both sides by the deterministic real number $\frac{\sqrt{\varrho}}{1+\sqrt{\varrho}} \sim 1$ (see (\ref{condition})), together with the shorthand notation in  (\ref{F_second}), we obtain that
\begin{align}\label{temp_4_F}
\E[ T^{\alpha b}_0]=&-\frac{1+\sqrt{\varrho}}{\sqrt{\varrho}}\frac{1}{N^2} \sum_{\alpha, b} s_{\alpha b}^{(4)}(t)  \E \Big[ F'(\X) \Dim \big( G_{\alpha \alpha} (G_{bb})^2 \big) \Big]\nonumber\\
&-\frac{1+\sqrt{\varrho}}{\sqrt{\varrho}}\frac{1}{N^2} \sum_{\alpha, b} s_{\alpha b}^{(4)}(t) \E \Big[ F'(\X) \Dim \Big((G_{\alpha \alpha})^2 (G_{bb})^2( \ud{G} +\frac{1}{1+\sqrt{\varrho}}) \Big) \Big]\nonumber\\
&+\frac{1+\sqrt{\varrho}}{\sqrt{\varrho}}\sum_{d \geq 2} \E[T^{\alpha b}_d]+O_{\prec}(N^{-1}),
\end{align}
where the first term is of degree zero obtained by replacing one factor $G_{\alpha \alpha}$ in the original term $ T^{\alpha b}_0$ with the deterministic real number $-\frac{1+\sqrt{\varrho}}{\sqrt{\varrho}}$, the second and third group of terms are also type-$\alpha b$ terms of the form in (\ref{form_F_ab}) whose degrees are increased to at least one. In this way, we have eliminated one pair of the index $\alpha$ for the leading term and obtained the analogue of (\ref{case_1a}) for the general form in (\ref{form_F}).

Using the arguments in the proof of Proposition \ref{lemma_expand_type} for the special form in (\ref{form}) and the expansion mechanism above to extend to the general form in (\ref{form_F}), we obtain the analogue of Proposition \ref{lemma_expand_type}. Hence, the fourth order terms on the right side of (\ref{derivative_F_1}) can be expanded into finitely many type-0 terms of the form in (\ref{form_F_ab}) of degrees $d \geq 0$. For any fixed $D \geq 1$, we write for short that
\begin{align}\label{derivative_F_10}
\frac{\dd }{\dd t}\E [F(\X(t))]=&\sum_{ \substack{T_{d} \in \mathcal{T}_{d} \\ 0 \leq d \leq D-1}} \E[T_d]+O_{\prec}\big(N^{-\frac{1}{2}}+\sqrt{N}\Psi^D\big),
\end{align}
where the summation above contains at most $(CD)^{cD}$ type-0 terms of the form in (\ref{form_F_ab}) and the number of the Green function entries in each term is at most $CD$ for some numerical constants $C,c>0$.

Hence it suffices to estimate the size of a type-0 term of the form in (\ref{form_F_ab}) using the estimate in (\ref{img_sample}). It is straightforward to check from (\ref{d_o_F}) that, given an arbitrary type-0 term in (\ref{form_F_ab}), denoted by $T_d$, we have $d^o=\sum_{i=1}^{i_0} n^o_i$ with $d^o=0$ or $d^o \geq 2$. If $d^o=0$, then $T_d$ reduces to
$$T_d=\frac{1}{N^{2+\# \mathcal{I}}} \sum_{\alpha,b,\mathcal{I}} c'_{\alpha,b,\mathcal{I}} F^{(i_0)}(\X) \prod_{i=1}^{i_0}  \Dim \Big( \big(\ud{G}+\frac{1}{1+\sqrt{\varrho}}\big)^{n_i^g} \big( \ud{\G}+\frac{1+\sqrt{\varrho}}{\sqrt{\varrho}} \big)^{n_i^{\mathfrak{g}}} \Big),$$
where $c'_{\alpha,b,\mathcal{I}}= \prod_{i=0}^{i_0} c^{(i)}_{\alpha,b,\mathcal{I}}(t)$ from the definition of $\Dim$ in (\ref{dim}) and the assumption that the coefficients $c^{(i)}_{\alpha,b,\mathcal{I}}(t)$ are real-valued functions of $t$. From the definition of $\ud \G$ in (\ref{ggu}), the relation in (\ref{gamma_N}) and the estimate in (\ref{img_sample}), we have 
\begin{align}
\E[\Im \ud{G}(t,z)]=O(N^{-1/3}); \qquad \E[\Im \ud{\G}(t,z)]=O(N^{-1/3}),
\end{align}
for any $t\in \R^+$ and $z \in S_{\mathrm{edge}}$. In combination with the local law in (\ref{GG}), the fact that $\{c'_{\alpha,b,\mathcal{I}}\}$ and $F^{(i_0)}$ are uniformly bounded and statement (3) in Lemma \ref{dominant_prop}, we have 
\begin{align}\label{d_11}
\big|\E [T_d(t,z_1,z_2)]\big|=O_{\prec}(N^{-\frac{1}{3}}).
\end{align}

Else, if $d^o \geq 2$, similarly to the estimates in (\ref{rain_1}), (\ref{rain_2}) and (\ref{ward_id}), using the definition of $\Dim$ in~(\ref{dim}), the local laws in (\ref{G}) and (\ref{GG}) and the fact that the coefficients $\{c^{(i)}_{\alpha,b,\mathcal{I}}(t)\}$ and the derivative $F^{(i_0)}$ are uniformly bounded, we have
\begin{align}\label{d_22}
\big| \E[T_d(t,z_1,z_2)]\big| \prec \frac{ \E[\Im m_N(t,z_1)]}{N \Im z_1}+\frac{\E[\Im m_N(t,z_2)]}{N \Im z_2}  +\frac{1}{N} =O_{\prec}( N^{-\frac{1}{3}-\epsilon}),
\end{align}
for any $t\in \R^+$ and $z_1, z_2 \in S_{\mathrm{edge}}$, where the last step follows from the estimate in (\ref{img_sample}).

Therefore, integrating both sides of (\ref{derivative_F_10}) up to time $T=8\log N$ and combining with (\ref{d_22}) and~(\ref{d_11}), we have
$$\big| \E [F(\X(T))]-\E [F(\X(0))]\big|=O_{\prec}(  N^{-\frac{1}{3}}\log N).$$
Combining with (\ref{approxxxx}) and that $F$ has bounded derivatives, we complete the proof of Theorem \ref{green_comparison}.
\end{proof}


\begin{thebibliography}{00}



	\bibitem{odd}  Adler, M., Forrester, P. J., Nagao, T. V., Van Moerbeke, P.: \emph{Classical skew orthogonal polynomials and random matrices}, Journal of Statistical Physics {\bf 99}(1), 141-170 (2000).
	
	\bibitem{erdos_gram} Alt, J., Erd\H{o}s, L., Kr\"uger, T.: \emph{Local law for random Gram matrices}, Electron. J. Probab.  {\bf 22}, 1-41 (2017).
	
	
	
          \bibitem{AEKS} Alt, J., Erd{\H o}s, L, Kr\"uger, T,  Schr\"oder, D.: \emph{Correlated random matrices: band rigidity and edge universality}, Ann.\ Probab.\ \textbf{48(2)}, 963-1001  (2020).		 


		\bibitem{AGZ}  Anderson, G.,  Guionnet, A., Zeitouni, O.: \emph{An introduction to random matrices}, Cambridge studies in advanced mathematics {\bf 118}, Cambridge University Press, Cambridge (2010).
		
			
        \bibitem{classical1}	Anderson, T.W.: \emph{An Introduction to Multivariate Statistical Analysis}(2nd ed.), Wiley, New York (1984).



		\bibitem{kendall tau} Bao, Z.G.: \emph{Tracy--Widom limit for Kendall's tau}, Ann.\ Stat.\ {\bf 47.6}, 3504-3532  (2019). 

		
		 \bibitem{bao} Bao, Z.G., Pan, G.M.,  Zhou, W.: \emph{Universality for the largest eigenvalue of sample covariance matrices with general population}, Ann.\ Stat.\ {\bf 43}(1), 382-421 (2015).
		 
		  \bibitem{bao2} Bao, Z.G., Hu, J., Pan, G.M.,  Zhou, W.: \emph{Canonical correlation coefficients of high-dimensional Gaussian vectors: Finite rank case}, Ann.\ Stat.\ {\bf 47}(1), 612-640 (2019).
		 
		 		 
		 
		 
		 
		 		\bibitem{Ben+Peche} Ben Arous, G., P\'ech\'e, S.: \emph{Universality of local eigenvalue statistics for some sample covariance matrices}, Comm. Pure Appl. Math. {\bf 58}, 1-42 (2005).
		 		
		 			\bibitem{signal1}	  Bianchi, P., Debbah, M., Maida, M. and Najim, J.: \emph{Performance of Statistical Tests for Single-Source Detection Using Random Matrix Theory,}  IEEE Transactions on Information Theory {\bf 57}(4), 2400-2419 (2011).
		
		        \bibitem{Alex+Erdos+Knowles+Yau+Yin}  Bloemendal, A., Erd\H{o}s, L., Knowles A., Yau, H.T. and Yin, J.: \emph{Isotropic local laws for sample covariance and generalized Wigner matrices}, Electron. J. Probab. {\bf 19}(33), 53pp (2014).	 

		
		\bibitem{Bourgade extreme} Bourgade, P.: \emph{Extreme gaps between eigenvalues of Wigner matrices}, arXiv:1812.10376 (2018).
		
			        \bibitem{BEY edge universality} Bourgade, P., Erd{\H o}s, L., Yau, H.-T.: \emph{Edge universality of beta ensembles}, Commun.\ Math.\ Phys.\ \textbf{332.1} 261-353  (2014).

			       \bibitem{fixed generalized wigner} Bourgade, P., Erd\H{o}s, L., Yau, H.-T., Yin, J.: \emph{Fixed energy universality for generalized Wigner matrices}, Comm.\ Pure Appl.\ Math.\ {\bf 69(10)}, 1815-1881 (2016).
		    			
		\bibitem{BoutetdeMonvel} de Monvel, A. Boutet, Khorunzhy, A.: \emph{Asymptotic distribution of smoothed eigenvalue density. II. Wigner random matrices}, Random Operators and Stochastic Equations {\bf 7}(2), 149-168 (1999).
				
		\bibitem{deift} Deift, P., Gioev, D.: \emph{ Random matrix theory: invariant ensembles and universality}, Courant Lecture Notes in Mathematics. Vol. 18. American Mathematical Soc., 2009.
		
		 \bibitem{convergence_kernel} Deift, P., Gioev, D., Kriecherbauer, T., Vanlessen, M.: \emph{Universality for Orthogonal and Symplectic Laguerre-Type Ensembles}, Journal of Statistical Physics {\bf 129}(5-6) (2007).
		 
		 
		 
		 
		 \bibitem{Ding+Yang} Ding, X.,  Yang, F.: \emph{A necessary and sufficient condition for edge universality at the largest singular values of covariance matrices}, Ann.\ Appl.\ Probab.\ {\bf 28}(3), 1679-1738 (2018).
		 
		\bibitem{gram}  Ding, X., and Yang, F.: \emph{Tracy--Widom distribution for the edge eigenvalues of Gram type random matrices}, arXiv:2008.04166 (2020).
		 
		\bibitem{Edelman}  Edelman, A.: \emph{The distribution and moments of the smallest eigenvalue of a random matrix of Wishart type}, Linear algebra and its applications {\bf 159}, 55-80 (1991).
		 
		 
		 		\bibitem{complex_convergence_rate} El Karoui, N.: \emph{A rate of convergence result for the largest eigenvalue of complex white Wishart matrices}, Ann.\ Probab.\ {\bf 34}(6), 2077-2117 (2006).
				
				
				\bibitem{nonnull1} El Karoui, N.: \emph{Tracy--Widom limit for the largest eigenvalue of a large class of complex sample covariance matrices}, Ann.\ Probab.\ {\bf 35}(2), 663 - 714 (2007).


							

	
		 
		 \bibitem{Erdos+Knowles+Yau}  Erd\H{o}s, L., Knowles, A. and  Yau, H.-T.: \emph{Averaging fluctuations in resolvents of random band matrices}, Ann. Henri Poincar\'e {\bf 14}, 1837-1926 (2013).		 
					    		   
			 \bibitem{sparse0}  Erd\H{o}s, L., Knowles, A. and  Yau, H.-T. and Yin, J.: \emph{Spectral statistics of Erd\H{o}s-R\'enyi Graphs II: Eigenvalue spacing and the extreme eigenvalues}, Communications in Mathematical Physics {\bf 314}(3), 587-640 (2012).			    		   
					    		    		
					    		  
			\bibitem{isotropic2}  Erd\H{o}s, L., Kr\"uger, T. and Schr\"oder, D.: \emph{Random Matrices with Slow Correlation Decay}, Forum of Mathematics Sigma {\bf 7}(8) (2019).
								        		         		
		 \bibitem{book}  Erd\H{o}s, L. and Yau, H.-T.: \emph{ A dynamical approach to random matrix theory. Courant Lecture Notes in Mathematics} {\bf 28}. Providence: American Mathematical Society (2017).
		
 		\bibitem{rigidity} Erd\H{o}s, L., Yau, H.-T. and Yin, J.: \emph{Rigidity of eigenvalues of generalized Wigner matrices}, Adv. Math. {\bf 229}(3), 1435-1515 (2012).

		
		 \bibitem{manova}  Fan, Z., and Johnstone, I. M.:\emph{Tracy--Widom at each edge of real covariance and MANOVA estimators}, arXiv:1707.02352 (2017).
		 
		\bibitem{sodin} Feldheim, O. N., Sodin, S.: \emph{A universality result for the smallest eigenvalues of certain sample covariance matrices}, Geometric And Functional Analysis {\bf 20}(1), 88-123 (2010).
				
		\bibitem{peter} Forrester P. J.: \emph{The spectrum edge of random matrix ensembles}, Nuclear Physics B {\bf 402}, 709-728 (1993).
		
		\bibitem{German} Geman, S.: \emph{A limit theorem for the norm of random matrices}, Ann.\ Probab.\, 252-261 (1980).
		

                \bibitem{nonnull3} Hachem, W., Hardy, A., Najim, J.: \emph{Large complex correlated Wishart matrices: Fluctuations and asymptotic independence at the edges}, Ann.\ Probab.\, {\bf 44}(3), 2264-2348 (2016).
                
                
             \bibitem{fisher}   Han, X., Pan, G., Zhang, B.: \emph{The Tracy--Widom law for the largest eigenvalue of F type matrices}, Ann.\ Stat.\ {\bf 44}(4), 1564-1592 (2016).

\bibitem{manova2}  Han, X., Pan, G., and Yang, Q.: \emph{A unified matrix model including both CCA and F matrices in multivariate analysis: The largest eigenvalue and its applications}, Bernoulli {\bf 24}(4B), 3447-3468 (2018).




          	\bibitem{moment} He, Y.,  Knowles, A.: \emph{Mesoscopic eigenvalue statistics of Wigner matrices}, Ann. Appl. Probab. {\bf 27}(3), 1510-1550 (2017).
		
		 \bibitem{He+Knowles} He, Y., Knowles, A.: \emph{Fluctuations of extreme eigenvalues of sparse Erd\H{o}s-R\'{e}nyi graphs}, arXiv:2005.02254, (2020).
		 
		\bibitem{kurt_complex} Johansson, K.: \emph{Shape fluctuations and random matrices}, Commun. Math. Phys. {\bf 209}, 437-476 (2000)..

		\bibitem{kurt_det} Johansson, K.: \emph{Random matrices and determinantal processes}, arXiv math-ph/0510038 (2005).
		
	
	 	\bibitem{johnstone} Johnstone, I. M.: \emph{On the distribution of the largest eigenvalue in principal components analysis}, Ann. Stat. {\bf 29}(2), 295-327 (2001).
	 
	 		\bibitem{johnstone_application} Johnstone, I. M.:  \emph{High dimensional statistical inference and random matrices}, arXiv math/0611589 (2006).	
	 	
	 	\bibitem{johnstone2} Johnstone, I. M.: \emph{Multivariate analysis and Jacobi ensembles: largest eigenvalue, Tracy–Widom limits and rates of convergence}, Ann. Stat. {\bf 36}, 2638-2716 (2008). 
					

					\bibitem{johnstone+paul} Johnstone, I. M., Paul, D.: \emph{PCA in High Dimensions: An Orientation}, Proc IEEE Inst Electr Electron Eng.{\bf 106(8)}, 1277-1292 (2018).

		
      \bibitem{pca}	Jolliffe, I.: \emph{Principal Component Analysis}(2nd ed.), 
      Springer (2002)
		

			
		\bibitem{KKP} Khorunzhy, A., Khoruzhenko, B., Pastur, L.: \emph{Asymptotic Properties of Large Random Matrices with Independent Entries},  Journal of Mathematical Physics \textbf{37(10)}, 5033-5060 (1996).
		
		
		 \bibitem{isotropic}  Knowles, A., Yin, J.: \emph{Anisotropic local laws for random matrices}, Probab. Theory Relat. Fields {\bf 169}(1-2), 257-352 (2017).
		

 	         \bibitem{LandonYau_edge} Landon, B, Yau, H-T.: \emph{Edge statistics of Dyson Brownian motion},  arXiv 1712.03881, 2017.
			  
			  \bibitem{fixed dbm} Landon, B., Sosoe, P., Yau, H.-T.: \emph{Fixed energy universality of Dyson Brownian motion}, Adv.\ Math.\ {\bf 346}, 1137-1332 (2019).
			  		
	         \bibitem{LS14b} Lee, J.\ O., Schnelli, K.: \emph{Tracy--Widom Distribution for the Largest Eigenvalue of Real Sample Covariance Matrices with General Population}, Ann.\ Appl.\ Probab.\ \textbf{26(6)}, 3786-3839 (2016).
	         
		 \bibitem{sparse} Lee, J.\ O., Schnelli, K.: \emph{Local law and Tracy--Widom limit for sparse random matrices}, Probab. Theory Relat. Fields {\bf 171}(1), 543-616 (2018).	
		 
		 
		 	\bibitem{LY} Lee, J.\ O., Yin, J.: \emph{A Necessary and Sufficient Condition for Edge Universality of Wigner Matrices},  Duke Math.\ J.\ \textbf{163(1)}, 117-173, (2014). 	 
		 

		\bibitem{LP} Lytova, A., Pastur, L.: \emph{Central Limit Theorem for Linear Eigenvalue Statistics of Random Matrices with Independent Entries}, Ann.\ Probab. \textbf{37}, 1778-1840 (2009).
				 
				 
		 \bibitem{Ma} Ma, Z.: \emph{Accuracy of the Tracy--Widom limits for the extreme eigenvalues in white Wishart matrices}, Bernoulli \textbf{18}(1) 322-359 (2012).
		 
		  \bibitem{MP} Marchenko, V. A., Pastur, L. A.: \emph{Distribution of eigenvalues for some sets of random matrices}, Matematicheskii Sbornik {\bf 114}(4), 507-536 (1967).
		 
		  \bibitem{metha} Mehta, M.: \emph{Random Matrices}, Pure and Applied Mathematics {\bf 142}, third version, Academic Press (2004).
		  
		          \bibitem{classical2} Muirhead, R.J.: \emph{Aspects of Multivariate Statistical Theory}, Wiley, New York (1982)
		  
		  
		  			\bibitem{nonnull2} Onatski, A.: \emph{The Tracy--Widom limit for the largest eigenvalues of singular complex Wishart matrices}, Ann. Appl. Probab. {\bf 18}, 470-490 (2008).

		  \bibitem{signal2} Onatski, A.: \emph{Testing hypotheses about the number of factors in large factor models}, Econometrica {\bf 77}, 1447-1479 (2009).

		  
		  \bibitem{paul+aue} Paul, D., Aue, A.: \emph{Random matrix theory in statistics: A review. Journal of Statistical Planning and Inference} {\bf 150}, 1-29 (2014).
		  
		 \bibitem{peche}  P\'ech\'e, S.: \emph{Universality results for the largest eigenvalues of some sample covariance matrix ensembles}, Probab. Theory Relat. Fields {\bf 143}(3), 481-516 (2009).
		  

		\bibitem{PY1}Pillai, N., Yin, J.: \emph{Universality of covariance matrices}, Ann.\ Appl.\ Probab.\ \textbf{24}(3), 935-1001 (2014).
		
			\bibitem{roy} Roy, S. N.: \emph{On a Heuristic Method of Test Construction and its use in Multivariate Analysis}, Ann. Math. Statist. {\bf 24}(2), 220-238 (June).
		
		\bibitem{Schnelli+Xu} Schnelli, K., Xu, Y.: \emph{Convergence rate to the Tracy--Widom laws for the largest eigenvalue of Wigner matrices}, arXiv:2102.04330 (2021).
		
		\bibitem{silverstein} Silverstein, J. W.: \emph{On the weak limit of the largest eigenvalue of a large dimensional sample covariance matrix}, Journal of Multivariate Analysis {\bf 30}(2), 307-311 (1989).

		   \bibitem{So1} Soshnikov, A.: \emph{Universality at the Edge of the Spectrum in Wigner Random Matrices}, Commun.\ Math.\ Phys.\ \textbf{207}, 697-733 (1999).
		     
		   \bibitem{sasha_det} Soshnikov, A.: \emph{Determinantal random point fields}, Russian Math. Surveys {\bf 55}(5), 923-975 (2000).
		   
		   \bibitem{sasha} Soshnikov, A.: \emph{A note on universality of the distribution of the largest eigenvalues in certain sample covariance matrices}, Journal of Statistical Physics {\bf 108}(5), 1033-1056 (2002).
		   
		     \bibitem{sasha_pf} Soshnikov, A.: \emph{ Janossy Densities II. Pfaffian Ensembles}, Journal of Statistical Physics {\bf 113}(3/4), 611-622 (2003).
		     
		      \bibitem{laguerre_poly} Szeg\H{o}, G.: \emph{Orthogonal polynomials} (Vol. 23). American Mathematical Soc. (1939).
		   
		   
		                         \bibitem{tao_vu2}   Tao, T., Vu, V.: \emph{Random matrices: Universality of local eigenvalue statistics up to the edge}, Communications in Mathematical Physics, {\bf 298}(2), 549-572 (2010).

                     \bibitem{tao_vu}   Tao, T., Vu, V.: \emph{Random matrices: The distribution of the smallest singular values}, Geometric And Functional Analysis {\bf 20}(1), 260-297 (2010).
                     
                     

		  
		   \bibitem{TW1} Tracy, C.,  Widom, H.: \emph{Level-Spacing Distributions and the Airy Kernel}, Commun. Math. Phys. {\bf 159}, 151-174 (1994).
		   
		    \bibitem{TW2} Tracy, C.,  Widom, H.: \emph{On Orthogonal and Symplectic Matrix Ensembles}, Commun. Math. Phys. {\bf 177}, 727-754 (1996).
		
				 \bibitem{wang_rate} Wang, H.: \emph{Quantitative Universality for the Largest Eigenvalue of Sample Covariance Matrices}, arXiv:1912.05473 (2019).
    
                            \bibitem{wang_ke} Wang, K.: \emph{Random covariance matrices: Universality of local statistics of eigenvalues up to the edge}, Random Matrices: Theory and Applications {\bf 1}(01), 1150005 (2012).

		    
               \bibitem{widom} Widom H.: \emph{On the relation between orthogonal, symplectic and unitary matrix ensembles}, Journal of Statistical Physics {\bf 94}(3), 347-363 (1999).
               
               \bibitem{separable} Yang, F.: \emph{Edge universality of separable covariance matrices}, Electronic Journal of Probability {\bf 24}, 1-57 (2019).
               
             \bibitem{yang2}  Yang, F.: \emph{Sample canonical correlation coefficients of high-dimensional random vectors: local law and Tracy--Widom limit}, arXiv:2002.09643 (2020).
               
               \bibitem{bai+yao}  Yao, J., Zheng, S., and Bai, Z.: \emph{ Large Sample Covariance Matrices and High-Dimensional Data Analysis} (Cambridge Series in Statistical and Probabilistic Mathematics), Cambridge: Cambridge University Press (2015).
		    
		 \bibitem{bai}    Yin, Y. Q., Bai, Z. D.,  Krishnaiah, P. R.: \emph{On the limit of the largest eigenvalue of the large dimensional sample covariance matrix}, Probab. Theory Relat. Fields {\bf 78}(4), 509-521 (1988).
		




		    
		    
				 
\end{thebibliography}
\end{document}